\documentclass[10pt]{book}
\usepackage[utf8]{inputenc}

\usepackage{amsmath, amsthm, amssymb}

\usepackage{newunicodechar}
\newunicodechar{∞}{\infty}

\usepackage[top=2.5cm,bottom=2cm,left=2.2cm,right=2.2cm]{geometry}
\usepackage{geometry}
\usepackage{url}
\usepackage{color}
\usepackage{hyperref}
\usepackage{ dsfont }
\usepackage{upquote}

\usepackage{enumerate}
\usepackage[shortlabels]{enumitem}

\usepackage[
backend=biber,
style=alphabetic,
sorting=ynt,
maxbibnames=99,
maxcitenames=99,
maxalphanames=99
]{biblatex}
\usepackage{graphicx}
\usepackage{subcaption}
\graphicspath{ {./Pictures/} }
\hypersetup{colorlinks=true,
     breaklinks=true,
     linkcolor=blue,}    % how to specify linkcolor according to chapter / contents ?
\usepackage{makeidx}

\usepackage{mathabx}

\makeindex

\usepackage{mathtools}

\usepackage{tikz-cd}
\tikzcdset{scale cd/.style={every label/.append style={scale=#1},
    cells={nodes={scale=#1}}}}

\newcommand{\Match}{\operatorname{Match}}
\newcommand{\cross}{\operatorname{cross}}
\newcommand{\Mod}{\operatorname{mod}}
\newcommand{\rep}{\operatorname{rep}}

\usepackage{enumitem}
\usepackage[mathscr]{euscript}

\usepackage{graphicx}
\usepackage[all]{xy}

\newtheorem{theorem}{Theorem}[section]
\newtheorem{example}[theorem]{Example}
\newtheorem{proposition}[theorem]{Proposition}
\newtheorem{lemma}[theorem]{Lemma}
\newtheorem{corollary}[theorem]{Corollary}
\theoremstyle{definition}
\newtheorem{definition}[theorem]{Definition}

\theoremstyle{remark}
\newtheorem{remark}[theorem]{Remark}
\title{Skein relations on punctured surfaces}
\author{Michael Tsironis}
\date{December 2025}

\begin{document}

\maketitle

\frontmatter

\tableofcontents

\mainmatter

\chapter{Introduction}

\pagenumbering{arabic}

The largest part of this thesis concerns introducing \emph{skein relations} for cluster algebras from punctured surfaces. These are identities in terms of cluster variables in a cluster algebra which showcase how certain variables can be expressed in terms of variables which contain some nice properties. Namely, cluster variables corresponding to arcs, which are considered to be \emph{incompatible}, on a surface, can be written as cluster variables corresponding to compatible arcs. Examples of incompatible arcs include intersecting arcs, self-intersecting arcs and closed curves, as well as arcs with opposite tagging at a puncture. In order to prove these important identities, we first construct a bridge between some combinatorial objects (graphs) and some algebraic representations. Our final aim of this thesis, in which skein relations play a crucial role, is to prove the existence of some bases on \emph{surface cluster algebras} that satisfy some nice properties. This expands on the existing bibliography by dealing with a larger collection of surface cluster algebras, by adding punctures (marked points) to the interior of the surface.\\

Cluster algebras were introduced by Fomin and Zelevinsky \cite{Fomin-Zelevinsky1} \cite{Fomin-Zelevinsky2} \cite{Fomin-Zelevinsky3} in 2001 with the initial aim of developing a combinatorial framework for the understanding of the Lustig's dual canonical bases \cite{Lusztig} and total positivity in algebraic groups. However, it gained a lot of interest on its own immediately after \cite{Caldero2006-mh}\cite{BMRRT}\cite{BrustleZhang}, as well as attention from various other fields such as Teichm{\"u}ller theory and higher rank geometry \cite{FockGon}, algebraic geometry, mirror symmetry \cite{GHKK}, mathematical physics, as it can be seen through their appearance in the Kontsevich-Soibelman \emph{wall crossing formula} for Donaldson-Thomas invariants \cite{KontsevichSoibelman} and many more \cite{Yuji},\cite{Davison2018-of}).\\

Cluster algebras are commutative rings which are generated through a distinct set of generators called \emph{cluster variables}, which are grouped into a set called \emph{clusters}. Relating back to the original motivation, explicitly stating the elements of these dual canonical bases is a very hard problem, but their initial conjecture was that all the monomials appearing in these cluster variables, belong to the dual canonical basis. This fundamental connection of cluster algebras and dual canonical bases, poses the natural question of finding and constructing bases which have some desired properties. One of these properties would be to require the basis to have some positivity properties, which is closely related to the initial \emph{Positivity Conjecture} of Fomin and Zelevinsky, which states that every cluster variable can be expressed as a Laurent polynomial in the initial variables with positive coefficients. This conjecture has been proven for specific cases by several authors (\cite{Fomin-Zelevinsky2},\cite{CalderoKeller}, \cite{Nakajima}), while in Lee and Schiffler \cite{LeeSchifler}  gave a purely combinatorial proof for all skew-symmetrizable cluster algebras. Moreover later Gross, Hacking, Keel and Kontsevich \cite{GHKK} established positivity, and much more, for cluster algebras of \emph{geometric type} via the construction of \emph{canonical theta basis construction}. This was a very deep result, since they constructed the theta basis for certain cluster varieties using scattering diagrams and the positivity was a simple corollary of their work, which is a prime example of the hidden role that cluster algebras play in various settings.\\

An important class of \emph{surface cluster algebras} was introduced by Fomin, Shapiro and Thurston (\cite{FST}, \cite{FominThurston}), where they gave a geometric construction associating a cluster algebra to a marked surface. These algebras are very interesting for a variant of reasons. First of all they have a topological interpretation as they provide coordinate charts for decorated Teichm{\"u}ller spaces, giving a combinatorial atlas of the moduli space \cite{Penner87}. Additionally, they are closely connected to representation theory and categories, since they represent combinatorially abstract notions. 

One of these first connections is by Buan, Marsh, Reineke, Reiten and Todorov from \cite{BMRRT} where they introduced the notion of the \emph{cluster category}, which served as a nice model for the combinatorics of a class of cluster algebras. However it was Amiot \cite{Amiot2009} in her fundamental work who contributed vastly in the categorization of cluster algebras by constructing \emph{generalized cluster categories}, extending the original cluster categories to a much broader setting. Relevant to this thesis, Br{\"u}stle and Zhang \cite{BrustleZhang} and Labardini-Fragoso \cite{labardinifragoso2009quiverspotentialsassociatedtriangulated},\cite{Labardini-Fragoso2009-xb}, finalized the formalization of the representation-theoretic model for surface cluster algebras. To give some more details regarding this connection, certain objects called \emph{cluster tilted objects} of a cluster category correspond to \emph{tagged triangulations} of the surface \cite{BrustleZhang},\cite{Schiffler2008-nd}.  This result provides another bridge between surface cluster algebras and representation theory, providing tools for structural results in both direction. This in turn, influenced a lot of work in representation theory such as  \cite{Assem2010-fb},\cite{Adachi_Iyama_Reiten_2014},\cite{canakcci2021lattice}. \\

The initial inspiration of this thesis stems from the work of Musiker, Schiffler and Williams \cite{Musiker_Schiffler_Williams_2013}, where they constructed two bases for cluster algebras coming from a triangulated surface without \emph{punctures}. This result relied heavily on one of their previous papers \cite{MSWpositivity}, in which they gave combinatorial formulas for the cluster variables in the cluster algebra, using the so-called \emph{snake graphs}. To each \emph{arc} (simple curve) on the surface, they associated a planar snake graph and subsequently the formula was given as a weighted sum over perfect matchings of the snake graph. Additionally they expanded this result to \emph{band graphs} that correspond to closed curves on the surface. One of the key points in their proof of the basis result, was the fact that \emph{skein relations} of intersecting curves on this setup had a nice property.\\
\medskip

Skein relations are algebraic identities which express how curves can be transformed through local intersections and smoothing operations. In their specific setup, a product of crossing arcs could be expressed in terms of non intersecting arcs and loops which have a unique term on the right hand side of the equation with no coefficient variables. In the same work they also conjectured that a similar result should be true in the case of \emph{punctured surfaces}, i.e. surfaces where marked points in the interior of the surface are also allowed. One should basically prove the equivalent skein relations in this setup. However the existing machinery was not enough to prove such relations at that point.\\

This brings us to our second major inspiration for this thesis. In the setting of punctured surfaces, plain arcs are not enough in order to capture all the cluster variables of the associated cluster algebra. In order to do so, a second type of arcs must be introduced, the so-called \emph{tagged arcs} \cite{FST}. These are arcs that are allowed to have a special \emph{tagging} on their endpoints which are adjacent to \emph{punctures} (marked points in the interior). Expansion formulas were already know for such arcs \cite{MSWpositivity}, alas larger graphs and a different more complicated version of perfect matching had been considered, making it complicated to study skein relations in this setup. However, Wilson \cite{wilson2020surface} introduced the notion of \emph{loop graphs}; these are simple graphs that can be associated to a tagged arcs. Using these loop graphs, Wilson gave an alternative expansion formula for the tagged arcs. \\

One can then ask whether these graphs can be used to show skein relations, and a natural approach is to attempt this question combinatorially in the spirit of the work by Çanakçı and Schiffler. In a series of papers (\cite{canakci3}),\cite{CANAKCI2013240}, \cite{canakci2}), they introduced abstract snake graphs, and gave alternative proofs for the skein relations occurring in the setting of surfaces, by constructing bijections between the sets of perfect matchings of the graphs appearing in both sides of the relations. This technique showcases the importance of these combinatorial objects, and hints on how one could work using loop graphs in order to prove skein relations.

In this thesis we initially explore skein relations on punctured surfaces.\\ As stated earlier, in \cite{Musiker_Schiffler_Williams_2013} the authors conjectured that one can extend the given bases of unpunctured surfaces, to the punctured setup. They explicitly stated 15 different cases of skein relations that had to be resolved in order for one to tackle the more general problem of bases on general surfaces. In order for us to prove these relations, we use both directly and indirectly the loop graphs introduced by Wilson. Although helpful in some sense, these loop graphs contain too much information, which complicates things when one explicitly tries to prove things using them. \\
However one of the more understood and easy to work with objects are quiver representations. In the classical setting of snake graphs Çanakçı and Schroll \cite{canakcci2021lattice} introduced abstract string modules and constructed an explicit bijection between the submodule lattice of an abstract string module and the perfect matching lattice of the corresponding abstract snake graph. This correspondence comprises the last key component of our approach. We thus expand this framework to a connection between loop graphs and what we call loop modules and loop strings. In Remark 7.11 of \cite{wilson2020surface} the author already indicates how one can do this association, however we make this connection explicit. More explicitly, we first make precise the existence of an isomorphism between the perfect matching lattice of a graph $\mathcal{G}$ and the submodule lattice of the associated module $M_\mathcal{G}$ as stated in the following Theorem~\ref{submodule lattice bijective to perfect matchings lattice}.\\

\begin{theorem}
Let $A=kQ^\bowtie/I$ and $M(w)$ be a loop module over $A$ with associated loop graph $\mathcal{G}^\bowtie$. Then $\mathcal{L}(\mathcal{G})$, which denotes the perfect matching lattice of $\mathcal{G}$ is in bijection with the canonical submodule lattice $\mathcal{L}(M(w))$.
\end{theorem}

The above theorem is the first stepping tool. However we additionally introduce the notion of \emph{loopstrings}, which is a generalization of the notion of strings in classical representation theory. The idea is to substitute loop graphs with what we call \emph{loop modules} and then in turn define these modules using word combinatorics, which entail the most important properties of the module. \\

%By introducing a generalization of the notion of a string module, which corresponds precisely to this module $M_\mathcal{G}$ and we name it a loopstring $w_\mathcal{G}$.\\
Çanakçı and Schiffler in \cite{CANAKCI2013240} (and the continuation of this paper introduced snake graph calculus in order to give an alternative proof of skein relations for unpunctured surfaces. They proved this by constructing an explicit bijection between the perfect matchings associated to the initial arcs and the perfect matchings of the arcs generated by the ``resolution" of the initial crossing. This provided our first idea of trying to generalize this construction by using the newly introduced loop graphs. However such a construction should not be considered trivial for a multitude of reasons, the first one being that in the case of punctured surfaces the skein relations are not known. The second and most important reason on why such a construction in the new setting would not be ideal, is the fact that the cases that need to be investigated for unpunctured surfaces are more complicated and working out the combinatorics through snake graph, or in our case loop graph calculus would require extreme time and effort, just for setting up each case.\\
However, a subsequent train of thought would be to take advantage of the bijection of the loop graphs and the newly introduced loopstrings and try to use this new tool as a means of simplifying some procedures. It should be noted here that the exact reason that makes loopstrings easier to work with, is the same reason on why these by themselves do not imply that the skein relations straightforwardly. The problem is, that string modules store much less information than snake graphs, which is extremely important when proving that the elements of the cluster algebra in both parts of the resolution coincide. Thankfully one can get away with this lose of information by realizing that there is a clever way of associating the correct monomial in the cluster algebra to some nice modules. We make this construction clear in Definition \ref{monomial associated to loopstring} and heavily rely on that and Lemma \ref{xvariablesagree} to show that we can recover the loss of information that happened when we changed the set up from the loop graphs to the loop modules. \\

By relying on the key ingredients listed above, we are able to prove skein relations for every punctured surface apart from some extreme cases, which are basically the extreme cases for which loop graphs are not defined. The following theorem sums up Theorem~\ref{skein relations regular crossing}, Theorem~\ref{skein for loop and arc}, Theorem~\ref{skein relation for single incompatibility at a puncture}, Theorem~\ref{Skein relations for double incompatibilities at two punctures} and  and Theorem~\ref{skein relation for generalized tagged self crossing arcs.}.\\

\begin{theorem}\label{skein relations}
    Let $(S,M,P,T)$ be a triangulated punctured surface and $\mathcal{A}$ the cluster algebra associated to it. Let $\gamma_1$ and $\gamma_2$ be two arcs, which are incompatible. Then there are multicurves $c_1$ and $c_2$ such that: 
        \[
        x_{\gamma_1} x_{\gamma_2} = Y^{-}x_{c_1}  + Y^{+} x_{c_2}
        \]
        where $Y^{-}, Y^{+}$ are monomials in $y_i$ coefficients and satisfy the condition that one of the two is equal to 1.
\end{theorem}

During the writing of this thesis we became aware that Banaian, Kang and Kelley were working independently on the same problem, using a different approach \cite{banaian2024skeinrelationspuncturedsurfaces}. We would like to thank them for their transparency and for the helpful communication.\\

The structure of this thesis, goes as follows:\\

In Section 2 we recall some well-known results, mainly on cluster algebras and more specifically on cluster algebras associated to triangulated surfaces, which will be extensively used in the rest of the thesis. We also introduce the notion of loopstrings and loop modules. \\

In Section 3, we construct an explicit bijection between the perfect matching lattice $\mathcal{L}(\mathcal{G})$ of  a given loop graph $\mathcal{G}$ and the associated canonical submodule lattice $\mathcal{L}(w_\mathcal{G})$. This result, although suggested by Wilson, is important to be dealt with with care, since the rest of the thesis uses directly and indirectly this construction in virtually every proof and thus making it worthy enough to be given some more attention.\\

In Section 3, which comprises the main part of the new results of this thesis, we prove skein relations for every possible configuration between two arcs on a given punctured surface building on the previously proven bijection. We do a case by case study, while we also group some of the cases when possible. In total these cases add up to 15 cases that were indicated in \cite{Musiker_Schiffler_Williams_2013}. \\

\chapter{Preliminaries}

\section{Cluster Algebras}
We start by recollecting the definition of a \emph{cluster algebra} which was first introduced by Fomin and Zelevinsky \cite{Fomin-Zelevinsky1}. We will not give the general definition of a cluster algebra, but rather restrict ourselves to the so-called \textit{skew-symmetric cluster algebras with principal coefficients}, which are closely related to bordered surfaces, the main interest of this thesis. \\
To define a cluster algebra $\mathcal{A}$ we need to start by fixing its ground ring. We will be dealing with cluster algebras of \textit{geometric type}, that is, the \emph{coefficients} of the cluster algebra are elements of a so-called \emph{semifield}.\\
Let $(\mathds{P},\cdot)$ be a free abelian group on the variables $\{y_1,\dots,y_n\}$. We define an addition on $\mathds{P}$ as follows:
\[
\prod_{j}{y_j^{a_j}} \oplus \prod_{j}{y_j^{b_j}} = \prod_{j}{y_j^{\min\{a_j,b_j\}}}.
\]
This makes $(\mathds{P},\cdot,\oplus)$ a semifield, i.e., an abelian multiplicative group endowed with a binary operation $\oplus$ which is commutative, associative and distributive with respect to the multiplication $\cdot$ in $\mathds{P}$. The group ring $\mathds{Z}\mathds{P}$, i.e., the ring of Laurent polynomials in $y_1,\dots, y_n$ will be the \textit{ground ring} of the cluster algebra $\mathcal{A}$.\\
Let also $\mathds{Q}\mathds{P}$ be the field of fractions of $\mathds{Z}\mathds{P}$, and $\mathcal{F}= \mathds{Q}\mathds{P}(x_1,\dots,x_n)$ the field of rational functions in the variables $x_1,\dots,x_n$ and coefficients in $\mathds{Q}\mathds{P}$. The field $\mathcal{F}$ will be the \textit{ambient field} of the cluster algebra $\mathcal{A}$.\\

So far, we have talked about where the cluster algebra ``lives", but not how it is generated. To define a cluster algebra, we need to determine an \textit{initial seed} ($\mathbf{x}$, $\mathbf{y}$, $Q$), which consists of the following:
\begin{itemize}
    \item $Q$ is a quiver (i.e., a directed graph) (see chapter~\ref{Quiver representations and loop modules}) without loops and 2-cycles and with $n$ vertices,
    \item $\mathbf{x}= (x_1,\dots,x_n)$ is the $n$-tuple from $\mathcal{F}$, which is called the \textit{initial cluster},
    \item $\mathbf{y}= (y_1,\dots,y_n)$ is the $n$-tuple of generators of $\mathds{P}$, which is called the \textit{initial coefficients}.
\end{itemize}

The next vital procedure in the construction of a cluster algebra is the so-called \textit{mutation}. The idea is to take a seed  $(\mathbf{x}, \mathbf{y}, Q)$ and transform it into a new one $(\mathbf{x}', \mathbf{y}', Q')$. Since there are $n$ variables in the initial cluster (or equivalently $n$ vertices in the quiver $Q$) we want to transform it in $n$ different directions (ways).

\begin{definition}\label{mutaion}
    Let $(\mathbf{x}, \mathbf{y}, Q)$ be a seed of a cluster algebra $\mathcal{A}$ and $1\leq k\leq n$. The \emph{seed mutation} $\mu_k$, in direction $k$, transforms the seed $(\mathbf{x}, \mathbf{y}, Q)$ into the seed $\mu_k(\mathbf{x}, \mathbf{y}, Q)= (\mathbf{x}', \mathbf{y}', Q')$ where:
    \begin{itemize}
        \item The quiver $Q'$ is obtained from $Q$ following the next steps:
        \begin{enumerate}
            \item for every path $i\longrightarrow k \longrightarrow j$, add one arrow $i\longrightarrow j$ to $Q'$, 
            \item reverse all the arrows adjacent to $k$,
            \item delete all the 2-cycles that were created in the previous steps.
        \end{enumerate}
        \item The new \emph{cluster} is $\mathbf{x}' = \mathbf{x}\setminus \{x_k\}\bigcup \{x_k'\}$ where the new \emph{cluster variable} $x_k'\in \mathcal{F}$ is given by the following \emph{exchange relation}:
        \[
        x_k'= \frac{y_k\prod_{i\rightarrow k}{x_i} + {\prod_{i\leftarrow k}{x_i}}}{x_k(y_k\oplus 1)},
        \]
        where the first product runs over all the vertices that are sources to arrows leading to the vertex $k$, while the second one runs over all the vertices that are targets from arrows coming from the vertex $k$.
        \item The new coefficient tuple is $\mathbf{y}'= \{y_1',\dots, y_n'\}$ where:
        \begin{equation*}
             y_j' = \begin{cases}
                 y_k^{-1} & \text{if } j=k,\\
                 y_j \prod_{k\rightarrow j}{y_k(y_k\oplus 1)^{-1}} \prod_{k\leftarrow j}{(y_k\oplus 1)} & \text{if } j\neq k.\\
             \end{cases}
         \end{equation*}
    \end{itemize}
\end{definition}

%%Mutating seeds gives rise to new $n$-tuples $\mathbf{x}$ and these are called \emph{clusters} of a \emph{labeled seed} $(\mathbf{x}, \mathbf{y}, Q)$. The variables appearing in these clusters are in turn called \emph{cluster variables}. \\

These cluster variables are the generators of the cluster algebra. Even though mutations are involutions, meaning that applying a mutation twice in a row in the same direction takes you back to the initial seed, by mutating in different direction each time, we recursively create new cluster variables. Doing this procedure possibly infinite times gives rise to possibly infinitely many cluster variables, which are set to be the generators of the cluster algebra.

\begin{definition}
    Let $(\mathbf{x}, \mathbf{y}, Q)$ be an initial seed and $\mathcal{X}$ the set of all cluster variables obtained by all possible mutations starting from the seed $(\mathbf{x}, \mathbf{y}, Q)$. The cluster algebra $\mathcal{A}=\mathcal{A}(\mathbf{x}, \mathbf{y}, Q)$ is the algebra $\mathds{Z}\mathds{P}[\mathcal{X}]$, i.e., the $\mathds{Z}\mathds{P}$-subalgebra of $\mathcal{F}$, generated by $\mathcal{X}$.
\end{definition}

\begin{example}%{\color{purple}[NEW]}
    Let $(\mathbf{x}, \mathbf{y}, Q)= ((x_1,x_2),(1,1), 1\rightarrow 2)$ be the initial seed of the cluster algebra $\mathcal{A}$. Since the initial coefficients are trivial, they coefficients will remain trivial in any seed, so we do not need to keep track fo them. In Figure~\ref{fig:Example Cluster Algebra A_2} we can see the so-called \emph{exchange graph}, where the edges correspond to the mutations, and the vertices correspond to the cluster variables. Mutating the seed $((\frac{x_1+1}{x_2},x_1),1\rightarrow2)$ in direction $\mu_1$ produces the seed $((x_2,x_1),1\leftarrow 2)$. If we keep on mutating, we will not produce any new cluster variables. Therefore, the cluster algebra $\mathcal{A}$ is said to be of  finite type and is generated by the set:
    \[
    \mathcal{X}=\{x_1,x_2,\frac{x_2+1}{x_1}, \frac{x_1+x_2+1}{x_1x_2}, \frac{x_1+1}{x_2}\}.
    \]
\end{example}

\begin{figure}
    \centering
    \includegraphics[scale=0.6]{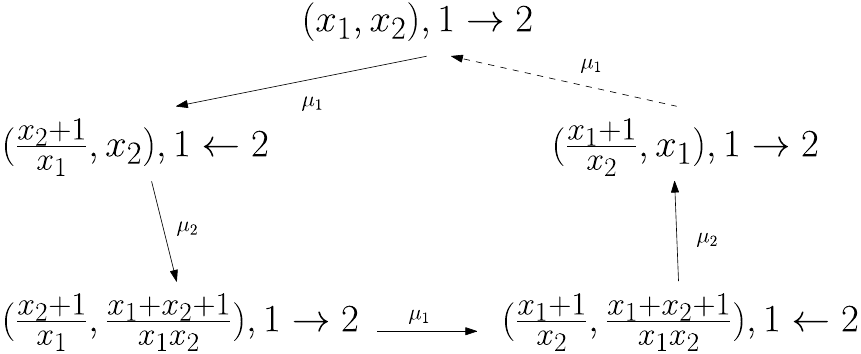}
    \caption{The exchange graph of the cluster algebra of type $A_2$.}
    \label{fig:Example Cluster Algebra A_2}
\end{figure}

\section{Punctured surfaces}

In this section we recall the notion of \textit{tagged arcs} on a \emph{punctured surface}. Our goal is to describe the lattice structure of a loop graph associated to a tagged arc, which in turn will help us introduce the skein relations when tagged arcs are involved. Therefore, we will restrict our attention to such \textit{loop graphs}. Fomin, Shapiro, and Thurston \cite{FST} established a cluster structure for triangulated oriented surfaces. Our motivation stems from the interplay between cluster algebras, bordered triangulated surfaces, and cluster categories. \\

Let $S$ be a compact oriented Riemann surface with boundary $\partial S$. Fix $M\subset \partial S$ to be a finite set of \textit{marked points} on $\partial S$, with at least one marked point on each connected component of the boundary. Furthermore, fix $P\subset S$ to be a finite set of marked points in the interior of the surface, which we will call \textit{punctures}. We refer to the triplet $(S,M,P)$ as a \textit{punctured surface} if $P\neq \emptyset$ and \textit{unpunctured  surface} otherwise. For technical reasons, we exclude the cases where $(S,M,P)$ is an unpunctured or once-punctured monogon, a digon, a triangle, or a once-, twice- or thrice-punctured sphere.

\begin{definition}\label{arc}
    An \textit{arc} of $(S,M,P)$ is a simple curve, up to isotopy class, in S connecting two marked points of $M$, which is not isotopic to a boundary segment or a marked point. A \textit{tagged arc} $\gamma$ is an arc whose endpoints have been ``tagged'' in one of two possible ways: \textit{plain} or \textit{notched}. This tagging must also satisfy the following conditions:
    \begin{itemize}
        \item If the beginning and end points of $\gamma$ are the same point, then the tagging at this point must be the same for both the beginning and the end.
        \item An endpoint of $\gamma$ lying on the boundary $\partial S$ of the surface must have a plain tagging.
        \end{itemize}    
\end{definition}

When an arc $\gamma$ is notched at least at one of its endpoints, we will denote it by $\gamma^\bowtie$, as a reminder of this tagging, and refer to it as a \emph{tagged arc}. An arc with plain tagging will be denoted simply by $\gamma$ and will be called an \emph{untagged arc}. 

\begin{remark}
    In Definition~\ref{arc}, when referring to a simple curve up to isotopy classes, we mean that two curves on the surface $S$ are considered the same if you can ``stretch'' them, without passing ``over'' a puncture.\\
    Additionally, given a tagged arc $\gamma^\bowtie$, we will call the same arc with plain tagging in both of its endpoints, the \textit{untagged version} of it and denote it by $\gamma$. The arcs $\gamma^\bowtie$ and $\gamma$ then belong to different isotopy classes.
\end{remark}

In order to associate a cluster structure to a given surface, one needs the notion of a \textit{tagged triangulation}, which is closely related to the initial seed of the associated cluster algebra. In order to define these triangulations, we first have to define when two arcs on the surface are considered to be \textit{compatible}. 

\begin{definition}\label{compatible arcs}
    Let $(S, M, P)$ be a punctured surface and $\gamma_1$ and $\gamma_2$ be two arcs on the surface. We will say that the arcs $\gamma_1$ and $\gamma_2$ are \textit{compatible} when the following conditions are met:
    \begin{itemize}
        \item There exist representatives in the isotopy classes of $\gamma_1$ and $\gamma_2$ that do not intersect on the surface.
        \item If the arcs $\gamma_1$ and $\gamma_2$ share an endpoint, and their untagged versions are different, then this endpoint must be tagged in the same way.
        \item If the untagged versions of $\gamma_1$ and $\gamma_2$ are the same (or equivalently opposite), then they must have the same tagging in exactly one of their endpoints. 
    \end{itemize}
   A \textit{tagged triangulation} $T$ is a maximal collection of pairwise compatible arcs and the $4$-tuple $(S,M,P,T)$ is called a \textit{triangulated surface}.
\end{definition}

%In Definition~\ref{compatible arcs}, conditions two and three may seem confusing at first sight, so we include Figure 1 where one can see a case of compatible and incompatible arcs. 

\begin{remark}
    Instead of considering only singular arcs or closed curves on a surface, we can also consider collections of those, which will be called \emph{multicurves}. Additionally, a multicurve $c$ will be called \emph{compatible} if any pair of distinct arcs or closed curves of $c$ are compatible and there are no closed curves which are contractible or contractible to a single puncture of the surface.
\end{remark}

\begin{remark}\label{self folded to tagged}
    Following definition~\ref{compatible arcs} one can notice that a tagged triangulation ``cuts" the surface $S$ into triangles, with two possible exceptions.\\ 
    The sides of a triangle may not be distinct,  resulting in what is called a \textit{self-folded triangle}. This situation is precisely the reason why the tagged arcs were introduced in the first place . 
    The second exception occurs when both the tagged and the untagged versions of a an arc are part of the triangulation.\\
    The aforementioned exceptions are, in fact, two sides of the same coin, as there is a unique way to transition from one to the other, as illustrated in Figure~\ref{fig:self folded triangle}. Let $p\in P$ and $q\in M\bigcup P$. Suppose that $l$ is a plainly tagged \emph{loop},i.e. a closed curve with both endpoints at $q$ that goes around $p$ and cuts out a \emph{once-punctured monogon} with side $r$, connecting the points $p$ and $q$. Then, we can replace the loop $l$ with the arc $r^\bowtie$, which is isotopic to $r$ and is tagged notched at $p$, setting:
    \[
    x_l= x_r x_{r^\bowtie},
    \]
    in the cluster algebra.
\end{remark}

\begin{remark}\label{tagged to plain triangulation}
A tagged triangulation that does not contain any tagged arcs is called an \textit{ideal triangulation}. Substituting the loop $l$ with the arc $r^\bowtie$, as described in Remark~\ref{self folded to tagged}, is precisely how we assign a tagged triangulation to a given ideal triangulation. Additionally, by replacing all the tagged arcs with plain arcs where possible and reversing the above procedure, we can transitions from a tagged triangulation back to an ideal triangulation. 
\end{remark}

\begin{figure}
    \centering
    \includegraphics[scale=0.7]{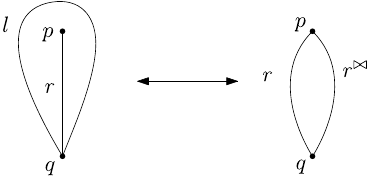}
    \caption{Self folded triangle $rlr$ and the associated tagged arc $r^\bowtie$, where $ x_l= x_r x_{r^\bowtie}$ in the cluster algebra.}
    \label{fig:self folded triangle}
\end{figure}

As mentioned earlier, a cluster algebra can be associated to a given triangulated surface $(S,M,P,T)$. We are now ready to show how a cluster algebra can be defined by a choice of an initial triangulation $T$ of a punctured surface $(S,M,P)$. For this, we are going to define a quiver $Q_T$ without loops or $2$-cycles, associated to $T$. 

\begin{definition}
    Let $(S,M,P)$ be a punctured surface and let $T=\{\tau_1, \tau_2, \dots, \tau_n\}$ be a tagged triangulation of the surface. We define the quiver $Q_T$ associated with the triangulation $T$ as follows:\\
 \underline{Vertices}: For each arc $\tau_i$ of the triangulation $T$, there exists a vertex of the quiver $Q_T$, which we denote by $i$.\\
 \underline{Arrows}: For every triangle in the triangulation $T$, there exists an arrow $i\rightarrow j$ whenever:
 \begin{itemize}
     \item The side $\tau_i$ follows the side $\tau_j$ in clockwise order, if the triangle is not self-folded.
     \item The arc $\tau_i$ is the radius of a self-folded triangle enclosed by a loop $l$, and $l$ follows the side $\tau_j$ in clockwise order within a triangle.
     \item The arc $\tau_j$ is the radius of a self-folded triangle enclosed by a loop $l$, and $\tau_i$ follows the side $l$ in clockwise order within a triangle.     
 \end{itemize}
If the triangulation $T$ contains tagged arcs that can be substituted through the procedure described in Remark~\ref{self folded to tagged}, we first associate a loop to each such arc and then proceed with the construction of the the quiver $Q_T$.
\end{definition}

To define a cluster algebra, we need to determine an initial seed ($\mathbf{x}$, $\mathbf{y}$, $Q$). If $(S, M, P, T)$ is a triangulated surface, we associate a cluster variable $x_i$ and an initial coefficient $y_i$ to each arc $\tau_i$ of $T$. \\
The cluster algebra $\mathcal{A}= \mathcal{A}(S,M,P,T)$ determined by the initial seed ($\mathbf{x}_T$, $\mathbf{y}_T$, $Q_T$), where $\mathbf{x}_T= \{x_1, \dots, x_n\}$ and $\mathbf{y}_T= \{y_1, \dots, y_n\}$, is called the \textit{cluster algebra associated with the triangulated surface $(S, M, P, T)$}.\\

\begin{figure}
    \centering
    \includegraphics[scale=0.5]{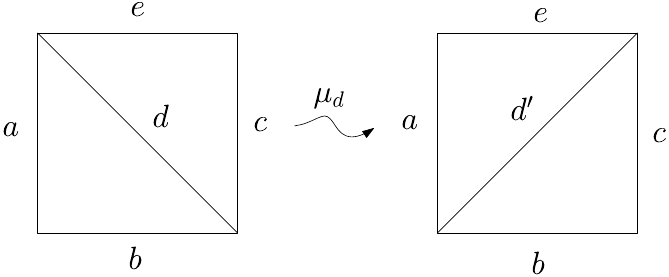}
    \caption{Flipping of the diagonal of a quadrilateral.}
    \label{fig:triangle flip}
\end{figure}

The following dictionary of notions between a triangulated surface $(S,M,P,T)$ and the corresponding cluster algebra $\mathcal{A}(S,M,P,T)$ contains some of the vital elements of this correspondence.
\begin{theorem}
There exist the following bijections of notions between a surface and the associated cluster algebra:
   \begin{align*}
    (S,M,P) & \ \ \ \ \ \ \ \ \  \mathcal{A}(S,M,P,T)\\
    \{\textit{Non-crossing tagged arcs}\} &\longleftrightarrow \{\textit{Cluster variables}\}\\
    \{\textit{Tagged triangulations}\} &  \longleftrightarrow  \{\textit{Clusters}\}\\
    \{\textit{Flips of tagged arcs}\} &  \longleftrightarrow \{\textit{Mutation of seeds}\}
    \end{align*}
\end{theorem}    

Before proceeding, we would like to elaborate further on how flips of tagged arcs work on a surface. First of all, it is well known that in unpunctured surfaces, each arc $d$ of a triangulation $T$ is the diagonal of a quadrilateral in the surface. We can then flip this arc $d$ by substituting it with the unique other diagonal $d'$ of the corresponding quadrilateral, as shown in Figure~\ref{fig:triangle flip}. \\
In punctured surfaces, if the new diagonal of the quadrilateral shares an endpoint with an arc of $T$ that is tagged ``notched" at that endpoint, then the new diagonal must also be tagged ``notched" there, in order to satisfy the second condition of Definition~\ref{compatible arcs}.\\
Additionally, as mentioned in Remark~\ref{self folded to tagged}, in punctured surfaces and in the presence of self-folded triangles, some arcs cannot be flipped. For example, in Figure~\ref{fig:self folded triangle} the arc $r$ cannot be flipped. This is precisely why tagged arcs were introduced: by associating a tagged arc (see Remark~\ref{tagged to plain triangulation}) to the loop that forms one side of the self-folded triangle in question, we are able to flip the arc, as illustrated in Figure~\ref{fig:flip of self folded triangle}.

\begin{figure}
    \centering
    \includegraphics[scale=0.7]{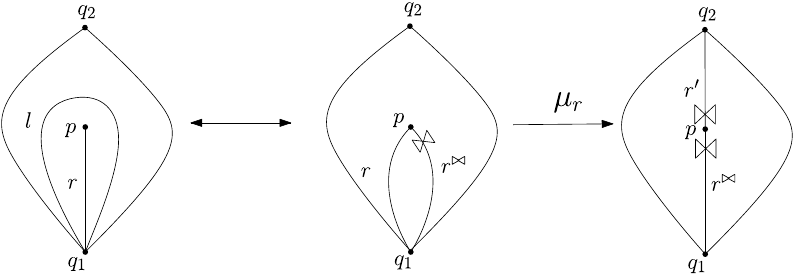}
    \caption{Flipping of a radius of a self-folded triangle.}
    \label{fig:flip of self folded triangle}
\end{figure}

\section{Loop graphs}

Having explained the basic dictionary between surfaces and cluster algebras one question that may arise is the following: Given an arc $\gamma$ on a triangulated surface $(S,M,P,T)$, which may possibly be self-crossing or even a closed curve, is there an element of the cluster algebra that is associated to it? It turns out that such arcs are actually elements of the cluster algebra, which also raises the next question: What is the element in the cluster algebra associated to that arc? This question is answered by the existence of a combinatorial formula for the Laurent expansion of the cluster element associated to the arc $\gamma$. This Laurent expansion is parametrized by the perfect matchings of a suitable graph, depending on the type of the arc $\gamma$.\\
There are three types of arcs on a surface, resulting in three types of graphs. If $\gamma$ is an untagged arc, then we can associate the so-called \emph{snake graph}. If $\gamma^{o}$ is a closed curve on the surface, then the associated graph is called a \textit{band graph}. Finally, if $\gamma^\bowtie$ is a tagged arc (assuming that at least one of its endpoints is tagged notched), then the associated graph is called a \textit{loop graph}. \\

We will start by defining what an abstract snake graph is, as well as giving the notion of a zig-zag in a snake graph which will turn out to be useful in some of the proofs that will follow.  

\begin{definition}\label{snake graph}
    A \textit{snake graph} $\mathcal{G}$ is a connected planar graph consisting of a finite sequence of tiles, which are squares (graphs with four vertices and four edges) in the plane:  $G_1,\dots, G_n,$ with $n\geq 1$,  which are glued together in the following way: two subsequent tiles $G_i$ and $G_{i+1}$ share exactly one edge, which can either be the North edge of $G_i$ and the South edge of $G_{i+1}$, or the East side of $G_i$ and the West edge of $G_{i+1}$.
\end{definition}

\begin{definition}
Let $\mathcal{G}= \{G_1, \dots, G_d\}$ a snake graph.
  \begin{itemize}
    \item A consecutive sequence of tiles $(G_i, G_{i+1}, \dots, G_j), 1\leq i\leq j\leq d$ is called a \textit{subgraph} of $\mathcal{G}$.
    \item Let $\mathcal{G}'$ be a subgraph of $\mathcal{G}$. We will say that $\mathcal{G}'$ is a \textit{zig-zag} if for every tile $G_{i+1}$ of $\mathcal{G}'$ not both of its South and North edges, or West and East edges, are glued to the tiles $G_i$ and $G_{i+2}$ respectively.
    \item We will also say that $\mathcal{G}'$ is the \emph{maximal zig-zag following} (\emph{maximal zig-zag preceding}) \emph{a tile} $G_i$ if adding any more tiles of $\mathcal{G}$ at the end (start) of $\mathcal{G}'$, results in it not being a zig-zag anymore. 
  \end{itemize}    
\end{definition}

We can also define a special function on a given snake graph $\mathcal{G}$, which is called \textit{sign function}. This is a map $f\colon\{$edges of $\mathcal{G}\}\to \{+, -\}$, which satisfies the following properties on every tile $G_i, 1\leq i\leq n$ of $\mathcal{G}$:
\begin{itemize}
    \item the North and West sides of $G_i$ have the same sign,
    \item the South and East side of $G_i$ have the same sign,
    \item the North and South sides of $G_1$ have the opposite sign,
\end{itemize}
as illustrated in Figure~\ref{fig:sign function}.\\
The above sign function can help us define the next category of graphs, the so-called ``band graphs".\\

\begin{figure}
    \centering
    \includegraphics[scale=1]{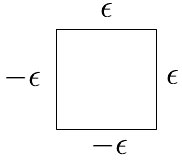}
    \caption{Sign function on a tile of a snake graph where $\epsilon= \pm$.}
    \label{fig:sign function}
\end{figure}

\begin{definition}
    A \textit{band graph} is a snake graph, with the additional property that the North or West side of the tile $G_n$ is identified with the South or East side of the tile $G_1$, provided they have the same sign. We will denote band graphs as $\mathcal{G}^{o}$.
\end{definition}

Lastly, we have yet to define what a \emph{loop graph} is. Loop graphs are, roughly speaking, created by gluing the first or last tile of a snake graph $\mathcal{G}$ to an interior tile of $\mathcal{G}$. The following procedure describing how two edges are ``glued'' together differs slightly from the original given in \cite{wilson2020surface}, since we are concerned with loop graphs associated to tagged arcs on punctured surfaces. Additionally, the definition of loop strings, which will be introduced later (see Definition~\ref{loopstring}), would not work on the general setting.\\

\begin{definition}\label{loop graph}
    Let $\mathcal{G}=\{G_1, \dots, G_d\}, d\geq 3$ be a snake graph and $c$ the edge of the tile $G_d$ which is not adjacent to any vertex of the tile $G_{d-1}$. Let also $x$ be the North-East vertex of $G_d$ and $y$ the remaining vertex adjacent to $c$. Let $\mathcal{G}'=\{G_{i+1}, \dots, G_d\}, i\geq 1$  be the maximal zig-zag preceding the tile $G_d$. We let $c'$ denote the North/East edge of the tile $G_i$ that is a boundary edge of $\mathcal{G}$. If $c'$ is the North edge of $G_i$ we let $y'$ denote the North-East vertex of $G_i$ and $x'$ the remaining vertex adjacent to $c'$.\\
    \indent A \textit{loop graph} $\mathcal{G}^\bowtie$ is obtained from $\mathcal{G}$ by identifying the vertices $x$ and $y$ to the vertices $x'$ and $y'$ respectively, and subsequently the edge $c$ to $c'$. We call the subgraph $\mathcal{G}'$, the \emph{right hook} of the loop graph $\mathcal{G}^\bowtie$. \\
    \indent Dually, let $\mathcal{G}''=\{G_{1}, \dots, G_j\}, j\leq d$  be the maximal zig-zag following the tile $G_1$. By identifying two vertices of the tile $G_1$ with two vertices of the tile $G_{j+1}$ we obtain a \emph{loop graph} $\mathcal{G}'^\bowtie$ and call the subgraph $\mathcal{G}''$ the \emph{left hook} of $\mathcal{G}'^\bowtie$.\\
    \indent Finally, we can do the above procedures at the same time creating a loop graph that has both a left and a right hook.
\end{definition}

\begin{example}\label{Isomorphic loop graphs}
    In Figure~\ref{fig:Loop graphs} we give two examples of graphs $\mathcal{G}^\bowtie_1$ and $\mathcal{G}^\bowtie_2$ with five tiles. For the loop graph $\mathcal{G}^\bowtie_1$ we have the right hook $\mathcal{G}'_1=\{G_3,G_4,G_5\}$, while the right hook of the loop graph $\mathcal{G}^\bowtie_2$ is $\mathcal{G}'_2=\{G_5,G_4,G_3\}$. Additionally, the loop graph $\mathcal{G}^\bowtie_1$ has the boundary edge $c$ on the North side of $G_5$, while the loop graph $\mathcal{G}^\bowtie_2$ has it on the East side of $G_5$. However, one can notice that in fact these two graphs are isomorphic.   
\end{example}

\begin{figure}
    \centering
    \includegraphics[scale=0.8]{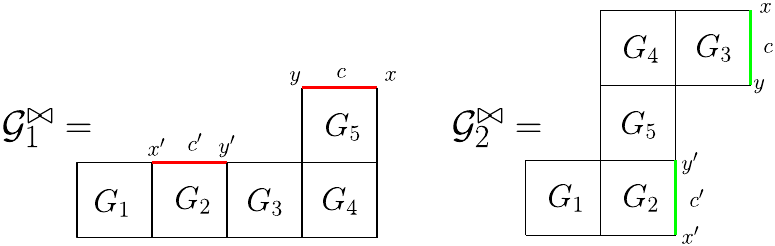}
    \caption{Example of isomorphic loop graphs.}
    \label{fig:Loop graphs}
\end{figure}

Having gone through the definitions of abstract snake, band and loop graphs, we are now ready to explain how one can associate the suitable graph to each arc on a triangulated surface. We will first associate a snake graph to a plain arc of the surface and then proceed to associate a loop graph to a given tagged arc.

\begin{definition}\label{snake graph associated to regular arc}
    Let $T$ be an ideal triangulation of a punctured surface $(S,M,P)$ and $\gamma$ a plain arc in the surface which is not in $T$. We choose an orientation of $\gamma$ and set $s\in M$ to be the starting point and $t\in M$ the ending point of the arc $\gamma$. Let $p_1, p_2, \dots, p_d$ be the intersection points of $\gamma$ with $T$ ordered by the sequence that $\gamma$ intersects $T$, starting from $s$ and finishing at $t$. Let $\tau_{i_j}$ be the arc containing the point $p_j$ for every $1\leq j\leq d$.\\
    \indent Assume first that $\tau_{i_j}$ is not the self-folded side of a self-folded triangle. Then $\tau_{i_j}$ is the diagonal of a quadrilateral $Q_{i_j}$ in $(S,M,P,T)$ and let $\Delta_j$ denote the first triangle in one side of $Q_{i_j}$ that $\gamma$ crosses, and $\Delta_{j+1}$ the other triangle. For every such intersection point $p_j$, we associate a tile $G_j$ which is formed by deleting the diagonal $\tau_{i_j}$ and embedding $Q_{i_j}$ in the plane so that the triangle $\Delta_j$ forms the lower left half of $G_j$.\\
    \indent Assume now that $\tau_{i_j}$ is the self-folded side of a self-folded triangle. Then we associate a tile $G_j$ to $p_j$ by gluing two copies of the self folded triangle such that the labels of the North and West (equivalently South and East) edges of $G_j$ are equal.\\
    The \textit{snake graph $\mathcal{G}_\gamma= \{G_1, \dots, G_d \}$ associated to $\gamma$} is formed by gluing the common edges of $G_j$ and $G_{j+1}$ if both tiles were not formed by an intersection in a self-folded triangle, while they are glued as shown in Figure~\ref{fig:pass through self folded triangle} otherwise.
\end{definition}

\begin{figure}
    \centering
    \includegraphics[scale=0.7]{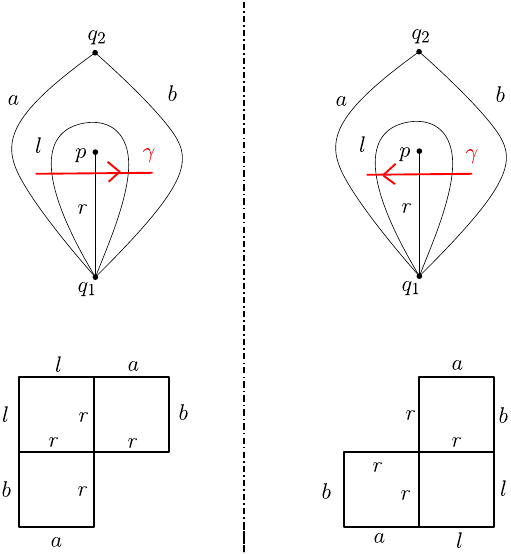}
    \caption{Local configuration of a snake graph, for arcs passing through a self-folded triangle.}
    \label{fig:pass through self folded triangle}
\end{figure}

The definition of a loop graph associated to a tagged arc requires the notion of a \textit{hook} which was used by Labardini-Fragoso and Cerulli Irelli (\cite{labardinifragoso2009quiverspotentialsassociatedtriangulated}, \cite{Cerulli_Irelli_2012})  in their work on representations of quivers with potential.\\

Let $\gamma^\bowtie$ be a tagged arc on a triangulated surface $(S,M,P,T)$ which is tagged ``notched" at its endpoint adjacent to the puncture $p$. We can then associate a new \textit{hooked arc} $\gamma^\bowtie_h$ which is obtained from $\gamma$ by replacing the endpoint of $\gamma^\bowtie$ that is tagged ``notched" with a \textit{hook} which travels around $p$ either clockwise or anti-clockwise intersecting each incident arc of the triangulation $T$ at $p$ once, as illustrated in Figure~\ref{fig:hooked}\\
If both of the endpoints of $\gamma^\bowtie$ are tagged notched, then we do the previous construction in both ends of $\gamma^\bowtie$, in order to obtain a doubly hooked arc $\gamma^\bowtie_h$.

\begin{figure}[h! tbp]
    \centering
    \includegraphics[scale=0.6]{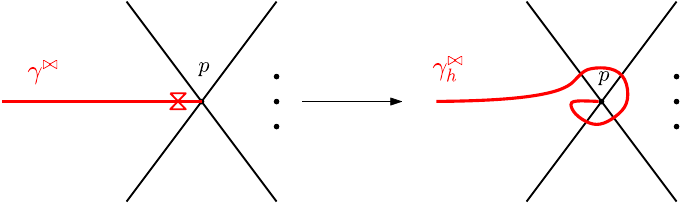}
    \caption{Creating a hooked arc $\gamma^\bowtie_h$ out of a tagged arc $\gamma^\bowtie$.}
    \label{fig:hooked}
\end{figure}

\begin{figure}
    \centering
    \includegraphics[scale=1]{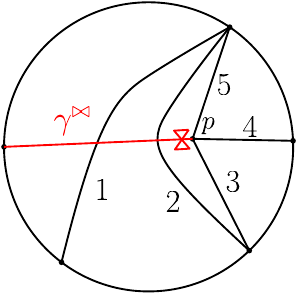}
    \caption{Tagged arc on a surface.}
    \label{fig:tagged arc}
\end{figure}

\begin{definition}\label{loop graph associated to tagged arc}
    Let $T$ be an ideal triangulation of a punctured surface $(S,M,P)$ and $\gamma^\bowtie$ a tagged arc in the surface, such that the underlying plain arc $\gamma$ is not in $T$, and the tagging is not adjacent to a puncture enclosed by a self-folded triangle. Consider $\mathcal{G}_{\gamma^\bowtie_h}= \{G_{h_1}, \dots, G_{h_n}, G_1,\dots G_d\}$ to be the snake graph associated to the hooked arc $\gamma^\bowtie_h$, where $\mathcal{G}_\gamma= \{G_1,\dots G_d\}$ is the snake graph associated to the plain arc $\gamma$.\\ 
    \indent We define $\mathcal{G}_{\gamma^\bowtie}$ to be the \textit{loop graph associated to the tagged arc $\gamma^\bowtie$}, which is obtained from $\mathcal{G}_{\gamma^\bowtie_h}$ by creating a loop at $G_{h_1}$ and $G_1$. \\
    If $\mathcal{G}_{\gamma^\bowtie_h}= \{G_1,\dots G_d, G_{h'_1}, \dots, G_{h'_n}\}$ was the snake graph associated to the hooked arc $\gamma^\bowtie_h$, then we define $\mathcal{G}_{\gamma^\bowtie}$ to be the \textit{loop graph associated to the tagged arc $\gamma^\bowtie$}, which is obtained from $\mathcal{G}_{\gamma^\bowtie_h}$ by creating a loop at $G_{d}$ and $G_{h'_d}$.\\
    Lastly if there is a tagging in both endpoints of the arc $\gamma^\bowtie$, then we have $\mathcal{G}_{\gamma^\bowtie_h}= \{G_{h_1}, \dots, G_{h_n},G_1,\dots G_d, G_{h'_1}, \dots, G_{h'_n}\}$ to be the snake graph associated to the hooked arc $\gamma^\bowtie_h$, and we define $\mathcal{G}_{\gamma^\bowtie}$ to be the \textit{loop graph associated to the tagged arc $\gamma^\bowtie$},  obtained from $\mathcal{G}_{\gamma^\bowtie_h}$ by creating a loop at $G_{h_1}$ and $G_1$ and a loop at $G_{d}$ and $G_{h_d}$.
\end{definition}

\begin{remark}
    One can notice that, by the definition of the hook around a puncture, we have two choices: either we can go clockwise or anti-clockwise around the puncture. This means that there are two different loop graphs associated to a singly tagged arc. However, this is not a problem, since up to isomorphism, the resulting loop graphs are equal, as can be seen in Figure~\ref{fig:Loop graphs}.
\end{remark}

\begin{example}
    In Figure~\ref{fig:tagged arc} the tagged arc $\gamma^\bowtie$ gives rise to two possible loop graphs depending on which direction we follow around the puncture $p$. If we follow the anti-clockwise orientation then the resulting graph is the graph $\mathcal{G}_1$ of Figure~\ref{fig:Loop graphs}, while if we follow the clockwise orientation we take the graph $\mathcal{G}_2$ of the same figure. It is easy to see that these two graphs are isomorphic, as expected. \\

\end{example}

\begin{example}
    In Figure~\ref{fig:tagged arc passing through a self folded triangle}  the tagged arc $\gamma^\bowtie$ is passing through a self-folded triangle (the one defined by the arcs $8$ and $9$). Therefore, in the construction of the loop graph $\mathcal{G}_{\gamma^\bowtie}$ we need to follow the local configuration appearing on the left hand side of Figure~\ref{fig:pass through self folded triangle}.
\end{example}

\begin{figure}
    \centering
    \includegraphics[scale=0.8]{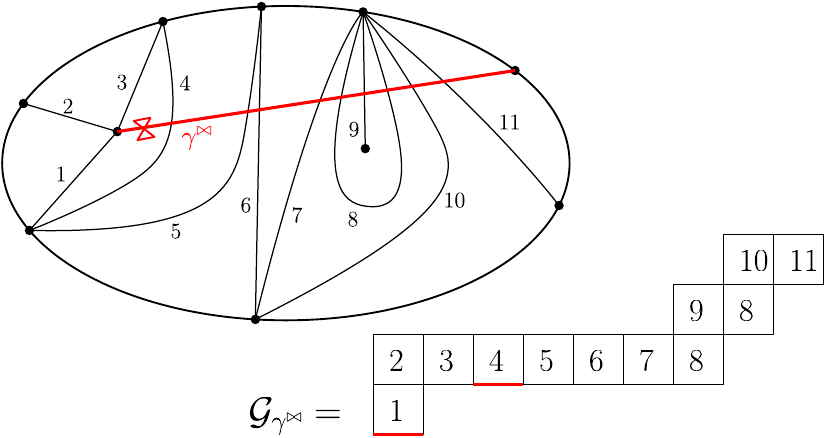}
    \caption{Tagged arc passing through a self folded triangle.}
    \label{fig:tagged arc passing through a self folded triangle}
\end{figure}

\section{Expansion formula of tagged arcs}

In the previous section we associated a loop graph to each tagged arc of a surface. The importance of this association becomes evident in this section, as initially Musiker, Schiffler and Williams \cite{Musiker_Schiffler_Williams_2013} and later Wilson \cite{wilson2020surface} provided an expansion formula for the associated cluster variable of an arc in the surface using these graphs. We will include only Wilson's result at the end of this section, as it is a generalization that encompasses all possible cases of a given arc in a surface.\\
We will start by explaining what a perfect matching of a graph is and how to associate a monomial to every perfect matching of a snake or loop graph.\\
A \textit{perfect matching} of a graph $\mathcal{G}$ is a subset $P$ of the edges of $\mathcal{G}$ such that each vertex of $\mathcal{G}$ is incident to exactly one of the edges in $P$. We define $\Match{\mathcal{G}}$ to be the collection of all perfect matchings of $\mathcal{G}$.\\

Each snake or loop graph has precisely two perfect matchings that contain only boundary edges. We will discuss the matchings for snake graphs and loops graphs separately.\\
If $\mathcal{G}$ is a snake graph, then we will denote by $P_{+}$, and call it the \textit{maximal perfect matching of $\mathcal{G}$}, the perfect matching that contains only boundary edges and includes the West side of the first tile of $\mathcal{G}$. Dually, we will denote by $P_{-}$, and call it the \textit{minimal perfect matching of $\mathcal{G}$}, the perfect matching that contains only boundary edges and includes the South side of the first tile of $\mathcal{G}$.\\
If $\mathcal{G}^\bowtie$ is a loop graph, the maximal perfect matching $P_+$ of $\mathcal{G}^\bowtie$ is defined to be the perfect matching that contains only boundary edges and is a subset of the maximal perfect matching associated to the plain snake graph $\mathcal{G}$. The minimal perfect matching of a loop graph is defined dually.\\

\begin{remark}
    A loop graph $\mathcal{G}^\bowtie$ can be represented in more than one ways as it was pointed out in Example~\ref{Isomorphic loop graphs}. Therefore, defining the maximal perfect matching of $\mathcal{G}^\bowtie$, relying on what we call the plain snake graph $\mathcal{G}$, could potentially create a problem. Namely, if $\mathcal{G}_1^\bowtie$ and $\mathcal{G}_2^\bowtie$ are isomorphic loop graphs, we must have that $P_{\mathcal{G}_{1,+}^\bowtie}=P_{\mathcal{G}_{2,+}^\bowtie}$. \\
    
\end{remark}

\begin{figure}
    \centering
    \includegraphics[scale=0.8]{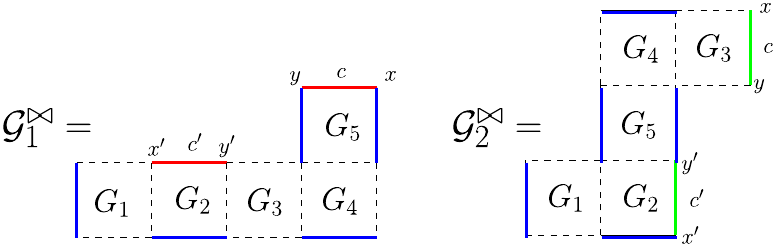}
    \caption{Maximal perfect matching of two isomorphic loop graphs.}
    \label{fig:Max perfect matching of loop graphs}
\end{figure}

\indent Let $P_{-}\ominus P= (P_{-}\cup P) \setminus (P_{-}\cap P)$ denote the \textit{symmetric difference} of an arbitrary perfect matching of a graph $G$ with the minimal perfect matching $P_{-}$. Notice that $P_{-}\ominus P$ is a set of boundary edges of a subgraph $G_P$ of $G$, which consists of a subset of tiles $G_i$ of the graph $G$. 

\begin{definition}
    Let $P\in\Match{G}$. The \textit{height monomial of $P$} is defined by
    \[
    y(P)= \prod_{G_i \in G_P}{y_i}.
    \]
    The \textit{weight monomial of $P$} is defined by
    \[
    x(P)= \prod_{\tau_i\in P}{x_{\tau_i}}
    \]
\end{definition}

\begin{theorem}
    Let $(S,M,P)$ be a punctured surface and $T$ an ideal triangulation of the surface. Let  $\mathcal{C}= \mathcal{C}(S,M,P,T)$ be the associated cluster algebra with principal coefficients, with respect to $T$. Suppose $\gamma^\bowtie$ is a (tagged) arc, which is not tagged notched at a puncture enclosed by a self folded triangle in $T$. Then the cluster variable $x_{\gamma^\bowtie}$ is equal to 
    \[
    x_{\gamma^\bowtie} = \frac{1}{\cross(\gamma, T)} \sum_{P\in \Match{G_{\gamma^\bowtie}}} {x(P)y(P)}
    \]
    where $x(P)$ is the weight of the perfect matching $P$, $y(P)$ is the height of $P$ and $\cross(\gamma^\bowtie)= \prod_{j=1}^{d}{x_{\tau_{i_j}}}$ is the \textit{crossing monomial of $\gamma^\bowtie$}. 
\end{theorem}

\begin{remark}
    As it is the case with unpunctured surfaces, we can also assign a cluster algebra element to each multicurve $c$ by defining $x_c=\prod_{\gamma\in c} x_\gamma$, where $\gamma$ runs over all arcs, tagged arcs and closed curves of $c$.
\end{remark}

\section{Quivers associated to loop graphs}

In this subsection we first recall from \cite{MSWpositivity} how we can associate a quiver to a given snake graph and later recall from \cite{wilson2020surface}, the generalized construction of the quiver associated to a loop graph.

\begin{definition}\cite{MSWpositivity} \label{quiver of snake graph}
Let $\mathcal{G}=(G_1,...,G_n)$ be a snake graph. The \textit{quiver associated to the graph $\mathcal{G}$} is defined to be the quiver $Q_\mathcal{G}$  with vertices ${1,...,n}$ and whose arrows are determined by the following rules:
\begin{itemize}
    \item [(i)] 
    there is an arrow $i\to i+1$ in $Q_\mathcal{G}$ if  $i$ is odd and $G_{i+1}$ is on the right of  $G_i$, or $i$ is even and $G_{i+1}$ is on the top of $G_i$,
    \item [(ii)]
    there is an arrow $i\to i-1$ in $G_\mathcal{G}$ if $i$ is odd and $G_{i}$ is on the right of  $G_{i-1}$, or $i$ is even and $G_{i}$ is on the top of $G_{i-1}$.
\end{itemize}
\end{definition}

Following Wilson, we can now define a quiver to a given loop graph.

\begin{definition}\cite{wilson2020surface}\label{quiver of a loop graph}
Let $\mathcal{G}^\bowtie=(G_1,\dots,G_n)^\bowtie$ be a loop graph, with underlying snake graph $\mathcal{G}=(G_1,\dots,G_n)$. The \textit{quiver associated to the loop graph $G^\bowtie$} is defined to be the quiver $Q_{\mathcal{G}^\bowtie}$ which is the same as the quiver $Q_\mathcal{G}$ with an additional arrow for each loop of $\mathcal{G}^\bowtie$, where the additional arrow is defined as follows:
\begin{itemize}
    \item [(i)]
    if there is a loop with respect to $G_1$ and $k\in \{2,\dots,n\}$ then the arrow $1 \to k$ (resp. $k \to 1$) is in $Q_{\mathcal{G}\bowtie}$ if $k$ is odd (resp. even) and $S(G_k)$ is the cut edge, or $k$ is even (resp. odd) and $W(G_k)$ is the cut edge,
    \item[(ii)]
    if there is a loop with respect to $G_n$ and $k\in \{1,\dots,n-1\}$ then the arrow $n \to k$ (resp. $k \to n$) is in $Q_{\mathcal{G}\bowtie}$ if and only if $k$ is odd (resp. even) and $N(G_k)$ is the cut edge, or $k$ is even (resp. odd) and $E(G_k)$ is the cut edge.
\end{itemize}
\end{definition}

One thing that one may notice, is that a loop graph has equivalent representations which could possibly lead to different quivers associated to each representation. However Definition~\ref{quiver of a loop graph} is consistent in the sense that the quiver that is associated to a loop graph is independent of the choice of the representation of the loop graph as the following propositions indicates.\\

\begin{proposition}\label{quiver associated to loop garph is unique}
    Let $\mathcal{G}_1^{\bowtie}$ and $\mathcal{G}_2^{\bowtie}$ be two different representations of the same loop graph $\mathcal{G}^{\bowtie}$. Then, the associated quivers $Q_{\mathcal{G}_1^{\bowtie}}$ and $Q_{\mathcal{G}_2^{\bowtie}}$ of the graphs $\mathcal{G}_1^{\bowtie}$ and $\mathcal{G}_2^{\bowtie}$ respectively are the same quiver.
    \begin{proof}
        We will assume that there is only one loop at the start of the loop graph $\mathcal{G}^{\bowtie}$, since the arguments for a possibly second loop at the end of the graph would be similar.\\
        Let $\mathcal{G}_1^{\bowtie}=\{G_1,\dots G_{d-1}, G_d,\dots G_n\}, 2\leq d\leq n$ where there is a loop with resect to the tiles $G_1$ and $G_d$. Therefore since there is only one loop, there are two possible representations of the loop graph and therefore it must be $\mathcal{G}_2^\bowtie= \{G_{d-1},\dots G_1, G_d,\dots G_n\}$. \\
        Let w.l.o.g. assume that $d$ is even and that the tile $G_2$ is on the right of the tile $G_1$ at the loop graph $\mathcal{G}_1^{\bowtie}$. We will now construct the quivers $Q_{\mathcal{G}_1^{\bowtie}}$ and $Q_{\mathcal{G}_2^{\bowtie}}$ separately.\\
        \indent Starting with $Q_{\mathcal{G}_1^{\bowtie}}$ we first construct the quiver $Q_{\mathcal{G}_1}$ associated to the underlying snake graph $\mathcal{G}_1$. The local configuration induced by the tiles $\{G_d,\dots,G_n\}$ at the quiver $Q_{\mathcal{G}_1}$ is the same as the local configuration at the quiver $Q_{\mathcal{G}_2}$, so we need to focus only on the arrows of the quiver which are induced from the first part of the loops graphs, up to the tile $G_d$.\\
        Since $G_2$ is on the right of $G_1$ following Definition~\ref{quiver of snake graph} we have an arrow $1\rightarrow2$ in $Q_{\mathcal{G}_1}$. Since $\{G_1,\dots G_{d-1}\}$ is a zig-zag we must have the following configuration locally in $Q_{\mathcal{G}_1}$: $1\rightarrow 2\rightarrow \dots\rightarrow d-1$. Additionally the zig-zag $\{G_1,\dots G_{d-1}\}$ is also maximal and therefor we have the arrow: $d-1\leftarrow d$. Lastly we need to add an arrow that is induced by the loop at the tiles $G_1$ and $G_d$. Since $d$ is an even number the tile $G_d$ is over the tile $G_{d-1}$ and therefore $W(G_d)$ is the cut edge in the loop graph $\mathcal{G}_1^{\bowtie}$. Following Definition~\ref{quiver of a loop graph}(i) we have that the quiver $Q_{\mathcal{G}_1^{\bowtie}}$ is the same as the quiver $Q_{\mathcal{G}_1}$ with an additional arrow $1\rightarrow d$.\\
        \indent We will now construct the quiver $Q_{\mathcal{G}_2^\bowtie}$ associated to the loop graph $\mathcal{G}_2^\bowtie$. $\mathcal{G}_2^\bowtie= \{G_{d-1},\dots G_1, G_d,\dots G_n\}$ and as stated earlier we need to only investigate the configuration of arrows in $Q_{\mathcal{G}_2^\bowtie}$ induced by the first part $\{G_{d-1},\dots G_1, G_d\}$.\\
        To begin with, the tile $G_1$ is on the left of the tile $G_d$ in the loop graph $\mathcal{G}_2^{\bowtie}$ since the cut edge in $\mathcal{G}_1^{\bowtie}$ is $W(G_d)$. Notice also that $G_1$ is an odd tile $\mathcal{G}_2$ and therefore by Definition~\ref{quiver of snake graph} (i) there is an arrow $1\to d$ in $Q_{\mathcal{G}_2}$. Since $\{G_{d-1},\dots G_1\}$ is a maximal zig-zag in $\mathcal{G}_2$ we must also have the following local configuration in $Q_{\mathcal{G}_2}$: ${d-1}\leftarrow \dots\leftarrow 1$.\\
        Lastly we need to add an extra arrow induced by the loop with respect to the tiles $G_{d-1}$ and $G_{d}$ in order to complete the quiver $Q_{\mathcal{G}^\bowtie_2}$. The cut edge in $\mathcal{G}^\bowtie_1$ was $W(G_d)$ which implies that the cut edge in $\mathcal{G}^\bowtie_2$ is $S(G_d)$. The cut edge being $S(G_d)$ and $d$ being even implies by Definition~\ref{quiver of a loop graph} that in $Q_{\mathcal{G}^\bowtie_2}$ we have additionally the arrow $d\rightarrow {d-1}$.\\
        Noticing that the quivers constructed $Q_{\mathcal{G}^\bowtie_1}$ and $Q_{\mathcal{G}^\bowtie_2}$ are the same completes the proof.
        
    \end{proof}
\end{proposition}

A nice visualization of the above proof can be done by looking at the Figure~\ref{fig:quivers associated to loop graph} where there is a loop at the start of the loop graph, ignoring the second loop at the end of it.

\begin{example}\label{quivers of loop graphs}
    In Figure~\ref{fig:quivers associated to loop graph} the loop graph on the left hand side has two loops, one at the start of the graph which ``connects" the tiles $1$ and $4$ and one at the end which ``connects" the tiles $6$ and $10$. Following Definition~\ref{quiver of snake graph} and viewing the loop graph as a regular snake graph we take the quiver:
    \[
    1\rightarrow 2 \rightarrow 3 \leftarrow 4\rightarrow 5 \rightarrow 6 \leftarrow 7 \rightarrow 8 \rightarrow 9 \rightarrow 10.
    \]
    In order to construct the quiver of the loop graph, following the two rules in Definition~\ref{quiver of a loop graph} we add the arrow $1\rightarrow 4$ and the arrow $6\rightarrow 10$ to the quiver resulting in the final quiver that can be viewed in the same figure. \\
    Working on the loop graph on the right hand side of Figure~\ref{fig:quivers associated to loop graph}, the quiver associated on the plain snake graph of that loop graph is the following:
    \[
    3\leftarrow 2\leftarrow 1\rightarrow 4\rightarrow 5\rightarrow 6 \rightarrow 10 \leftarrow 9 \leftarrow 8 \leftarrow 7.
    \]
    Looking at the two loops at the start and at the end of the loop graph we also add the arrows $3\leftarrow 4$ and $6\leftarrow 7$ respectively.\\ 
    It is easy to see that the quivers associated to both loop graphs (which are equivalent) are the same.
\end{example}

\begin{figure}
    \centering
    \includegraphics[scale=0.7]{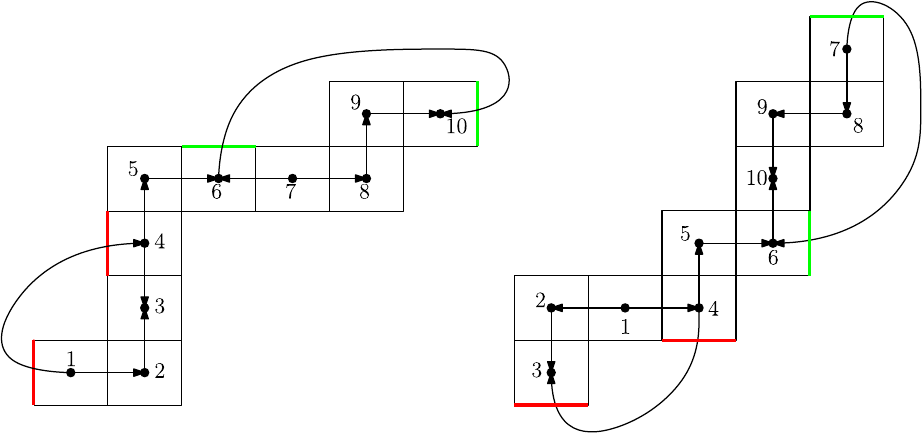}
    \caption{Quiver associated to loop graph.}
    \label{fig:quivers associated to loop graph}
\end{figure}

\section{Abstract strings and loopstrings}
Abstract strings are a well-established tool which has been used as a means of simplifying the information of a snake graph by associating a string to a given graph and by extension to a given arc on a surface. In this chapter we introduce a generalization of such abstract strings, which we call loopstrings. The definition of a loopstring captures the combinatorial structure of a loop graph and is similar to the notion of an abstract string associated with a snake graph.

%We are introducing this notion as a means of associating a parallel to the classical strings, to a given loop graph. Subsequently we have decided to define them, imposing conditions which are parallel versions of the ones in loop graphs. \\

Let $\{\rightarrow,\leftarrow\}$ be a set of two letters where the first one is called direct arrow and the second one inverse arrow. An \textit{abstract string} is a finite word in this alphabet or is the additional word denoted by $\emptyset$, which stands for the empty word. We will generalize this construction by adding two new letters in this alphabet and by imposing additional properties for these letters, which will correspond to the loops that appear in loop graphs. 

\begin{definition}\label{loopstring}
    Let $\{\hookleftarrow,\hookrightarrow\}$ be a set of two letters where the first one is called \emph{inverse loop} and the second one \emph{direct loop}.
    We create a new alphabet from the set $\{\rightarrow,\leftarrow,\hookleftarrow,\hookrightarrow\}$. An \textit{abstract loopstring} $l$ is either an abstract string, or an abstract string with maximum two additional letters from $\{\hookleftarrow, \hookrightarrow\}$ subject to the following:
    \begin{itemize}
       \item A direct or inverse loop cannot be the first or last letter of the string $l$. 
       \item A direct (respectively inverse) loop can be placed in a position only if all the following arrows up to the end of the $l$ or the preceding arrows up to   the start of $l$ are inverse (respectively direct) arrows.
    \end{itemize}

\end{definition}

To make the above Definition~\ref{loopstring} clear we will present some examples of loopstrings and some examples of non-loopstrings, pointing out the rules that are violated in each case.

\begin{example}
    The following sequences are loopstrings: 
    \begin{enumerate}
    \item[(i)] $\leftarrow \leftarrow \hookrightarrow \rightarrow \rightarrow \leftarrow \hookrightarrow \leftarrow \leftarrow \leftarrow$,
    \item[(ii)] $\rightarrow \rightarrow \rightarrow \hookleftarrow \leftarrow \rightarrow$,
    \item[(iii)] $\leftarrow \rightarrow \hookleftarrow \rightarrow$,
    \item[(iv)] $\rightarrow \rightarrow \rightarrow \hookleftarrow \rightarrow \rightarrow$,
    \item[(v)]$\leftarrow \hookrightarrow \hookrightarrow \leftarrow \leftarrow$.
    \end{enumerate}
    The following sequences are not loopstrings:
    \begin{enumerate}
        \item[(vi)] $\hookleftarrow \leftarrow \rightarrow$,
        \item[(vii)] $\leftarrow \leftarrow \rightarrow \hookrightarrow \leftarrow \rightarrow$ ,
        \item[(viii)] $\leftarrow \hookleftarrow \rightarrow \leftarrow$,
        \item[(ix)] $\rightarrow \hookleftarrow \hookleftarrow \hookrightarrow \leftarrow \leftarrow$.
    \end{enumerate}
The sequences (i) and (v) have exactly two loops. These loops satisfy the conditions of Definition~\ref{loopstring} since the first loops encountered in each case follows arrows of opposite direction, while the second loop in both cases precedes arrows of opposite direction. The sequences (ii) and (iii) satisfy all conditions since the inverse loop in (ii) follows a sequence of three consecutive direct arrows, while the inverse loop in (iii) precedes a sequence of one direct arrow. Regarding the sequence (iv) there is one inverse loop which both precedes and follows a sequence of direct arrows, which is obviously more than enough.\\The fact that it can be considered either as the start or the end translates to being associated to two different tagged arcs on a punctured surface.\\
The sequence (ix) has three loops and therefore cannot be a loopstring. The sequence (vi) is not a loopstring, since the loop is the first letter of the sequence. The sequence (vii) is not a loopstring, since on the left of the direct loop there is a direct arrow, while on the right of the loop there may exist an inverse arrow but this sequence of direct arrows (in this case, a sequence of one arrow) is not continuing up the end of the string. Lastly the sequence (viii) is not a loopstring since the loop follows a sequence of inverse arrows, but it should have been a direct loop for (viii) to be considered a loopstring.
\end{example}

\begin{remark}
    The reader can notice a repeating pattern on the second condition of Definition~\ref{loopstring}. We require the direct or inverse loop to follow or precede a sequence of inverse or direct arrows respectively. This is a parallel to the notion of maximal zig-zag that follows or precedes a tile as it was encountered in Definition~\ref{loop graph}. This comes as no surprise since in the classical case of associating a string to a snake graph, zig-zag tiles give rise to a sequence of direct or inverse arrows.
\end{remark}

\begin{remark}
    Our Definition~\ref{loopstring} is not unique in the sense that we could have taken a different root. In the case of loop graphs, one defines at first a snake graph and then identifies some edges on that graph. We could try to define a loopstring, starting with a string and then adding the two new letters of the alphabet in the suitable spots. However we decided to stick to our approach since now loopstrings are defined more abstractly instead of depending on a preexisting string, resembling in some way the way that abstract snake graphs were defined in the first place. 
\end{remark}

We finish this chapter by remarking that although abstract loopstrings are given as words on an alphabet, in practice when we are dealing with loopstrings associated to loop graphs or tagged arcs, we can view them as a special kind of graphs, by adding numbers in between the arrows and the loops, which correspond to the vertices of a graph. Another interesting observation is the fact if $(S,M,P,T)$ is a triangulated surface, then loopstrings associated to tagged arcs, as it is also the case with regular strings associated to regular arcs, have a close connection to the quiver associated to the triangulation $T$ of the surface $(S,M,P)$. This connection will be better understood through examples later, after we introduce the connections of loops graphs and loopstrings.

\section{Construction of the loop graph of an abstract loopstring}

In this section we will describe how one can go from an abstract loopstring to a loop graph. In practice, we will be usually working the other way around, building a loopstring out of a loop graph as it is described in the next section, but still it is important to point out that this procedure is bijective.\\
We start by defining the plain string associated to a loopstring which will simplify the construction.

\begin{definition}\label{plain string associated to loopstring}
   Let $w = a_1 \dots a_n$ be an abstract loopstring where $a_{h_1}, a_{h_2}\in\{\rightarrow,\leftarrow,\hookleftarrow,\hookrightarrow\}, 2\leq h_1, h_2\leq n-1$ and $a_i\in \{\rightarrow,\leftarrow\}$ for every $1\leq i\leq n$ with $i\neq h_1, h_2$. We set $a'_{h_1}$ to be equal to $\rightarrow$ (resp. $\leftarrow$) if  $a_{h_1}$ is equal to $\rightarrow$ or $\hookrightarrow$ (resp. $\leftarrow$ or $\hookleftarrow$). Similarly we set $a'_{h_2}$ to be equal $\rightarrow$ or $\leftarrow$ based on what $a_{h_2}$ is. 
   We call $w'= a_1 \dots a'_{h_1}\dots a'_{h_2}\dots a_n$ to be the \emph{plain string associated to the loopstring} $w$.
\end{definition}

We will now construct the loop graph of a loop string by combining a classic construction of Çanakçı-Schroll and some of the tools that we have already introduced in previous chapters.

\begin{definition}\label{loop garph associated to a loopstring}
Let $w= a_1 \dots a_{h_1}\dots a_{h_2}\dots a_n$ be a loopstring where $a_{h_1}, a_{h_2}\in\{\rightarrow,\leftarrow,\hookleftarrow,\hookrightarrow\}$. Let also $w'$ be the plain string associated to $w$. Let $G_{w'}=\{G_1,\dots G_{h_1}\dots G_{h_2}\dots G_n\}$ be the snake graph associated to the $w'$ as it is described in \cite{canakcci2021lattice}. \\
We call the loop graph $G^\bowtie_{w}$ the \emph{loop graph associated to the loopstring} $w$, where $G^\bowtie_{w}$ is the loop graph obtained from the snake graph $G_{w'}$ by identifying the edges of the tiles $G_1$ and $G_{h_1}$ (resp. $G_{h_2}$ and $G_n$) if $a_{h_1}\in \{\hookleftarrow,\hookrightarrow\}$ (resp. $a_{h_2}\in \{\hookleftarrow,\hookrightarrow\}$) as described in Definition~\ref{loop graph}.
    
\end{definition}

\section{Construction of the loopstring of a loop graph}

Given a loop graph $\mathcal{G}$, we want to associate a loopstring $w$. We will do this by combining a variety of construction that have already been introduces. Of course this is not the only way that one can define it, but we chose this approach, since it builds on previous results.\\

\begin{definition}

Let $\mathcal{G}^\bowtie = (G_1,\dots,G_n)$ be a loop graph and $\mathcal{G}=(G_1,\dots,G_n)$ the underlying snake graph associated to $\mathcal{G}^\bowtie$. Let $Q_\mathcal{G}$ be the quiver of the underlying snake graph $G=(G_1,\dots,G_n)$ (which is unique based on Remark~\ref{quiver associated to loop garph is unique}). We define $w' = a_1 \dots a_n$ where:
\begin{itemize}
    \item $a_i = \rightarrow$ if there is an arrow $i\to i+1$ in $Q_\mathcal{G}$,
    \item $a_i = \leftarrow$ if there is an arrow $i+1\to i$ in $Q_\mathcal{G}$.
\end{itemize}
We now associate possibly two more letters from $\{\hookleftarrow, \hookrightarrow\}$ in the following cases:
\begin{itemize}
    \item if there is a loop with respect to $G_1$ and $G_{h_1}, 1<h_1<n$ then we define $a_{h_1} = \hookleftarrow$ (resp. $a_{h_1} = \hookrightarrow$) if the arrow $1\to h_1$ (resp. $h_1\to 1$) is in $Q_{\mathcal{G}\bowtie}$,
    \item if there is a loop with respect to $G_{h_2},1<h_2<n$ $G_d$ then we define $a_{h_2} = \hookleftarrow$ (resp. $a_{h_2} = \hookrightarrow$) if the arrow $h_2\to n$ (resp. $n\to h_2$) is in $Q_{\mathcal{G}\bowtie}$.
\end{itemize}
The loopstring $w$ is then defined as $w= a_1 \dots a_{h_1}\dots a_{h_2}\dots a_n$ and is called \emph{the loopstring associated to the loop graph} $\mathcal{G}^\bowtie$.

\end{definition}

The next remark is a well known fact about snake graphs, but it can be easily seen that it also holds true for loop graphs. We mention it here for the sake of completion and since it will be widely used later in the proofs of the main results.

\begin{remark}
    Based on the definition of a loopstring associated to a loop graph we can notice two things:
    \begin{itemize}
        \item if three consecutive tiles of the loop graph form a zig-zag then the two arrows associated to these tiles have the same direction.
        \item if three consecutive tiles of the loop graph form a straight piece then the two arrows associated to these tile have alternating directions.
    \end{itemize}
\end{remark}

\begin{remark}\label{planarreprofloopgraph}
   Notice that equivalent planar representations of a loop graph will produce two different loopstrings. We will introduce an equivalence relation for loopstrings which recovers equivalent planar representations of a loop graph. 
\end{remark}

\begin{definition}\label{eqloopstrings}
    Let $w=a_1\dots a_d \dots a_n$ and $w' = b_1\dots b_d \dots b_n$ be two different loopstrings, for which $a_d, b_d\in\{\hookleftarrow,\hookrightarrow\}$ and $a_i, b_i\in\{\leftarrow,\rightarrow\}$ for every $1\leq i\leq n$ and $i\neq d$. Then we will say that the two loopstrings are \textit{left equivalent} if $w = w'$ or the following are true:
    \begin{itemize}
        \item[(i)] $a_d\neq b_d$,
        \item[(ii)] if $\{a_1, \dots, a_{d-1}\}$ are direct arrows (resp. inverse arrows), then we must have that $\{b_1, \dots, b_{d-1}\}$ are inverse arrows (resp. direct arrows),
        \item[(iii)] $a_i = b_i$ for each $d+1 < i \leq n$.
    \end{itemize}
    We can define similarly \emph{right equivalent} loopstrings when there is a direct or inverse loop only at the end part of the loopstrings. Lastly we call two loopstrings $w, w'$ \emph{equivalent} if they are left and right equivalent and we will write $w\sim w'$.

\end{definition}

Notice that when two loopstrings $w_1$ and $w_2$ are left equivalent, then the arrow $a_{d+1}$ has opposite direction from the arrow $b_{d+1}$.\\

\begin{example}
    Looking at Figure~\ref{fig:quivers associated to loop graph} we have that the loopstring $w$ associated to the loop graph on the left hand side is the following:
    \[
    w= \rightarrow \rightarrow \hookleftarrow \rightarrow \rightarrow \hookleftarrow \rightarrow \rightarrow \rightarrow,
    \]
    while the loopstring $w'$ associated to the graph on the right hand side of the figure is:
    \[
    w'= \leftarrow \leftarrow \hookrightarrow \rightarrow \rightarrow \hookrightarrow \leftarrow \leftarrow \leftarrow.
    \]
    Following Definition~\ref{eqloopstrings} it is easy to see that $w\sim w'$ which is what we should hope for in order for our constructions to be consistent. This is exactly what we also prove in the next Lemma~\ref{eqofloopgraphs}.
\end{example}

The fact that the above defined relation is indeed an equivalence relation is easy to see, since all properties can be trivially checked. However, what is more important is that equivalent loopstrings, correspond to isomorphic loop graphs. We will prove this for the case that the two loopstrings are left equivalent, but similar arguments can be used for the other cases.

\begin{lemma}\label{eqofloopgraphs}
    %{\color{red} I need to rewrite this lemma.}
    Suppose that $w = a_0\dots a_n$ and $w' = b_0\dots b_n$ are two loopstrings, which have a direct or inverse loop at the beginning and no direct or inverse loop at the end. Then, the associated loop graphs $\mathcal{G}_w^\bowtie$ and $\mathcal{G}_{w'}^\bowtie$ are isomorphic if and only $w$ and $w'$ are left equivalent.
\end{lemma}
\begin{proof}
    If $w = w'$ there is nothing to prove. So suppose from now on, that $w \neq w'$.
    We will first prove the direct implication.\\
    Suppose that the loop graphs $\mathcal{G}_w^\bowtie = \{G_1, \dots, G_k,\dots, G_n\}^\bowtie$ and \\ $\mathcal{G}_{w'}^\bowtie = \{G_1', \dots, G_l',\dots, G_n'\}^\bowtie$ are isomorphic, where $G_k$ and $G_l'$ are the tiles in which one of their edges is identified with on edge of the first tile of the respective loop graph. Since $w \neq w'$, the underlying snake graphs $\mathcal{G}_w$ and $\mathcal{G}_{w'}$ without the identification of the two respective edges are not equal. Additionally, since the two graphs $\mathcal{G}_w^\bowtie$ and $\mathcal{G}_{w'}^\bowtie$ are isomorphic we conclude that $k=l$. \\
    W.l.o.g. assume that $G_2$ is on the right of $G_1$.  Since the two graphs are isomorphic but not equal when viewed as plain snake graphs, we conclude that $G_2'$ must be on top of $G_1'$, since there are only two possible different planar representations of a loop graph (Remark \ref{planarreprofloopgraph}).
    Therefore, we can deduce that $a_0 = \hookleftarrow$ and $b_0 = \hookrightarrow$. This, in turn implies that $a_i = \rightarrow$ and $b_i = \leftarrow$ for each $2\leq i < k-1$, by definition of loopstrings. Lastly, since the two graphs are isomorphic, the subgraphs $\mathcal{G}_v^\bowtie = \{ G_k,\dots, G_n\}^\bowtie$ and $\mathcal{G}_{v'}^\bowtie = \{ G_l',\dots, G_n'\}^\bowtie$ are equal, and therefore we have that $a_j = b_j$ for each $k\leq j\leq n$. Therefore, the loopstrings $w$ and $w'$ are left equivalent.\\
    
    Let us prove now the other direction. Suppose that the loopstrings $w$ and $w'$ are left equivalent. Suppose w.l.o.g that $a_0 = \hookleftarrow$. Since they are not equal, we also have that $b_0 = \hookrightarrow$ from (i) of Definition \ref{eqloopstrings}. Suppose also that $a_1,\dots, a_k$ is the maximum subcollection of consecutive tiles of $w$ which are direct arrows. Therefore, due to (ii) of Definition \ref{eqloopstrings}, $b_1,\dots, b_k$ is the maximum subcollection of consecutive tiles of $w'$ which are inverse arrows. we will treat the case that $k$ is even, since the other case is symmetrical.\\
    Sine $k$ is even, the tile $G_{k+1}$ of the graph $\mathcal{G}^\bowtie_w$ is on the right of the tile $G_k$ and the south edge of this tile $S(G_{k+1})$ is identified with $W(G_1)$, by the construction of the loop graph by a given loopstring. Similarly, regarding the graph $\mathcal{G}^\bowtie_{w'}$, we can see that $G'_{k+1}$ is on top of $G'_k$ and that the edge $W(G'_{k+1})$ is identified with the edge $S(G'_1)$. By identifying each tile $G_i$ of the graph $\mathcal{G}^\bowtie_w$ with the tile $G'_{k+1-i}$ of the graph $\mathcal{G}^\bowtie_{w'}$, for every $1\leq i\leq k$, and each tile $G_j$ with $G'_j$ for every $k+1\leq j\leq n$. We can construct an isomorphism $\Phi\colon \mathcal{G}^\bowtie_w\to\mathcal{G}^\bowtie_{w'}$ which sends each vertex and edge of the graph  $\mathcal{G}^\bowtie_w$ to the appropriate edge and vertex of $\mathcal{G}^\bowtie_{w'}$ respectively, keeping in mind the aforementioned identification of tiles. Therefore, two left equivalent loopstrings, give rise to two different planar representations of the same loopgraph.
\end{proof}

\begin{remark}
    Using the above definition we can see that two different planar representations of the same loop graph produce two equivalent loopstrings. So working under equivalence of loop graphs, the construction of loopstrings coming from a loop graph is well defined. From now on when we are talking about loopstrings we will always mean it up to equivalence of loopstrings.
\end{remark}

\begin{remark}
    Lemma \ref{eqofloopgraphs} is essential for the arguments that will follow. A lot of the arguments regarding loop graphs can be reduced to the case where the second tile of the graph $G$ is on the right of the first tile, since as we can see by the above Lemma, if the second tile was on top of the first one, we could just consider the isomorphic planar representation of $G$. Of course symmetric arguments can be applied when the loop is at the end of the graph.
\end{remark}

\section{Quiver representations and loop modules}\label{Quiver representations and loop modules}

In this chapter we introduce the notion of \emph{loop modules} which generalize the now classical notion of \emph{string modules}. These are modules over a path algebra that correspond to loopstrings, in the same way that string modules correspond to abstract strings.\\
We begin by reviewing some basic background on quiver representations and their morphisms. Quivers and their representations are extremely useful tools for studying the structure of a group or algebra through combinatorial objects such as graphs and easily understood linear maps. We then discuss the definition of a path algebra and by extension, that of a string algebra. The importance of string modules is evident from the fact that the finite-dimensional modules over a string algebra are precisely the finite-dimensional string and band modules.

Let us formally define what a quiver is. A \emph{quiver} $\mathcal{Q}=(\mathcal{Q}_0, \mathcal{Q}_1, s, t)$ consists of a set of vertices $\mathcal{Q}_0$, a set of arrows $\mathcal{Q}_1$, a map $s\colon\mathcal{Q}_1\to \mathcal{Q}_0$ assigning to its each arrow its source, and a map $t\colon\mathcal{Q}_1\to \mathcal{Q}_0$ assigning to each arrow its target. Let $k$ be an algebraically closed field. \\
\indent A \emph{representation} $M= (M_i,\phi_\alpha)_{i\in\mathcal{Q}_0,\alpha\in\mathcal{Q}_1}$ of a quiver $\mathcal{Q}$ is a collection of $k$-vector spaces $M_i$ together with a collection of $k$-linear maps $\phi_\alpha\colon M_{s(\alpha)}\to M_{t(\alpha)}$. \\
Quivers are mathematically convenient objects to work with, as they can be regarded simply as directed graphs. A quiver representation then assigns a vector space to each vertex and a linear map to each arrow.

\begin{example}\label{examples of representations}
    Let $\mathcal{Q}$ be the quiver $1\leftarrow 2\rightarrow 3$. Then, the following are representations of $\mathcal{Q}$:
    \begin{align*}
    &M_1: k\xleftarrow {\begin{pmatrix}
                       1  & 0 & 0
                      \end{pmatrix} } 
        k^3 \xrightarrow{ \begin{pmatrix}
                       1 & 1 & 0\\
                       0 & 1 & 1\\
                       0 & 0 & 1
                       \end{pmatrix}} 
        k^3\\
    &M_2: k\xleftarrow {\begin{pmatrix}
                       1 & 0
                      \end{pmatrix} } 
        k^2 \xrightarrow{ \begin{pmatrix}
                       1 &  0\\
                       0 &  1\\
                       0 &  1
                       \end{pmatrix}} 
        k^3\\
    &M_3: k\xleftarrow {\begin{pmatrix}
                       1 & 0 
                      \end{pmatrix} } 
        k^2 \xrightarrow{ \begin{pmatrix}
                       1 & 1 \\
                       0 & 1 \\
                       0 & 0 
                       \end{pmatrix}} 
        k^3\\
    &M_4: k\xleftarrow {\begin{pmatrix}
                       1 & 0
                      \end{pmatrix} } 
        k^2 \xrightarrow{ \begin{pmatrix}
                       1 & 1\\
                       0 & 1
                       \end{pmatrix}} 
        k^2\\
    &M_5: k\xleftarrow{\mathmakebox[0.7cm][c]{1}} 
        k \xrightarrow{ \begin{pmatrix}
                       1 \\
                       0
                       \end{pmatrix}} 
        k^2\\
    \end{align*}
\end{example}

If $\mathcal{Q}$ is a quiver and $M=(M_i,\phi_\alpha)$ and $M'=(M_i',\phi_\alpha')$ are two representations of $\mathcal{Q}$, then a \emph{morphism of representations} $f\colon M\to M'$ is a collection of $k$-linear maps $(f_i)_{i\in \mathcal{Q_0}}$, such that for each arrow $\alpha: i\to j$ we have $\phi_\alpha'\circ f_i=f_j\circ\phi_\alpha$, i.e. the linear maps must be compatible with the maps $\phi_\alpha$.

\begin{example}\label{morphism/injections of representations} 
    If $\mathcal{Q}: 1\leftarrow 2\rightarrow 3$ and $M_1, M_2$ and $M_3$  are the representations appearing in Example~\ref{examples of representations} then the following diagrams commute:
    \[
    \begin{tikzcd}[ampersand replacement=\&, column sep = 7em, row sep=5em]
        M_1 \& k 
        \& k^3 \arrow[l, "{\begin{pmatrix} 1 & 0 & 0 \end{pmatrix}}"'] \arrow[r, "{\begin{pmatrix} 1 & 1 & 0 \\ 0 & 1 & 1 \\ 0 & 0 & 1 \end{pmatrix}}"] 
        \& k^3\\
        M_2 \arrow[u, "f"] \& k \arrow[u,bend left, "1"'] \arrow[ur, phantom, "\circlearrowright"]
        \& k^2 \arrow[l, "{\begin{pmatrix} 1 & 0  \end{pmatrix}}"'] \arrow[u,bend left, "{\begin{pmatrix} 1 & 0 \\ 0 & 0 \\ 0 & 1 \end{pmatrix}}"'] \arrow[r, "{\begin{pmatrix} 1 & 0 \\ 0 & 1 \\ 0 & 1 \end{pmatrix}}"'] \arrow[ur, phantom, "\circlearrowright"]
        \& k^3 \arrow[u,bend left, "{\begin{pmatrix} 1 & 0 & 0 \\ 0 & 1 & 0 \\ 0 & 0 & 1 \end{pmatrix}}"']
    \end{tikzcd}
    \]
    \[
    \begin{tikzcd}[ampersand replacement=\&, column sep = 7em, row sep=5em]
        M_1 \& k 
        \& k^3 \arrow[l, "{\begin{pmatrix} 1 & 0 & 0 \end{pmatrix}}"'] \arrow[r, "{\begin{pmatrix} 1 & 1 & 0 \\ 0 & 1 & 1 \\ 0 & 0 & 1 \end{pmatrix}}"] 
        \& k^3\\
        M_3 \arrow[u, "f'"] \& k \arrow[u,bend left, "1"'] \arrow[ur, phantom, "\circlearrowright"]
        \& k^2 \arrow[l, "{\begin{pmatrix} 1 & 0  \end{pmatrix}}"'] \arrow[u,bend left, "{\begin{pmatrix} 1 & 0 \\ 0 & 1 \\ 0 & 0 \end{pmatrix}}"'] \arrow[r, "{\begin{pmatrix} 1 & 1 \\ 0 & 1 \\ 0 & 0 \end{pmatrix}}"'] \arrow[ur, phantom, "\circlearrowright"]
        \& k^3 \arrow[u,bend left, "{\begin{pmatrix} 1 & 0 & 0 \\ 0 & 1 & 0 \\ 0 & 0 & 1 \end{pmatrix}}"']
    \end{tikzcd}
    \]
    Notice additionally that $f_i, f_i', 1\leq i\leq 3$ are injections, so $f\colon M_2\to M_1$ and $f'\colon M_3\to M_1$ are injective morphisms of representations.
\end{example}

\begin{remark}
    If $Q$ is a quiver, one can consider the category of quiver representations $\rep{Q}$. The objects of this category are the representations of $Q$, and the morphisms are morphisms between representations. Composition of morphisms is given by the composition of the corresponding linear maps $\phi_a$. Even though we have not explicitly discussed the notions of \emph{indecomposable representations} and \emph{direct sums of representations}, the equivalence of categories stated in Theorem~\ref{from modules to representations}, illustrates the deep connection between modules and representations. In practice we will use these notions interchangeably.
\end{remark}

\begin{remark}\label{canonical subrepresentattions}
  If $N$ is a \emph{subrepresentation} of $M$, then we refer to the \textit{canonical embedding} of $N$ into $M$ as the injective map $f\colon N\hookrightarrow M$ where each component $f_i\colon N_i\to M_i$ is the linear map induced by the identity on the non-zero components of $N$. The injective maps $f$ and $f'$ from Example~\ref{morphism/injections of representations} are both canonical embeddings of the representations $M_2$ and $M_3$ into the representation $M_1$.
\end{remark}

Having introduced the category of representations $\rep{Q}$, this is a good moment to revisit one of the first natural questions: ``why do we care about representations of quivers?". The answer lies in two fundamental results (Theorems~\ref{from algebras to path algebras} and \ref{from modules to representations}), which together show that the study of finite-dimensional algebras can, in many cases, be reduced to the study of representations of \emph{bound quivers}.\\

Let $Q$ be a quiver. One can associate to it the so-called \emph{path algebra} $kQ$. This is a $k$-algebra whose basis consists of all finite \emph{paths} in $Q$-that is, finite sequences of  consecutive arrows from one vertex to another. The multiplication of two basis elements $c_1\cdot c_2$ is defined as the \emph{concatenation} of the paths if the endpoint of $c_1$ coincides with the starting point of $c_2$; otherwise, the product is defined to be zero. This multiplication is extends linearly to all of $kQ$, making $kQ$ a (non-commutative) $k$-algebra. In particularly, it can be viewed as a graded $k$-algebra, with grading given by path length.

\begin{remark}
Let $Q=(Q_0, Q_1)$ be a quiver, and let $c$ be a path in $Q$. We can canonically associate a string $w_c$, by assigning an arrow in the string $w_c$ for each arrow in the path $c$, as illustrated in Example~\ref{strings on quivers}. This correspondence provides a convenient alternative representation of strings associated with quivers and will be used extensively throughout the remainder of this thesis.
\end{remark}

\begin{example}\label{strings on quivers}
    Let $Q$ be the quiver $1\xleftarrow{\alpha} 2 \xrightarrow{\beta} 3$, and consider the path 
    \[
    c=(1|\alpha, \beta, \beta, \beta, \beta, \beta|3).
    \]
    We associate to this path the string $w_c$ which is given by:
    \[
    w_c= 1\leftarrow2\rightarrow3\leftarrow2\rightarrow3\leftarrow2\rightarrow3.
    \]
    This presentation of the string $w_c$, differs from the one given in Def~\ref{loopstring}, but its properties remain the same.  
\end{example}

\indent If $kQ$ is the path algebra of a quiver $Q$, it is not necessarily finite-dimensional algebra. In particular, if $Q$ contains oriented cycles, then $kQ$ is infinite-dimensional. Therefore, to study finite dimensional algebras, we must consider \emph{quotients } of path algebras by suitable ideals. One such suitable ideal is the so-called \emph{admissible ideal}. This is, roughly speaking, an \emph{arrow ideal}-that is, a two-sided ideal generated by paths of length at least two. If $I$ is an admissible ideal of the path algebra $kQ$, then the pair $(Q,I)$ is called a \emph{bound quiver}, and the quotient algebra $kQ/I$ is called a \emph{bound quiver algebra}.

Having introduced the notion of a bound quiver algebra, we are now in a position to address a very natural question that arises when studying quiver representations: ``Why do we study them in the first place?"

\begin{theorem}\label{from algebras to path algebras}
    If $\mathcal{A}$ is a basic finite dimensional $k$-algebra, then there exists a quiver $Q$ and an admissible ideal $I$ such that $\mathcal{A}\cong kQ/I$.
\end{theorem}

\begin{remark}
    The assumption in Theorem~\ref{from algebras to path algebras} that $\mathcal{A}$ is a basic algebra is not restrictive. Indeed, the module category of any finite dimensional algebra is equivalent to the module category of a basic finite-dimensional algebra. Therefore, the study of basic algebras is sufficient, from a representation theoretic perspective point of view.
\end{remark}

The result of Theorem~\ref{from algebras to path algebras} is fundamental, as it reduces the study of basic finite-dimensional algebras to the study of the bound quiver algebras. The next fundamental results builds upon this by further reducing the study of finitely generated right modules to the study of the representations of a quiver.

\begin{theorem}\label{from modules to representations}
    Let $\mathcal{A}=kQ/I$ be a bound quiver algebra where Q is a finite connected quiver. Then, the following equivalence of categories if true:
    \[
    \Mod{kQ} \cong \rep{(Q,I)}.
    \]
\end{theorem}

\subsection{String modules}

We have already introduced abstract strings and loopstrings in Der~\ref{loopstring}. However, another way to view these objects is as modules over an algebra. In the remainder of this chapter, we will describe the classical construction of a \emph{string module} from an abstract string, and subsequently, the construction of what we will call a \emph{loopstring module} from a given abstract loopstring. This correspondence between modules and (loop)strings is important, as it allows us to use these notions interchangeably in the remainder of the thesis.\\

Let $kQ/I$ be a bound quiver algebra and let $c$ be a path in the quiver. Consider the string 
\[
w_c=a_0 u_1a_1\dots a_{n-1}u_na_n,
\]
where $a_i\in{Q_0}$ for every $0\leq i\leq n$, and $u_j\in\{\leftarrow
,\rightarrow\}$ for every $1\leq j\leq n$. Define the index set $\mathcal{I}_a:= \{i| a_i=a\}$. Then the \emph{string module} $M(w_c)$ corresponding to the string $w_c$ is defined as follows:
\begin{itemize}
    \item At each vertex $x\in Q_0$ we assign the vector space $M_x=\bigoplus_{i\in\mathcal{I}_x}K_i$ where $K_i=k$ for every $i\in \mathcal{I}_x$.
    \item For each arrow $a\in Q_1$, where $x\xrightarrow{\alpha}y$, we assign a linear map $M_a\colon \bigoplus_{j\in\mathcal{I}_x}K_j \to \bigoplus_{i\in\mathcal{I}_y}K_i$ given by the matrix:
    \[
    (M_a)_{i,j} = \begin{cases}
                 Id_k & \text{if $|i-j|=1$},\\
                 0 & \text{otherwise}.\\
             \end{cases}
    \]
\end{itemize}

\begin{example}\label{Example of string module associated to a string}
    Let $Q= 1\xleftarrow{\alpha} 2 \xrightarrow{\beta} 3$ and $w_c=1\leftarrow2\rightarrow3\leftarrow2\rightarrow3\leftarrow2\rightarrow3$ as in Example~\ref{strings on quivers}. We will now construct the corresponding string module $M(w_c)$.\\
    First we compute the index sets:
    \[
    \mathcal{I}_1=\{0\},  \quad \mathcal{I}_2=\{1,3,5\},  \quad  \mathcal{I}_3=\{2,4,6\},
    \]
    which indicate the positions in the string $w_c$ where each vertex of the quiver $Q$ appears. Accordingly, the vector spaces assigned to each vertex are: 
    \[
    M_1= k, \quad M_2=k^3, \quad M_3=k^3.
    \]
    Finally, we describe the linear maps associated to the arrows $\alpha$ and $\beta$. These are given by the matrices:
    \[
    M_\alpha= \begin{pmatrix}
                       Id_k  & 0 & 0
                      \end{pmatrix}, \quad M_\beta = \begin{pmatrix}
                       Id_k & Id_k & 0\\
                       0 & Id_k & Id_k\\
                       0 & 0 & Id_k
                       \end{pmatrix}.\\
    \]
    Therefore the string module $M(w_c)$ associated to the string $w_c$ is as follows:
    \[
    M(w_c): k\xleftarrow {\begin{pmatrix}
                       1  & 0 & 0
                      \end{pmatrix} } 
        k^3 \xrightarrow{ \begin{pmatrix}
                       1 & 1 & 0\\
                       0 & 1 & 1\\
                       0 & 0 & 1
                       \end{pmatrix}} 
        k^3.
    \]
    We can also depict the string module $M(w_c)$ in the following ways:
    \[
    M(w_c)=
    \begin{tikzcd} 
        & k \arrow[dl, "\alpha", "1_k"'] \arrow[dr, "\beta", "1_k"']& & k \arrow[dl, "\beta", "1_k"'] \arrow[dr, "\beta", "1_k"'] & & k \arrow[dl, "\beta", "1_k"'] \arrow[dr, "\beta", "1_k"']\\
        k & & k & & k & & k
    \end{tikzcd},
    \]
    \[
    M(w_c)=
    \begin{tikzcd} 
        & 2 \arrow[dl] \arrow[dr]& & 2 \arrow[dl] \arrow[dr] & & 2 \arrow[dl] \arrow[dr]\\
        1 & & 3 & & 3 & & 3
    \end{tikzcd},
    \]
   where the second presentation is called the \emph{string presentation} of the module $M(w_c)$.
\end{example}

\begin{remark}
    Notice that the module $M(w_c)$ appearing in Example~\ref{Example of string module associated to a string} is the same as the module $M_1$ in Example~\ref{examples of representations}. The string presentation of the modules $M_2$ and $M_3$ of Example~\ref{examples of representations} is the following:
    \[
    M_2=
    \begin{tikzcd} 
        & 2 \arrow[dl] \arrow[dr]& \\
        1 & & 3 
    \end{tikzcd}
    \oplus
    \begin{tikzcd} 
        & 2 \arrow[dl] \arrow[dr]& \\
        3 & & 3 
    \end{tikzcd},
    \]
    \[
    M_3=
    \begin{tikzcd} 
        & 2 \arrow[dl] \arrow[dr]& & 2 \arrow[dl] \arrow[dr] & & \\
        1 & & 3 & & 3
    \end{tikzcd}
    \oplus 3.
    \]
    It is clear that $M_2$ and $M_3$ are maximal submodules of $M_1$, as in each case we remove a copy of the top vertex $2$ from $M_1$. However, the resulting submodules are not isomorphic, since the top vertex is removed from ``different positions" in the structure of $M_1$. This is also reflected in the fact that the inclusion maps from $M_2$ and $M_3$ to $M_1$ are distinct, as can be seen in Example~\ref{morphism/injections of representations}.
    In this thesis, the idea of removing a top from a module to produce a submodule will appear frequently. For brevity, whenever we refer to a \emph{canonical submodule} $N$, we will  implicitly mean the pair $(N,f)$, where $f\colon N\hookrightarrow M$ is the associated canonical embedding, as described in Remark~\ref{canonical subrepresentattions}. 
\end{remark}

\begin{remark}
    String modules are of particular interest because they appear in the classification of indecomposable representations of certain algebras. Specifically, these are the so-called \emph{finite-dimensional string algebras} for which the indecomposable modules are completely classified by string and band modules. In this thesis, we will not explicitly construct the band module associated to bands  as their construction is analogous to that of string modules described in this chapter and can be found in detail in the relevant literature (\cite{Butler1987AuslanderreitenSW}, \cite{Laking2016}).
\end{remark}

\subsection{Loop modules}

In this chapter we associate a loop module to any given loopstring. This association is particularly important, as it allows us to interchangeably work with loopstrings and their corresponding modules, especially when studying the submodule structure of a loop module. The construction closely parallels that of string modules, but is expanded by adding a non-zero values to the matrix entries corresponding to loops. Since our focus will be on loopstrings associated to tagged arcs on a punctured surface, we will restrict our attention to such cases.\\

Let $(S,M,P,T)$ be a punctured surface and $Q_T$ the quiver associated to the triangulation $T$. Let $I$ be an admissible ideal of the quiver $Q_T$ and define the bound quiver algebra $kQ_T/I$. Let $\gamma^\bowtie$ be a doubly notched arc on the surface and consider $c_{\gamma^\bowtie}$ to be the path on the quiver $Q_T$ which follows the intersection points of $\gamma^\bowtie$ with the triangulation $T$. Consider the loop string
\[
w:=w_{c_{\gamma^\bowtie}}= a_0u_1\dots a_{k-1}u_ka_{k+1}\dots a_{l-1}u_la_{l+1}\dots a_{n-1}u_na_n,
\]
where $a_i\in Q_T$ for every $0\leq i\leq n$, $u_j\in\{\leftarrow,\rightarrow\}$ for every $1\leq j\leq n$ with $j\neq k,l$ and $u_k, u_l\in\{ \hookleftarrow,\hookrightarrow\}$. Define the index set $\mathcal{I}_a:= \{i| a_i=a\}$. Then the \emph{loop module} $M(w)$ corresponding to the string $w$ is defined as follows:
\begin{itemize}
    \item At each vertex $x\in Q_0$ we assign the vector space $M_x=\bigoplus_{i\in\mathcal{I}_x}K_i$ where $K_i=k$ for every $i\in \mathcal{I}_x$.
    \item For each arrow $a\in Q_1$, where $x\xrightarrow{\alpha}y$ we assign a linear map $M_a\colon \bigoplus_{j\in\mathcal{I}_x}K_j \to \bigoplus_{i\in\mathcal{I}_y}K_i$ given by the matrix:
    \[
    (M_a)_{i,j} = \begin{cases}
                 Id_k & \text{if $|i-j|=1$},\\
                 Id_k & \text{if $i=0$ and $j=k$, or $i=k$ and $j=0$},\\
                 Id_k & \text{if $i=l-1$ and $j=n$, or $i=n$ and $j=l-1$},\\
                 0 & \text{otherwise}.\\
             \end{cases}
    \]
\end{itemize}

We would like now to clarify two points regarding the construction of the loop module.
First, since there is a hook at positions $k$ and $l$ of the loopstring $w$ we can deduce that there must exist an edge connecting the vertices $a_0$ and $a_k$, as well as an edge connecting $a_{l-1}$ and $a_n$ in the quiver $Q_T$. This ensures the existence of two matrices, denoted $M_\alpha$ and $M_\beta$, which are assigned to these edges. The existence of these matrices implies that exactly one of the two conditions  $i=0$ and $j=k$, or $i=k$ and $j=0$  (resp. $i=l-1$ and $j=n$, or $i=n$ and $j=l-1$) must hold. This gives rise to exactly two non-zero entries in the matrices, as intended.\\
Additionally, although the construction was given for the case where the loopstring has two loops, it naturally reduced to the case with a single loop.\\

\begin{example}\label{loopstring presentation}
    Let $(S,M,P,T)$ be the triangulated surface, and $\gamma^\bowtie$ the tagged arc, appearing in Figure~\ref{fig:tagged arc}. We can then associate the following equivalent loopstrings to $\gamma^\bowtie$:
    \[
    w_{\gamma^\bowtie} = 1\leftarrow 2 \hookrightarrow 3 \leftarrow 4 \leftarrow 5,
    \]
    \[
    w'_{\gamma^\bowtie} = 1\leftarrow 2 \hookleftarrow 3 \rightarrow4 \rightarrow 5.
    \]
    The loop modules $M(w_{\gamma^\bowtie})$ and $M(w'_{\gamma^\bowtie})$ associated to the loopstrings $w_{\gamma^\bowtie}$ and $w'_{\gamma^\bowtie}$ are isomorphic and equal to the following:
    \[
    M(w_{\gamma^\bowtie}) = M(w'_{\gamma^\bowtie}) =
    \begin{tikzcd}
    & & k \arrow[ld,"1"] \arrow[rd,"1"]& \\
    & k \arrow[ld,"1"] \arrow[rd,"1"] & & k \arrow[ld,"1"]\\
    k & & k & 
    \end{tikzcd}.
    \]
    A different way of depicting the module $M(w_\gamma^\bowtie)$ which resembles closely the loopstring structure is the following:
    \[
    M(w_{\gamma^\bowtie}) =
    \begin{tikzcd}
    & & & & 5 \arrow[ld] \arrow[llld, bend right]\\
    & 2 \arrow[ld] \arrow[rd] & & 4 \arrow[ld]\\
    1 & & 3 & 
    \end{tikzcd},
    \]
    while the equivalent way of representing $M(w'_{\gamma^\bowtie})$ is:
    \[
    M(w'_{\gamma^\bowtie}) =
    \begin{tikzcd}
    & & 5 \arrow[ld] \arrow[rd] & &\\
    & 2 \arrow[ld] \arrow[rrrd, bend right] & & 4 \arrow[rd] & \\
    1 & & & & 3
    \end{tikzcd}.
    \]    
    We will refer to the above presentations as the \emph{loopstring presentation} of the module. In practice, this will be the primary way we visualize loopstrings or loop modules from now on, as this depiction makes it easier to identify maximal submodules. By locating a top and removing it we can directly read off submodule structures.
\end{example}

In Example~\ref{loopstring presentation}, notice that equivalent loopstrings give rise to isomorphic modules. The natural question that arise then is the following: ``Is the construction of the loop modules compatible with the different loopstrings that may be associated to the same tagged arc $\gamma^\bowtie$?". We should expect to associate the same loop module $M(w_{c_{\gamma^\bowtie}})$ to a given tagged arc, as seen in Example~\ref{loopstring presentation}, and this is indeed the case, as the next proposition indicates. 

\begin{proposition}
    Let $\gamma^\bowtie$ a tagged arc on the triangulated surface $(S,M,P,T)$ and $w_{c_{\gamma^\bowtie}}$ and $w'_{c_{\gamma^\bowtie}}$ two loopstrings associated to $\gamma^\bowtie$. Then the modules $M(w_{c_{\gamma^\bowtie}})$ and $M(w'_{c_{\gamma^\bowtie}})$ are the same.
    \begin{proof}
        Let us assume that $\gamma^\bowtie$ has only one tagging. Let 
        \[
        w_{c_{\gamma^\bowtie}}=a_0u_1\dots a_{k-1}u_ka_ku_{k+1}\dots u_na_n,
        \]
        \[
        w'_{c_{\gamma^\bowtie}}=a_{k-1}u_1'\dots a_0u_ka_ku_{k+1}\dots u_na_n,
        \]
        be the two loopstrings. Notice that the sequence of vertices $(a_0,\dots,a_{k-1})$ in the loopstring $w_{c_{\gamma^\bowtie}}$, appears in the opposite order in the loopstring $w'_{c_{\gamma^\bowtie}}$.\\
        By the construction of the loop modules we have non-zero entries in the matrices when there are consecutive vertices $a_i, a_{i+1}$ (i.e. when $|i-j|=1$). Additionally, since there is only one loop starting from the left, we have an additional non-zero entry (when $i=0$ and $j=k$, or $i=k$ and $j=0$).\\
        Therefore, the consecutive vertices $a_i,a_{i+1}, 0\leq i\leq n-1$ in the string $w_{c_{\gamma^\bowtie}}$ give rise to a non-zero entry. If $k\leq i\leq n-1$, then the vertices $a_i,a_{i+1}$ appear in that order in the loopstring $w'_{c_{\gamma^\bowtie}}$, and so they give rise to a non-zero entry in the equivalent matrix. If $0\leq i\leq k-2$, then the vertices $a_i,a_{i+1}$ appear in the opposite order in the loopstring $w'_{c_{\gamma^\bowtie}}$, and so they still give rise to a non-zero entry in the equivalent matrix.\\
        Lastly, let us assume that $i=k-1$. Looking at the string $w'_{c_{\gamma^\bowtie}}$ we can notice that the vertex $a_{k-1}$ is the first vertex and the vertex $a_k$ is the first vertex after the hook $u_k$. We can therefore rewrite:
        \[
        w'_{c_{\gamma^\bowtie}}=b_0u_1'\dots b_{k-1}u_kb_ku_{k+1}\dots u_nb_n,
        \]
        where $b_0=a_{k-1}$.. Notice that then, we must have a non-zero entry in the corresponding matrix, (following the construction of the matrix $(M_\alpha)_{i,j}$,) since in this case $i=0$ and $j=k$.\\
        Therefore, the corresponding matrices assigned to the arrows in the quiver representations $M(w_{c_{\gamma^\bowtie}})$ and $M(w'_{c_{\gamma^\bowtie}})$ must be the same, and subsequently the modules are isomorphic.
    \end{proof}
\end{proposition}

%{\color{red} The next remark must be probably rewritten.}
\begin{remark}
    In the case of the strings algebras, we mentioned in the previous chapter that the string and band modules give a complete classification of the indecomposable modules of those algebras. \\
    It is not too difficult to show that loop modules are indecomposable modules, by combining the fact that string modules are indecomposable and exploring what happens with the extra non empty entries to the suitable matrices. Therefore, a natural question, would be to ask if the loop modules, together with some other possible modules (e.g. string and band modules), could fully classify the indecomposable representations of some algebras. However, when working with path algebras, we usually require the representations to satisfy the conditions of the admissible ideal $I$, which not necessarily happens in our definition of loop modules.
\end{remark}

\chapter{Bijection of loop graphs and loopstrings}

In this section we aim to prove that there is a bijection between the lattice of a loop graph and he submodule lattice of the associated loopstring. At first we will expand on some results on perfect matchings and generalize some results from \cite{canakcci2021lattice} which will be needed for the final proof of Theorem~\ref{submodule lattice bijective to perfect matchings lattice}.

\section{Minimal and maximal perfect matchings}

In this section we will define two special perfect matchings of a loop graph $Q_{\mathcal{G}^\bowtie}$ that will be of a great importance later. Our definition is dual to the one given by Wilson \cite{wilson2020surface}. However, we will also prove in Lemma~\ref{eqofpmatch} that if the given graph $G$ has at least two tiles, that definition is equivalent to the one given by Çanakçı-Schroll \cite{canakcci2021lattice}. %(I have not given the exact definition yet)

\begin{remark}
  Based on the definition of maximal and minimal perfect matchings of a loop graph $\mathcal{G}^\bowtie=(G_1,\dots,G_n)^\bowtie$ we can notice the following:
  \begin{itemize}
      \item if $G_2$ is on the right of $G_1$ then $W(G_1)$ (resp. $N(G_1)$ or $S(G_1)$) belong to the minimal (resp. maximal) perfect matching, if they are boundary edges,
      \item if $G_2$ is on top of $G_1$ then $W(G_1)$ or $S(G_1)$ (resp. $S(G_1)$) belong to the minimal (resp. maximal) perfect matching, if they are boundary edges.
  \end{itemize}
\end{remark}

We will denote the minimum perfect matching by $P_{min}$ and the maximal perfect matching by $P_{max}$.\\
Also if $\mathcal{G}$ is a snake or loop graph we will denote by $E^{bdry}(\mathcal{G})$ the collection of boundary edges of $\mathcal{G}$. We also note that if $\mathcal{G}$ is a loop graph then the edges that are identified with each other are not boundary.\\
If $P$ and $P'$ are perfect matchings of a graph $\mathcal{G}$, then we will denote by $P\ominus P'$ their \textit{symmetric difference}, namely the collection of all edges that belong either only to $P$ or only to $P'$.

\begin{remark}
    Suppose that $G$ is a snake graph and that  $H = \{G_1,\dots,G_k\}$ is a subgraph of $G$ where the tiles $G_1,\dots,G_k$ form a maximal zig-zag piece. Suppose $P = P_{max}$ or $P = P_{min}$. Then the following are true:
    \begin{itemize}
        \item each tile $G_2,\dots, G_{k-1}$ has exactly one of its edges in $P_{|H}$,
        \item if the tile $G_1$ has two of its edges in $P_{|H}$, then none of the edges of $G_k$ is in $P_{|H}$,
        \item if the tile $G_1$ has none of its edges in $P_{|H}$, then two of the edges of $G_k$ are in $P_{|H}$.
    \end{itemize}
\end{remark}

\begin{lemma}\label{pmatchinloop}
Suppose that $\mathcal{G}^\bowtie=(G_1,\dots, G_k, \dots,G_n)^\bowtie$ is a loop graph in which edges of the tiles $G_1$ and $G_k$ are identified, $G_2$ is on the right of $G_1$ and $P$ is a perfect matching of $G^\bowtie$. Then $E^{bdry}(G_j)\in P\ominus P_{min}$, $1\leq j < k-1$,  if and only if $E^{bdry}(G_i)\in P\ominus P_{min}$ for all $j<i\leq k-1$.
    \begin{proof}
    Suppose that $E^{bdry}(G_j)\in P\ominus P_{min}$ for some $1\leq j < k-1$. We will first show that $E^{bdry}(G_{k-1})\in P\ominus P_{min}$. By the definition of $P_{min}$, no edges of $G_1$ are in $P_{min}$. Therefore $G_{k-1}$, which is the last tile of the zig-zag, must have two of its edges in $P_{min}$ and all the other tiles $G_2,\dots, G_{k-2}$ have exactly one of their edges in $P_{min}$.\\
    Notice that each perfect matching is connected to the minimal perfect matching by a series of flips that satisfy the twist parity condition. Therefore, since $E^{bdry}(G_j)\in P\ominus P_{min}$, it follows that we must first flip the edges of the tile $G_{k-1}$. Thus, $E^{bdry}(G_{k-1})\in P\ominus P_{min}$.\\
    Notice, that after flipping the edges of $G_{k-1}$ we can only flip the edges of $G_{k-2}$, since it is the only other tile with two of its edges in $P$. Continuing inductively, the second part of the Lemma follows.      
    \end{proof}
\end{lemma}

\begin{corollary}\label{welldefinedconloopstring}
Suppose that $\mathcal{G}^\bowtie=(G_1,\dots, G_k, \dots,G_n)^\bowtie$ is a loop graph in which edges of the tiles $G_1$ and $G_k$ are identified, $G_2$ is on the right of $G_1$ and $P$ is a perfect matching of $G^\bowtie$. If $H_i = \{T_{i_1},\dots,T_{i_n}\}$ is a connected subgraph of $\mathcal{G}^\bowtie$, then the corresponding loopstring $w_i$ of $H_i$ is equivalent to a "connected" loopstring $w_i'= a_{j_1},\dots,a_{j_n}$.

\begin{proof}
The proof is a direct consequence of Lemma \ref{eqloopstrings} and Lemma \ref{pmatchinloop}.    
\end{proof}

\end{corollary}

\section{Perfect matching lattices and submodule lattices}

Our aim in this section is to make the bijection in Remark 7.11 \cite{wilson2020surface} explicit.\\

%\subsection{}

In this section, when referring to a loop graph $\mathcal{G}^\bowtie=(G_1,\dots, G_k, \dots,G_n)^\bowtie$ we will mean a loop graph with a loop at the beginning of the graph in which edges of  the tiles $G_1$ and $G_k$ are identified. The results in this section can be naturally generalised in the dual case where we have one loop at the end of the graph or in the case that there are two loops, one at the beginning and one at the end of the graph.

%\begin{lemma}
%Suppose $\mathcal{G}^\bowtie$ is a loop graph and $P$ is a perfect matching of this graph. Then:
%\begin{itemize}
 %   \item[(i)] $P\ominus P_{min}$ cannot be equal to a strict subset of $\{G_1, \dots, G_k\}$
%\end{itemize}

%\begin{proof}

%\begin{itemize}
 %   \item[(i)] W.l.o.g. assume that $G_2$ is on the right of $G_1$
%\end{itemize}
  
%\end{proof}
        
%\end{lemma}

\begin{lemma}\label{eqofpmatchsnake}
Suppose that $G=\{G_1,\dots,G_n\}$ is a snake graph without loops. Then, the tile $G_i$ of $G$ has two of its boundary edges in $P_{min}$ if and only if it corresponds to a socle.

\begin{proof}
    We will use induction on the number of tiles of the graph.
    \begin{itemize}
        \item \textit{base case}\\
        If the graph $G$ has only one tile, then obviously that tile corresponds to a socle and it has two of its edges in $P_{min}$.
        \item \textit{induction step}\\
        Suppose now, that our assumption is true for every graph with $k$ tiles. We will prove that this is also true for a graph $G=\{G_1,\dots, G_k,G_{k+1}\}$ with $k+1$ tiles. \\
        There are two cases. Either the tile $G_{k+1}$ is on the right of $G_k$ or it is on top of it. We will deal with the first case, since the other is completely symmetrical.\\
        \begin{enumerate}
            \item  Assume that $E(G_{k+1})$ is in $P_{min}$. We will prove then, that $G_{k+1}$ does not correspond to a socle. We will consider cases on the configuration of the tile $G_{k-1},G_k$ and $G_{k+1}$.\\
            If these three tiles form a straight piece, then we can deduce that $N(G_k)$ and $S(G_k)$ are also in $P_{min}$. Therefore, when we consider the subgraph $G'=\{G_1,\dots, G_k\}$ and the induced $P_{min}'$ on that graph, using our induction hypothesis, we obtain that $G_k$ correspond to a socle. Subsequently, $G_{k+1}$ cannot correspond to a socle.\\
            If the tiles $G_{k-1},G_k$ and $G_{k+1}$ form a zig-zag, suppose that $G_l$ is the fist tile of the maximal zig-zag piece. Since $E(G_{k+1})$ is in $P_{min}$, we deduce that $N(G_k)$ and $S(G_k)$ are also in $P_{min}$. Therefore, when we consider the subgraph $G'=\{G_1,\dots, G_k\}$ and the induced $P_{min}'$ on that graph, using our induction hypothesis, we obtain that $G_l$ correspond to a socle. Subsequently, there is a local configuration of arrows \\ $l\leftarrow \dots \leftarrow k\leftarrow k+1$, which means that $G_{k+1}$ does not correspond to a socle.
            \item Assume now that $N(G_{k+1})$ and $S(G_{k+1})$ are in $P_{min}$. We will prove then, that $G_{k+1}$ corresponds to a socle. Again, we will consider cases on the configuration of the tile $G_{k-1},G_k$ and $G_{k+1}$.\\
            If these three tiles form a straight piece, then we can deduce that none of the boundary edges of $G_k$ are in $P_{min}$. Therefore, when we consider the subgraph $G'=\{G_1,\dots, G_k\}$ and the induced $P_{min}'$ on that graph, which contains only the east boundary edge of $G_k$, using our induction hypothesis, we obtain that $G_k$ does not correspond to a socle. Subsequently, $G_{k+1}$ must correspond to a socle.\\
            If the tiles $G_{k-1},G_k$ and $G_{k+1}$ form a zig-zag, then we can deduce that the induced $P_{min}'$ of the subgraph $G'=\{G_1,\dots, G_k\}$, contains $W(G_k)$ and $E(G_k)$. Therefore, using our induction hypothesis on $G'$,  we obtain that $G_k$ corresponds to a socle of the quiver which is generated by the subgraph $G'$. Subsequently, there is a local configuration of arrows $k-1 \rightarrow k$. Since, the tiles $G_{k-1},G_k$ and $G_{k+1}$ form a zig-zag this local configurations of arrows extends to the following configuration: $k-1\rightarrow k\rightarrow k+1$. This means that $G_{k+1}$ corresponds to a socle, which completes the proof.
        \end{enumerate}
    \end{itemize}
\end{proof}

\end{lemma}

\begin{remark}
The previous Lemma \ref{eqofpmatchsnake}, shows that our definition of minimal and maximal perfect matchings is equivalent to the definition of Çanakçı-Schroll. This also builds a bridge between the definition of Wilson's minimal and maximal perfect matchings, since it shows that a tile $G_i$ of the snake graph has two of its edges in $P_{max}$ if and only if it corresponds to a socle. 
\end{remark}

The dual of the previous Lemma is also true. The only difference is that we need to exclude the first and the last tile, since when they correspond to a top, one of their boundary edges is part of $P_{min}$. We state it explicitly in the next Lemma, omitting the proof.

\begin{lemma}\label{dualofsocle}
 Suppose that $G=\{G_1,\dots,G_n\}$ is a snake graph. Then, the tile $G_i$, for $1<i<n$, of the snake graph has none of its boundary edges in $P_{min}$ if and only if it corresponds to a top.   
\end{lemma}

We would like to have a similar result also for the case of loop graphs, since this is the main interest of this report. It turns out that the same statement is also true for loop graphs, and the next Lemma is going to be one of our main tools for the proofs that will follow.

\begin{lemma}\label{eqofpmatch}
    Suppose that $G^\bowtie$ is a loop graph with at least one loop at the beginning or at the end of the graph. Then, the tile $G_i$ of $G^\bowtie$ has two of its boundary edges in $P_{min}$ if and only if it corresponds to a socle.
\begin{proof}
    Consider the plain snake graph $G$ associated to the loop graph $G^\bowtie$. Then, by definition, the minimal perfect matching $P_{min}$ of $G^\bowtie$ can be extended to the minimal perfect matching $P_{min}'$ of $G$. Suppose that $Q$ is the quiver associated to $G^\bowtie$ and $Q'$ is the quiver associated to $G'$. Observe now the following: \\
    \begin{itemize}
        \item[(i)] By construction, the vertex of $Q$ associated to the last tile of the loop graph, is a socle \textit{if and only if} it is a socle in $Q'$.
        \item[(ii)] If $G_k$, $1<k<n$, is a tile in which at least one of its boundary edges has been identified with another edge of the graph, then, by construction the vertex $k$ of the quiver $Q$ never corresponds to a socle nor a top.
    \end{itemize}
    Continuing now with the proof, suppose that there is a loop at the end of the graph $G^\bowtie$, and let $G_k$ be the tile in which we have identified one of its boundary edges with an edge of the tile $G_n$. Fist of all, it can be easily seen that if $1<=i<n$ and $i\neq k$, the local configuration around the vertex $i$ on the quivers $Q$ and $Q'$ is the same, so the tile corresponds to a socle in $Q$ \textit{if and only if} it corresponds to a socle in $Q'$ and so the result follows from Lemma \ref{eqofpmatchsnake}.\\
    Regarding the tile $G_k$, as noted above in (ii), it cannot correspond to a socle in $Q$. However, we can also notice that it has only one boundary edge, so it can't have two boundary edges in $P_{min}$.\\
    Lastly, regarding the tile $G_n$, as noted above in (i), it corresponds to a socle in $Q$ \textit{if and only if} it corresponds to a socle in $Q'$, and so the result follows by Lemma \ref{eqofpmatchsnake}. This completes the proof, since the case that there is a loop at the beginning of the graph follows in exactly the same way.
\end{proof}
    
\end{lemma}

We will now proceed with the first main step on proving that there is a bijection between the perfect matching lattice of a loop graph and the canonical submodule lattice the associated loop word. For this, we will prove that each perfect matching of a loop graph can be canonically associated to a submodule of the corresponding loop word.

\begin{proposition}\label{Ptomod}
    
Let $A = kQ^\bowtie/ I$, where $I$ is an admissible ideal in $ kQ^\bowtie$. Let $w$ be a loopstring, and $M(w)$ be the associated loop module over $A$. Suppose $P$ is a perfect matching of $\mathcal{G}^\bowtie$. Then $P\ominus P_{min}$ gives rise to a canonical submodule $M(P)$ of $M(w)$.

\begin{proof}
    We have that $P\ominus P_{min} = \bigcup {H_i}$ where each $H_i$ is a connected subgraph of $\mathcal{G}(w)$ and for each $i$ there is a canonical embedding of graphs $\phi_i\colon H_i\to\mathcal{G}(w)$. We will show that each $H_i$ gives rise to a canonical submodule $M(w_i)$ and these in turn will give rise to the canonical submodule $M(P) = \bigoplus M(w_i)$.\\
    Let $T_{i_1}, \dots, T_{i_n}$ be the tiles of $H_i$. Then we can consider the sub-loopstring $w_i = a_{i_1} \dots a_{i_n-1}$ (by Lemma \ref{welldefinedconloopstring}) of the loopstring $w$ which is defined to be the restriction of $w$ to the elements of the set $\{\rightarrow,\leftarrow,\hookleftarrow,\hookrightarrow\}$ which are corresponding to the tiles of $H_i$. 
    First of all, assume that $H_i$ contains only one tile $T_{i_1}$. Using Lemma \ref{eqofpmatch}, we see that then $T_{i_1}$ corresponds to a socle and then obviously $M(w_i)$ is a submodule of $M(w)$. \\
    Suppose, from now on that $H_i$ contains at least two tiles.\\
    
    We know that each subloopstring $w_i$ gives rise to a representation of $A$. 
    Since each $H_i$ is connected, the representation $M(w_i)$ associated to $w_i$ is irreducible. Therefore $M(w_i)$ would not be a subrepresentation, and subsequently a submodule of $M(w)$, if and only if none of the following is true:
    \begin{itemize}
        \item there exists an arrow $i_1\to i_1-1$ in the quiver $Q^\bowtie$,
        \item there exists an arrow $i_n\to i_n+1$ in the quiver $Q^\bowtie$
    \end{itemize}
    We will show that none of the above is true for $M(w_i)$ focusing only on the first case, since the second is  similar.\\
    First, suppose that $T_j = G_j$ for each $1\leq j\leq k$, namely that $H_i$ contains the whole zig-zag of the loop. Then $i_1 = 1$ so there is not a vertex $i_1 - 1$ and subsequently no arrow $i_1\to i_1-1$ in the quiver $Q^\bowtie$.\\
    Suppose now that $H_j$ contains a strict subset of the tiles $\{G_1, \dots, G_k\}$. Using Lemma \ref{pmatchinloop}, we see that $T_{i_1} = G_j$ for some $1< j\leq k-1$ and since the tiles $G_{j-1}$ and $G_j$ are part of the zig-zag, we have that there exists an arrow $j-1\to j$, and so there is no arrow $i_1\to i_1-1$ in the quiver $Q^\bowtie$.\\
    Lastly suppose that $T_{i_1}$ is not part of the loop of the loop graph. Assume that there is an arrow $i_1\to i_1-1$ in the quiver $Q^\bowtie$. We consider the following cases:\\
    Suppose that there is an arrow $i_1\to i_2$. Then $T_{i_1}$ corresponds to a top and using Lemma \ref{dualofsocle} we take that none of the edges of $T_{i_1}$ belong to $P_{min}$. Therefore, due to the twist parity condition, $T_{i_1}\in P\ominus P_{min}$ if and only if we have flipped the edges of both the tiles $T_{i_1-1}$ and $T_{i_2}$. But then we would also have that the tile $T_{i_1-1}$ belongs to $H_i$ which is not true.\\
    Suppose now that there is an arrow $i_2\to i_1$. Then there must exist a zig-zag of tiles \\
    $T_{i_1-1}, T_{i_1}, T_{i_2},\dots T_{i_l}$, $2\leq l\leq n$, where $T_l$ corresponds to a top. The tiles $T_{i_1-1}, T_{i_1}, T_{i_2},\dots T_{{i_l}-1}$ have exactly one of their edges in $P_{min}$. Additionally, since $T_l$ is a top, it has none of its edges in $P_{min}$. Therefore, since we have that $T_l\in P\ominus P_{min}$ we see that due to the twist parity condition, we first need to flip the edges of $T_{{i_l}-1}$ and inductively the edges of $T_{i_1-1}$. However, in that case we would also have that the tile $T_{i_1-1}$ belongs to $H_i$ which is not true.\\

    Lastly, we need to show now that there is a canonical embedding $M(P)\hookrightarrow M(w)$. Each canonical embedding of graphs $\phi_i\colon H_i\to\mathcal{G}(w)$ determines a tile $T_{j_i}$ which is the first tile of $H_i$, when we consider it as a subgraph of $\mathcal{G}(w)$. This tile corresponds to a vertex of the quiver $Q^\bowtie$ and some vertex of $w_i$. Identifying these two vertices gives rise to a canonical embedding $M(w_i)\hookrightarrow M(w)$ which in turn gives rise to a canonical embedding $M(P)\hookrightarrow M(w)$.

\end{proof}

\end{proposition}

Conversely, given a canonical submodule $N = N(w_1)\oplus \dots\oplus N(w_n)$ of a loop module $M(w)$, with canonical embedding $\phi$, we want to associate a perfect matching $P_N$ of the loop graph $\mathcal{G}_w$. Suppose $H_i$ is the loop subgraph of $\mathcal{G}_w$ which corresponds to the embedding $\phi$.
We define:
\[
P^{c}(H_i) = E^{bdry}(H_i)\backslash P_{min}|_{H_i}, 
\]
\[
P_N = (\bigcup{P^c(H_i)}) \bigcup P_{min}|_{\mathcal{G}^\bowtie \backslash \bigcup{H_i}}.
\]

\begin{remark}
    Note that $P^c(H_i) = P_{max}(H_i)$ when the subgraph $H_i$ contains at least two tiles. However, if $H_i$ consists of only one tile then $P^c(H_i)$ is equal to either $P_{max}(H_i)$ or $P_{min}(H_i)$.
\end{remark}

\begin{proposition}\label{modtoP}
The set of edges $P_N$ as defined above is a perfect matching of $\mathcal{G}^\bowtie$ and $N = M(P_N)$.

\begin{proof}
In order to prove that this is indeed a perfect matching, we just have to prove that $P^c(H_i)$ is a perfect matching of the subgraph $H_i$.\\
We will prove this by induction on the number of tiles of $H_i$.

\begin{itemize}

\item \textit{base case:}\\
If $H_i$ contains only one tile then this tile corresponds to a socle and according to Lemma \ref{eqofpmatch}. this tile has two of its edges in $P_{min}$. But then, this means that $P^c(H_i)$ is also a perfect matching of $H_i$.\\

\item \textit{induction step:}\\
Suppose that, for each subgraph $H$ of $\mathcal{G}_w$ which contains less than $k$ tiles, $P^c(H)$ is a perfect matching of $H$. We will prove that this is also true when $H$ is a subgraph of $\mathcal{G}_w$ with $k$ tiles.\\
Suppose that $H_i = \{T_{i_1}, \dots, T_{i_k} \}$ with $i_k>1$. Suppose that the module $N(w_i')$ is a maximal submodule of $N(w_i)$. Then, $N(w_i')$ is either a direct sum of two submodules or is an indecomposable module.. \\
Suppose that $N(w_i')$ is indecomposable. This means that the subgraph $H_i'$, which is the graph corresponding to $N(w_i')$, is connected. Therefore, $H_i' = \{T_{i_1}, \dots, T_{i_{k-1}} \}$ or $H_i' = \{T_{i_{2}}, \dots, T_{i_k} \}$. W.l.o.g. suppose that \\ $H_i' = \{T_{i_1}, \dots, T_{i_{k-1}} \}$ and that $T_{i_k}$ is on the right of $T_{i_{k-1}}$, meaning that $E(T_{i_{k-1}})$ is an interior edge. Using the induction hypothesis, $P^c(H_i')$ is a perfect matching of $H_i'$. Since $E(T_{i_{k-1}})$ is an interior edge, it cannot belong to $P^c(H_i')$. However, in a perfect matching, all edges must be matched, so $N(T_{i_{k-1}})$ and $S(T_{i_{k-1}})$ are in $P_{min}$. Therefore, in order to prove that $P^c(H_i)$ is a perfect matching of $H_i$ we just need to prove that $E(T_{i_k})$ is in $P_{min}$, since then we will also have that all the edges of the tile $T_{i_k}$ are matched. Suppose that  $E(T_{i_k})$ is not in $P_{min}$.
However, this can happen only if there is an additional tile on the right or above of $T_{i_k}$. This means that the tile $T_{i_{k+1}}$ is a tile of $\mathcal{G}_w$ and that the simple corresponding to $T_{i_k}$ is a top of $M(w)$, since none of it's edges belong in $P_{min}$ (this is the dual statement of Lemma  \ref{eqofpmatch}). But then, the maximal submodule of $N(w_i)$ which does not contain the simple corresponding to $T_{i_k}$ would be a direct sum of two submodules, which is a contradiction. \\
Suppose now that  $N(w_i') = N_1 \oplus N_2$, where $N_1$, $N_2$ are submodules of $N_w$. Then we know that $H_1 = \{T_{i_1}, \dots, T_{i_{d-1}}\}$ and $H_2 = \{T_{i_{d+1}}, \dots, T_{i_k}\}$ are subgraphs corresponding to $N_1$ and $N_2$ respectively for some $1 < d < k$. Therefore, the simple module corresponding to $T_d$ is a top of both $N_1$ and $N_2$. So the tiles $T_{d-1}, T_d$ and $T_{d+1}$ form a straight piece and w.l.o.g assume that this is a horizontal piece. Therefore, $E(T_{d-1})$ and $W(T_{d+1})$ are interior edges. Also, by the induction hypothesis we have that $P^c(H_1)$ and $P^c(H_2)$ are perfect matchings. Since $E(T_{d-1})$ and $W(T_{d+1})$ are interior edges, they cannot belong to $P_{min}(H)$ and so by the definition of $P^c$, we have that these  two edges are part of the perfect matchings $P^c(H_1)$ and $P^c(H_2)$ respectively. But then this means that also $P^c(H_i)$ is a perfect matching of $H_i$, since we can flip the tiles of $T_d$ and this will still be a perfect matching.

\end{itemize}

To prove the last statement, by definition of $P_g(H_i)$ we have that:
\[
M(P^c(H_i)  \bigcup P_{min}|_{\mathcal{G}^\bowtie \backslash {H_i}} = N(w_i),
\]
and so we take that $M(P_N) = N$.
    
\end{proof}

\end{proposition}

\section{Bijection between perfect matching lattices and canonical submodule lattices}

We are now ready to prove the main theorem of this section. Namely, we prove the existence of a bijection between the perfect matching lattice and the submodule lattice of the associated module.

\begin{theorem}\label{submodule lattice bijective to perfect matchings lattice}
Let $A=kQ^\bowtie/I$ and $M(w)$ be a loop module over $A$ with associated loop graph $\mathcal{G}^\bowtie$. Then $\mathcal{L}(\mathcal{G})$, which denotes the perfect matching lattice of $\mathcal{G}$ is in bijection with the canonical submodule lattice $\mathcal{L}(M(w))$.

\begin{proof}

By propositions \ref{Ptomod} and \ref{modtoP} we have that the two lattices are equal as sets. Suppose that $H(\mathcal{G})$ and $H(M)$ are the Hasse diagrams of $\mathcal{L}(\mathcal{G})$ and $\mathcal{L}(M(w))$. We need to show that there exists an edge between two vertices in $H(\mathcal{G})$ if and only if there exists an edge between the two corresponding vertices of $H(M)$.\\

Suppose that $P$ and $P'$ are two perfect matchings of $\mathcal{L}(\mathcal{G})$ which are connected with an edge in $\mathcal{L}(\mathcal{G})$. Let $P\ominus P_{min} = \bigcup{H_i}$ and $P'\ominus P_{min} = \bigcup{H_i'}$, and $\phi\colon \bigcup{H_i}\hookrightarrow \mathcal{G}^\bowtie$ and $\phi'\colon \bigcup{H_i'}\hookrightarrow \mathcal{G}^\bowtie$  the corresponding canonical embedding of graphs.
Since $P$ and $P'$ are connected with an edge, we have that they agree everywhere apart from one tile $T_l$ in which they have opposite matchings. W.l.o.g. suppose that the tile $T_l$ is in $P\ominus P_{min}$. \\
We can associate to $P'$ a submodule $N$ of $M(w)$. Since $P$ and $P'$ agree in all other tiles apart from $T_l$ by our construction in Prop.\ref{Ptomod} this submodule is exactly $M(w_P)$ removing the simple corresponding to $T_l$. However, this means that these two modules $N$ and $M(w_P)$ are connected with an edge in $H(M)$. Following Prop.\ref{Ptomod} we also have two induced canonical embeddings of modules $\Phi'\colon N\hookrightarrow M(w)$ and $\Phi\colon M(w_P)\hookrightarrow M(w)$. Since these embeddings are induced by the canonical embeddings of graphs $\phi$ and $\phi'$, which agree everywhere apart from the tile $T_l$, we see that $\Phi|_N = \Phi'$.\\
Conversely suppose that  $(M(w_i), \phi_i)$ and $(M(w_j), \phi_j)$ are two canonical submodules of $M(w)$ connected by an edge in $H(M)$. Then $w_i$ and $w_j$ agree in every vertex apart from one vertex $l$. W.l.o.g. assume that $l$ belongs to $w_i$. Following our construction in Prop.\ref{modtoP} there are two perfect matchings $P_{M(w_i)}$ and $P_{M(w_j)}$ corresponding to $M(w_i)$ and $M(w_j)$ respectively. Then we have that :
\[
(P_{M(w_i)}\ominus P_{min}) \bigcup (P_{M(w_j)}\ominus P_{min}) = T,
\]
where T is a tile of $G^\bowtie$. But this means that $P_{M(w_i)}$ and $P_{M(w_i)}$ agree on all tiles apart from $T$ and so by the definition of the perfect matching lattice of a loop graph there must be an edge connecting the two perfect matchings.

\end{proof}

\end{theorem}

\chapter{Skein relations}

\section{ Skein relations for crossing arcs}
Our goal in this section is to explore and prove skein relation in the cases that involve:
\begin{itemize}
    \item [(i)] a plain arc and a singly notched arc, which cross each other,
    \item [(ii)] a plain arc and a doubly notched arc, which cross each other,
    \item [(iii)] two singly notched arcs, which cross each other,
    \item [(iv)] a singly notched arc and a doubly notched arc, which cross each other,
    \item [(v)] two doubly notched arcs which cross each other.
\end{itemize}
These are the first five cases which are listed in the Appendix of \cite{Musiker_Schiffler_Williams_2013}.\\
The following theorem addresses these cases:

\begin{theorem}\label{skein relations regular crossing}
    Let $(S,M,P,T)$ be a triangulated punctured surface and $\mathcal{A}$ the cluster algebra associated to it. Let $\gamma_1$ and $\gamma_2$ two arcs, which cross on the surface and let $(\gamma_3,\gamma_4)$ and $(\gamma_5,\gamma_6)$ be the two pairs of arcs obtained by smoothing the crossing. Then : 
        \[
        x_{\gamma_1} x_{\gamma_2} = Y^{-}x_{\gamma_3} x_{\gamma_4}  + Y^{+} x_{\gamma_5} x_{\gamma_6},
        \]
        where $x_{\gamma_i}$ denotes the cluster algebra element associated to the arc $\gamma_i$ and $Y^{-}, Y^{+}$ are monomials in $y_i$ coefficients and satisfy the condition that one of the two is equal to 1.
\end{theorem}

\subsection{ Resolution of crossing between two arcs, with non empty overlap}
Suppose that $(S,M,P)$ is a punctured surface and $T$ is a triangulation. Then, one can associate a Cluster Algebra $\mathcal{A}= \mathcal{A}(x_T,y_T,Q_T)$ and a Jacobian algebra to that surface and that triangulation. We have also showcased, how one can associate a loopstring, given a loop graph.

Suppose that  there is a crossing between two arcs $\gamma_1$ and $\gamma_2$, where both arcs can have one or both of their endpoints tagged notched in their endpoints. We can then associate two loopstrings to them, $w_2$ and $w_1$ respectively .\\
In \cite{canakcci2021lattice}, the authors interpreted the notion of crossing of snake graphs and their resolutions, as these were defined in \cite{CANAKCI2013240}, in terms of string combinatorics. This interpretation will prove very useful to us, since when there is a crossing between two arcs which follows the restrictions of this section (i.e. a plain arc and a singly notched arc crossing), on the surface $S$, the resolution of the crossing does not depend significantly on the  existence of punctures on the surface. \\
The following definition, is a generalization of the first part of Definition 4.1 in \cite{canakcci2021lattice}.

\usetikzlibrary { decorations.pathmorphing, decorations.pathreplacing, decorations.shapes}
\usetikzlibrary{intersections,positioning,backgrounds,fit,calc}
\usetikzlibrary{through}

\begin{definition}\label{resofloopstrings}
    Suppose that $w_1$ and $w_2$ are two abstract loopstrings which possibly contain a hook at their start or their end. Assume that $w_1= h_1u_1ambv_1h_2$ and $w_2= h_3u_2cmdv_2h_4$ where $u_i, v_i, m$ are substrings, $a,d$ are direct arrows, $b,c$ are inverse arrows, and $h_i$ are hooks. We then say that $w_1$ and $w_2$ \textit{cross} in $m$.\\

\begin{figure}[h!tbp]
\[
\resizebox{.8\textwidth}{!}{
\begin{tikzpicture}
\node at (-1,0){$\begin{tikzpicture}[scale=1.5, every node/.style={transform shape},node distance=1cm and 1.5cm]
\coordinate[label=left:{}] (1);
\coordinate[right=1cm of 1] (2);
\coordinate[right=1cm of 2] (3);
\coordinate[below right=.8cm of 3] (4);
\coordinate[right=1cm of 4] (5);
\coordinate[above right=.8cm of 5] (6);
\coordinate[right=1cm of 6] (7);
\coordinate[right=1cm of 7](8);
\coordinate[above=.5cm of 8](9);
\coordinate[left=.2cm of 8](10);

\coordinate[left=1cm of 2](11);
\coordinate[above=.5cm of 11](12);

\coordinate[right=-1cm of 1,label=right:{$w_1=$}] (1');
\coordinate[right=1 of 12,label=right:{$h_1$}](2');

\draw[thick,decorate,decoration={snake,amplitude=.4mm,segment length=2mm}]
(2)-- node [anchor=north,scale=.9]{$u_1$} (3)
(6)-- node [anchor=north,scale=.9]{$v_1$} (7);

\draw[thick]
(4)-- node [anchor=north,scale=.9]{$m$} (5);

\draw[thick,->] (3)--node [anchor=south west,scale=.9]{$a$}(4);
\draw[thick,->] (6)--node [anchor=south east,scale=.9]{$b$}(5);

\draw [thick,->] (7) to [out=0,in=-90,looseness=0.8] (9) to[out=90,in=90,looseness=2] (7);

\draw [thick,->] (2) to [out=180,in=-90,looseness=0.8] (12) to[out=90,in=90,looseness=2] (2);

\node[] at (10) {$h_2$};

\end{tikzpicture}$};

\node at (10,-1){$\begin{tikzpicture}[scale=1.5, every node/.style={transform shape},node distance=1cm and 1.5cm]
\coordinate[label=left:{}] (1);
\coordinate[right=1cm of 1] (2);
\coordinate[above right=.8cm of 2,label=right:{}] (3);
\coordinate[right=1cm of 3] (4);
\coordinate[below right=.8cm of 4] (5);
\coordinate[right=1cm of 5] (6);

\coordinate[right=1cm of 6](7);
\coordinate[above=.5cm of 7](8);
\coordinate[left=1cm of 1](9);

\coordinate[right=-2cm of 1,label=right:{$w_2=$}] (1');
\coordinate[right=.6cm of 6,label=right:{$h_4$}] (2');
\coordinate[right=1cm of 10,label=right:{$h_3$}] (3');

\draw[thick,decorate,decoration={snake,amplitude=.4mm,segment length=2mm}]
(1)-- node [anchor=north,scale=.9]{$u_2$} (2)
(5)-- node [anchor=north,scale=.9]{$v_2$} (6);

\draw[thick]
(3)-- node [anchor=south,scale=.9]{$m$} (4);

\draw[thick,->] (3)--node [anchor=south east,scale=.9]{$c$}(2);
\draw[thick,->] (4)--node [anchor=south west,scale=.9]{$d$}(5);

\draw [thick,->] (6) to [out=0,in=-90,looseness=0.8] (8) to[out=90,in=90,looseness=2] (6);
\draw [thick,->] (1) to [out=180,in=-90,looseness=0.8] (10) to[out=90,in=90,looseness=2] (1);

\end{tikzpicture}$};

\end{tikzpicture}
}
\]
\end{figure}

The \textit{resolution of the crossing of $w_1$ and $w_2$ with respect to $m$} are the loopstrings $w_1, w_2, w_3$ and $w_4$ which are defined as follows:
    \begin{enumerate}[(i)]
        \item  $w_3 = h_1u_1amdv_2h_4$.
        \item  $w_4 = h_3u_2cmbv_1h_2$.
        \item 
         \begin{equation*}
             w_5 = \begin{cases}
                 h_1u_1eu_2^{-1}h_3 & \text{if $u_1, u_2 \neq \emptyset$ and where $e= \leftarrow$}.\\
                 h_1u_2'h_3 & \text{if $u_1 = \emptyset$ and $u_2\neq \emptyset$}.\\
                 h_1u_1'h_3 & \text{if $u_2 = \emptyset$ and $u_1\neq \emptyset$}.
             \end{cases}
         \end{equation*}
         where $u_2'$ is such that $u_2= u_2'r''$ where $r''$ is a maximal sequence of direct arrows and $u_1'$ is such that $u_1= u_1'r'$ where $r'$ is a maximal sequence of inverse arrows.
        \item
         \begin{equation*}
             w_6 = \begin{cases}
                 h_2v_1^{-1}fv_2h_4 & \text{if $v_1, v_2 \neq \emptyset$ and where $f= \rightarrow$}.\\
                 h_2v_2'h_4 & \text{if $v_1 = \emptyset$ and $v_2\neq \emptyset$}.\\
                 h_2v_1'h_4 & \text{if $v_2 = \emptyset$ and $v_1\neq \emptyset$}.
             \end{cases}
         \end{equation*}
         where $v_2'$ is such that $v_2=r''v_2'$, where $r''$ is a maximal sequence of inverse arrows and $v_1'$ is such that $v_1=r''v_1'$, where $r''$ is a maximal sequence of direct arrows
    \end{enumerate}

    The resolution of a crossing of $w_1$ and $w_2$ is illustrated in Figure (\ref{Fig:res of w1,w2})
    
\end{definition}

\begin{remark}
    Looking at definition \ref{resofloopstrings}, one can notice that it is almost identical to the definition given in \cite{canakcci2021lattice}. The main difference is the inclusion of the hooks $h_1, h_2, h_3$ and $h_4$. This should come as no surprise, since the resolution of the crossing between two arcs, should behave in the same way, even if the arcs are tagged notched on their endpoints. We could give the previous definition in two steps, by first reducing the notched arcs to their plain versions, applying definition 4.1,  and then adding the tagging at the appropriate arcs. \\

\end{remark}

\begin{remark}
    Since we are dealing with all possible five cases listed at the start of the section, we are allowing endpoints of the arcs to be untagged to their endings. In this case the definition remains the  same by just setting $h_i =\emptyset$ to the appropriate hooks. 
\end{remark}

At this point we need to make some clarifications regarding the notation of a hook and how it "interacts" with other strings.

\begin{remark}
Suppose that $(S,M.P,T)$ is a triangulated surface, $p$ is a puncture and $\alpha_1, \dots, \alpha_n$ are the arcs of the triangulation $T$ that are adjacent to the puncture $p$ written in a clockwise order. Then, the collection  of these arc is called a hook around $p$ and we write $h_p = \alpha_1\dots\alpha_n$. Note that the labeling of the arcs is arbitrary and there is no clear choice of which adjacent arc to $p$ we could name $a_i$. However, if $\beta_1, \dots, \beta_n$ is a different labeling, which satisfies the rule that  the arcs are labeled in the clockwise order, $h_p'=\beta_1 \dots \beta_n$ is equivalent to $h_p$.\\
We need to also note that you cannot define a loopstring $w = h$.\\
In practice, during the proofs and the examples we are allowed to change equivalent representations of $h$ without explicitly mentioning it every time. The necessity of it, is showcased in the next example \ref{examplehooks}.  
\end{remark}

\begin{example}\label{examplehooks}
    Example including picture where I am explaining the notation of $w=uh$.
\end{example}

Now that we have explained what a hook is, we need to make one more convention regarding the notation of the resolution. 

\begin{remark}
    Let $w_i$, $1\leq i\leq6$ as they were defined in \ref{resofloopstrings}. Suppose that $v_1 = a_1\dots a_n$ and $h = a_{n-k}\dots a_n a_{n+1}\dots a_m$, where $k< n$. Then \\ $v_1h = a_1\dots a_n a_{n+1}\dots a_m$, instead of $a_1\dots a_n a_{n-k}\dots a_n a_{n+1}\dots a_m$ which does not correspond to a module in the Jacobian algebra.
\end{remark}

In our next remark, we will illustrate a key ingredient for some of the proofs that will follow. The key idea is that if $w$ is a string module, and $a$ is a vertex corresponding to a top $w$, then removing $a$ we take a maximal submodule $w' = w\setminus a$ of $w$. Therefore in principle, we could "reach" each submodule of $w$ by sequentially removing tops $a_i$, $1\leq i\leq n$, from $w$.

\begin{remark}\label{backtrackinginlattice}
    Suppose that $w$ is a loopstring (or equivalently a module). If $m$ is a submodule of $w$ then there exists a sequence $(a_1, \dots, a_n)$ such that $m$ can be "reached" from $w$ by removing one simple belonging to the top of each submodule. However, this sequence is not necessarily unique. For example if we consider the quiver $Q: 1\longrightarrow 2\longleftarrow 3\longleftarrow 4$ and the submodule $0\longrightarrow k\longleftarrow 0\longleftarrow 0$ of $k\longrightarrow k\longleftarrow k \longleftarrow k$, then three possible sequences would be $(1,4,3), (4,1,3)$ or $(4,3,1)$.\\
    An additional thing that one can notice, is that if we are given the end of a sequence $\{b_k, \dots, b_n\}$ leading to a submodule $m$ of $w$, and another sequence \\ $\{a_1, \dots ,a_{k-1}, a_k, \dots, a_n\}$, where $\{a_k, \dots, a_n\}$ is a rearrangement of the sequence $\{b_k, \dots, b_n\}$, then we also have that $\{a_1, \dots ,a_{k-1}, b_k, \dots, b_n\}$ is a sequence from $w$ to $m$.\\
    For example, if we were given in the previous quiver $Q$, and there exists sequences $(\dots, 1, 3)$ and $(4,3,1)$ leading from $k\longrightarrow k\longleftarrow k \longleftarrow k$ to $0\longrightarrow k\longleftarrow 0\longleftarrow 0$, we could deduce that there exists also the sequence $(4,1,3)$ leading from $k\longrightarrow k\longleftarrow k \longleftarrow k$ to $0\longrightarrow k\longleftarrow 0\longleftarrow 0$.
    
\end{remark}

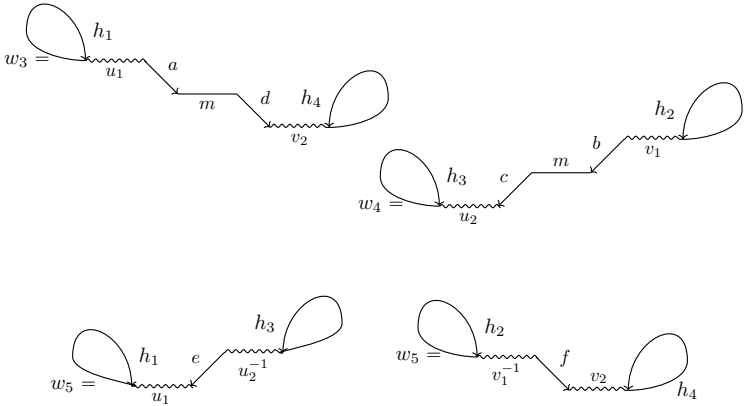
\begin{figure}
\[
\resizebox{.8\textwidth}{!}{
\begin{tikzpicture}
\node at (0,0){$\begin{tikzpicture}[scale=1.5, every node/.style={transform shape},node distance=1cm and 1.5cm]
\coordinate[label=left:{}] (1);
\coordinate[right=1cm of 1] (2);
\coordinate[right=1cm of 2] (3);
\coordinate[below right=.8cm of 3] (4);
\coordinate[right=1cm of 4] (5);
\coordinate[below right=.8cm of 5] (6);
\coordinate[right=1cm of 6] (7);
\coordinate[right=1cm of 7](8);
\coordinate[above=.5cm of 8](9);
\coordinate[left=.2cm of 8](10);

%%%
\coordinate[left=1cm of 2](11);
\coordinate[above=.5cm of 11](12);
\coordinate[right=1cm of 7](13);
\coordinate[above=.5cm of 13](14);
%%%

\coordinate[right=-.5cm of 1,label=right:{$w_3=$}] (1');
\coordinate[right=1cm of 12,label=right:{$h_1$}] (2');
\coordinate[left=1cm of 14,label=left:{$h_4$}] (3');

\draw[thick,decorate,decoration={snake,amplitude=.4mm,segment length=2mm}]
(2)-- node [anchor=north,scale=.9]{$u_1$} (3)
(6)-- node [anchor=north,scale=.9]{$v_2$} (7);

\draw[thick]
(4)-- node [anchor=north,scale=.9]{$m$} (5);

\draw[thick,->] (3)--node [anchor=south west,scale=.9]{$a$}(4);
\draw[thick,->] (5)--node [anchor=south west,scale=.9]{$d$}(6);

%%%%%%%%%%
\draw [thick,->] (7) to [out=0,in=-90,looseness=0.8] (14) to[out=90,in=90,looseness=2] (7);
\draw [thick,->] (2) to [out=180,in=-90,looseness=0.8] (12) to[out=90,in=90,looseness=2] (2);
%%%%%%%%%%

\end{tikzpicture}$};

\node at (9,-2){$\begin{tikzpicture}[scale=1.5, every node/.style={transform shape},node distance=1cm and 1.5cm]
\coordinate[label=left:{}] (1);
\coordinate[right=1cm of 1] (2);
\coordinate[above right=.8cm of 2,label=right:{}] (3);
\coordinate[right=1cm of 3] (4);
\coordinate[above right=.8cm of 4] (5);
\coordinate[right=1cm of 5] (6);
%\coordinate[right=1cm of 6] (7);
\coordinate[right=1cm of 6](8);
\coordinate[above=.5cm of 8](9);
\coordinate[left=.2cm of 8](10);

%%%
\coordinate[left=1cm of 1](11);
\coordinate[above=.5cm of 11](12);
%%%

\coordinate[right=-1.5cm of 1,label=right:{$w_4=$}] (1');
\coordinate[right=-1cm of 9,label=left:{$h_2$}] (2');
\coordinate[left=-1cm of 12,label=right:{$h_3$}] (3');

\draw[thick,decorate,decoration={snake,amplitude=.4mm,segment length=2mm}]
(1)-- node [anchor=north,scale=.9]{$u_2$} (2)
(5)-- node [anchor=north,scale=.9]{$v_1$} (6);

\draw[thick]
(3)-- node [anchor=south,scale=.9]{$m$} (4);

\draw[thick,->] (3)--node [anchor=south east,scale=.9]{$c$}(2);
\draw[thick,->] (5)--node [anchor=south east,scale=.9]{$b$}(4);

\draw [thick,->] (6) to [out=0,in=-90,looseness=0.8] (9) to[out=90,in=90,looseness=2] (6);
\draw [thick,->] (1) to [out=180,in=-90,looseness=0.8] (12) to[out=90,in=90,looseness=2] (1);

\end{tikzpicture}$};

\node at (0,-7){$\begin{tikzpicture}[scale=1.5, every node/.style={transform shape},node distance=1cm and 1.5cm]
\coordinate[label=left:{}] (1);
\coordinate[right=1cm of 1] (2);
\coordinate[above right=.8cm of 2,label=right:{}] (3);
\coordinate[right=1cm of 3] (4);
\coordinate[below right=.8cm of 4] (5);
\coordinate[right=1cm of 5] (6);
\coordinate[left=1cm of 1] (7);
\coordinate[above=.5cm of 7] (8);
\coordinate[right=1cm of 4] (9);
\coordinate[above=.5cm of 9] (10);

\coordinate[right=-1.5cm of 1,label=right:{$w_5=$}] (1');
\coordinate[left=1cm of 10,label=left:{$h_3$}] (2');
\coordinate[right=1cm of 8,label=right:{$h_1$}] (3');

\draw[thick,decorate,decoration={snake,amplitude=.4mm,segment length=2mm}]
(1)-- node [anchor=north,scale=.9]{$u_1$} (2)
(3)-- node [anchor=north,scale=.9]{$u_2^{-1}$} (4);

\draw[thick,->] (3)--node [anchor=south east,scale=.9]{$e$}(2);

\draw [thick,->] (1) to [out=0,in=-90,looseness=0.8] (8) to[out=90,in=90,looseness=2] (1);
\draw [thick,->] (4) to [out=180,in=-90,looseness=0.8] (10) to[out=90,in=90,looseness=2] (4);

\end{tikzpicture}$};

\node at (9,-7){$\begin{tikzpicture}[scale=1.5, every node/.style={transform shape},node distance=1cm and 1.5cm]
\coordinate[label=left:{}] (1);
\coordinate[right=1cm of 1] (2);
\coordinate[below right=.8cm of 2,label=right:{}] (3);
\coordinate[right=1cm of 3] (4);
%\coordinate[below right=.8cm of 4] (5);
%\coordinate[right=1cm of 4] (6);
%\coordinate[right=1cm of 4](7);
\coordinate[right=1cm of 4](8);
\coordinate[above=0.5cm of 8](9);

\coordinate[left=1cm of 1](10);
\coordinate[above=0.5cm of 10](11);

\coordinate[right=-1.5cm of 1,label=right:{$w_5=$}] (1');
\coordinate[right=1cm of 11,label=right:{$h_2$}] (2');

\draw[thick,decorate,decoration={snake,amplitude=.4mm,segment length=2mm}]
(1)-- node [anchor=north,scale=.9]{$v_1^{-1}$} (2)
(3)-- node [anchor=south,scale=.9]{$v_2$} (4);

\draw[thick,->] (2)--node [anchor=south west,scale=.9]{$f$}(3);

\draw [thick,->] (4) to [out=0,in=-90,looseness=0.8] (9) to[out=90,in=90,looseness=2] (4);
\draw [thick,->] (1) to [out=180,in=-90,looseness=0.8] (11) to[out=90,in=90,looseness=2] (1);

\node[] at (8) {$h_4$};

\end{tikzpicture}$};

\end{tikzpicture}
}
\]
\caption{Resolution of $w_1$ and $w_2$.} \label{Fig:res of w1,w2}
\end{figure}

\begin{proposition}\label{removertopwelldefined}
    Let $w_1, w_2, w_3, w_4, w_5$ and $w_6$ as they were defined in \ref{resofloopstrings}. Suppose that $m_1$ (resp. $m_2$) is a proper submodule of  $w_3\oplus w_4$ (resp $w_5\oplus w_6$). If $m_1$  (resp. $m_2$) is reached from $w_3\oplus w_4$ (resp. $w_5\oplus w_6$) by the sequence $(a_i), 1\leq i\leq n$, then we can apply the same sequence $(a_i)$ to the module $w_1\oplus w_2$ (resp. $w_1\oplus (w_2\setminus m)$) to obtain a submodule $m_1'$ (resp. $m_2'$) of it.
    \begin{proof}

       % We need to prove the statement for both cases. \\
        %\begin{itemize}
           % \item 
           Suppose that $m_1$ is a submodule of $w_3\oplus w_4$ and $(a_i)$, for $1\leq i\leq n$ is a sequence of simples of $w_3\oplus w_4$, such that $(\dots((w_3\oplus w_4/ a_1 ) /a_2)/ \dots )/a_n = m_1$. In the context of loopstrings, this corresponds to removing a sequence of vertices $(a_i), 1\leq i\leq n$, such that in each step, there is no incoming arrow to that vertex. \\
        We will prove the statement by induction. For the base case, we want to prove that if $a_1$ belongs to the top of the module $w_3\oplus w_4$, then it also belongs to the top of the module $w_1\oplus w_2$. It follows immediately by the structure of each string, that if $a_1$ belongs to one of the substrings $u_1, u_2, v_1, v_2$ or is the top of the hook $h$ then it is also a top for the module $w_3\oplus w_4$. If $a_1$ is the first letter of the string $m$, then since this is also assumed to be a top, it means that it was removed from the string $w_4$, and therefore $w_4= u_2 \xleftarrow{c} a_1 \longrightarrow (m\setminus a_1)v_1h$. Therefore $a_1$ belongs also to the top of $w_2$. A similar argument works also for when $a_1$ is assumed to be the last letter of the string $m$.\\

        Assume now, that the sequence of simples $(a_i), 1\leq i\leq n-1$ has already been removed from both loopstrings $w_3\oplus w_4$ and $w_1\oplus w_2$.\\ We need to take into account two cases, depending on which substring, each $a_i$ belongs. The following construction will do exactly this, showcasing how depending on these cases, which top must be removed and from which submodule.\\
        \underline{Construction:}\\
        Suppose that $a_i, 1\leq i\leq n$ is not an element of the string $m$. Then, there is unique choice up to embedding that we can do, regarding from what substring of $w_1\oplus w_2$ we can remove it. Indeed, if $a_1\in u_1$ then we can remove $a_1$ only from the copy of $u_1$ viewed as a substring of $w_1$.\\
        However, if $a_1$ is an element of the substring $m$, then there are two choices. We could either remove it from $m$ viewed as a substring of $w_1$ or a substring of $w_2$. In order to make this construction precise, we impose that if we remove an element of $m$ viewed as a substring of $w_3$ (resp. $w_4$), then we remove the same element from $m$ viewed as a substring of $w_1$ (resp. $w_2$), if this is possible. \\
        In the above construction there  is one exception. Without loss of generality assume that $v_1\neq \emptyset$, (since otherwise we could argue using the hook $h_2$). Suppose that $a_n= m^{l}$ where $m^{l}$ is the last letter of the string $m$, viewed as a substring of $w_3$, and $v_1^1 \neq a_i$ for $1\leq i\leq n-1$, where $v_1^1$ is the first letter of $v_1$. Attempting to remove this simple from the loopstring $w_1$ would lead to a contradiction, since locally in $w_1$ the arrow configuration is as follows:
        \[
        m^l \leftarrow v_1^1.
        \]
        Therefore, we need to apply the sequence $a_i$ in a different way. We retrace the removal of tops and instead of removing local tops from the substring $m$ viewed as a substring of $w_1$, we remove these tops from $m$ viewed as a substring of $w_2$ when possible. This procedure is possible since locally the whole substring $m$ is a top in $w_2$ and therefore in each step we know that we can actually still remove inductively the sequence $a_i$ for $1\leq i\leq n-1$. This allows us to remove $a_n= m^l$ from $w_2$ since locally now the arrow configuration is the following:
        \[
        m^l \rightarrow v_2^1,
        \]
        where $v_2^1$ is the first letter of $v_2$. \\

        The above construction  also indicates us why the induction step is true, since by the way that it was defined we can always remove the next top from the suitable submodule.\\
        
        Suppose now that $m_2$ is a submodule of $w_5\oplus w_6$ and $(a_i)$, for $1\leq i\leq n$ is a sequence of simples of $w_5\oplus w_6$, such that $(\dots((w_3\oplus w_4/ a_1 ) /a_2)/ \dots )/a_n = m_2$. It follows by a simple induction that we can also remove the same sequence $(a_i)$ from the module $w_1\oplus (w_2\setminus m)$ since in this case there is always a unique choice up to embedding of submodules that we can inductively remove tops from $w_1\oplus (w_2\setminus m)$. \\

        \end{proof}
\end{proposition}

\begin{remark}
    Before going any further, we need to explain how the construction in the above proposition \ref{removertopwelldefined} works and why the exception mentioned is the only case that would lead to a contradiction. \\
    Our goal is to be able to remove tops and map submodules of $w_3\oplus w_4$ to submodules of $w_1\oplus w_2$ in a unique way. Since the substring $m$ appears twice in the module $w_1\oplus w_2$ it is obvious that every time that we remove a simple from $m$ we have to make a choice. In our construction we choose to remove a simple from $w_1$ if originally the removed simple was from $w_3$, when possible. However, we could have chosen to remove that simple from $w_2$ when removing it from $w_2$. Doing so would not create any problems apart from the fact that the exception that we mention would be different, and we would run on the same contradiction without the necessary changes. Based on this proposition, we will later construct a morphism from the submodules of $w_3\oplus w_4$ to the submodules of $w_1\oplus w_2$ and we will argue that this map is an injection by construction. \\
    Notice that by making this choice and since all the other substrings $h_1u_1, v_1h_2,$ $h_3u_2$ and $v_2h_4$ appear only once on each side, removing a different sequence of tops from $w_3\oplus w_4$ leads to a different submodule of $w_1\oplus w_2$. The only non-obvious part on why this assignment is unique regards the aforementioned exception. By assuming that $v^1_1$ has not been removed already, we make another indirect choice, namely when $v^1_1$ is not removed and $m^l$ is removed, we "biject" the substring $m$ of $w_3$ to the substring $m$ of $w_2$. However, if both $v_1^1$ and $m^l$ have been removed already we "biject" the substring $m$ of $w_3$ to the substring $m$ of $w_1$. In practice we have the following correspondence:
    \begin{align*}
    (h_1u_1)' \oplus w_4 \longrightarrow ((h_1u_1)'\oplus mv_1h_2)\oplus h_3u_2,\\
    \end{align*}
    where $(h_1u_1)'$ is a submodule of $h_1u_1$.

\end{remark}

From now on, we will not differentiate between a string $w$ and the associated module $M(w)$, denoting both in the same way. Let also $SM(w)$ denote the set of all the submodules of the module $M_w$.\\
The first big step towards proving skein relations is establishing a bijection between the submodules of both hands of the resolution. The existence of a bijection between $SM(w_3\oplus w_4) \bigcup SM(w_5\oplus w_6)$ and  $SM(w_1\oplus w_2)$ should come as no surprise, since as we will see later, these are just two finite sets with the same cardinality. However, the way that this bijection is constructed is very crucial and will allow us later to pass from the loopstrings to the associated monomials.

\begin{theorem}\label{bijofmod}
    Let $w_1, w_2, w_3, w_4, w_5$ and $w_6$ as they were defined in \ref{resofloopstrings}. Then there exists a bijection:
    \[
    \Psi \colon SM(w_3\oplus w_4) \bigcup SM(w_5\oplus w_6) \to SM(w_1\oplus w_2),
    \]
    which is defined as follows:
    \begin{gather*}
    \Psi(w_3\oplus w_4) = w_1\oplus w_2, \\
    \Psi(w_5\oplus w_6) = w_1\oplus (w_2\setminus m),     
    \end{gather*}
    where $w_2\setminus m\coloneqq u_2\oplus v_2$.\\
    If $w$ is a submodule of $w_3\oplus w_4$  which is reached by the sequence $\{a_i\} 1\leq i\leq n$, and $w'$ is the submodule of $w_1\oplus w_2$ which is reached by the same sequence, then we define:
    \[
    \Psi(w) = w'.
    \]
    If $w$ is a submodule of $w_5\oplus w_6$, then $\Psi(w)$ is defined in a similar way.
    
\end{theorem}

We will present the proof of the previous Theorem in steps. At first, we will prove that the map $\Psi$ is an injective map. For this, we will need to make a remark regarding the structure of the submodule lattice of a loopstring module.

We are now ready to prove the injectivity of the map $\Psi$.

\begin{proposition}\label{injofPsi}
    The map $\Psi$ as it was defined in Theorem \ref{bijofmod} is an injection.
    \begin{proof}
        Let $m_1, m_2\in SM(w_3\oplus w_4) \bigcup SM(w_5\oplus w_6)$ such that $\Psi(m_1)=\Psi(m_2)$.\\
        First of all it is obvious by construction of the map $\Psi$ that if both $m_1, m_2\in SM(w_3\oplus w_4)$ or $m_1, m_2\in SM(w_5\oplus w_6)$, then $m_1= m_2$.\\
        We just have to show now, that if $m_1\in SM(w_3\oplus w_4)$ and $m_2\in SM(w_5\oplus w_6)$ then $\Psi(m_1)\neq \Psi(m_2)$. Assume that $\Psi(m_1)=\Psi(m_2)$. Since $m_1\in SM(w_3\oplus w_4)$, we know by construction of the map $\Psi$, that there exists a (not necessarily unique) sequence $(a_i), 1\leq i\leq n$ of simples, such that when viewed as a loopstring $(w_3\oplus w_4) \setminus \{a_i\} = m_1$. By proposition \ref{removertopwelldefined} and by construction of $\Psi$, we have that $\Psi(m_1) = \Psi(w_3\oplus w_4)\setminus \{a_i\}$.\\
        Similarly, working on $m_2$, there is a sequence $(b_i), 1\leq i\leq k$, such that $m_1= (w_5\oplus w_6)\setminus \{b_i\}$ and $\Psi(m_2) = \Psi(w_5\oplus w_6)\setminus \{b_i\}$.\\
        However, since $\Psi(m_1)=\Psi(m_2)$ we obtain:
        \[
        \Psi(w_3\oplus w_4)\setminus \{a_i\} = \Psi(w_5\oplus w_6)\setminus \{b_i\}.
        \]
        Since  $\Psi(w_3\oplus w_4) = w_1\oplus w_2$, and $\Psi(w_5\oplus w_6) = w_1\oplus (w_2\setminus m)$, we further obtain that:
        \[
         (w_1\oplus w_2)\setminus\{a_i\} =  (w_1\oplus (w_2\setminus m))\setminus \{b_i\}.
        \]
        The previous equality shows us that the sets $\{a_i\}$ and $(\{b_i\}\bigcup m)$ must be equal. By Remark \ref{backtrackinginlattice}, we can deduce, that there must exist a sequence $\{m_{j_1}, \dots m_{j_l}, b_1, \dots, b_k\}$ leading from $w_1\oplus w_2$ to $\Psi(m_2)$. However, this also implies that using the same sequence $\{m_{j_1}, \dots m_{j_l}, b_1, \dots, b_k\}$, we can go from the module $w_3\oplus w_4$ to the module $m_2$. This is a contradiction, since this would mean that we could remove either the first or the last letter of the string $m$ without removing any other simple from the loopstrings $h_1u_1$ or $v_2h_4$, which cannot happen, by construction of the loopstrings $w_3$ and $w_4$.
    \end{proof}
\end{proposition}

What remains to be shown, is that $\Psi$ is also a surjection. For this, we shall argue on the size of both the domain and the codomain of the map. Let us notice first of all, that there is a finite number of elements both in the domain and the codomain. If we manage to prove that this is the same, then surjectivity will follow by the injectivity of $\Psi$.

\begin{remark}\label{termsincoeffreecase}
    We have already proven in Section 2, that if $w$ is a loopstring module with associated loop graph $\mathcal{G}$ then the perfect matching lattice of $\mathcal{G}$ is in bijection with the canonical submodule lattice $\mathcal{L}(w)$. We also know from Theorem 5.7 in \cite{wilson2020surface}, that each perfect matching of $\mathcal{G}$, corresponds to a summand of $x_w$, which denotes the associated element of $w$ in the cluster algebra. \\
    It is also easy for one to prove skein relations in the coefficient free case, as it was noted in section 8.4 of \cite{MSWpositivity}, since one can reduce the cases that we are working on, to the classic untagged skein relations, which are already known. Therefore, we have that if we set the $y$ variables to be equal to $1$, then the associated elements of the cluster algebra in each side of our map $\Psi$ must agree. Therefore, the number of the terms, which does not depend on the $y$ variables, must agree.
\end{remark}

\begin{proof}[Proof of Theorem \ref{bijofmod}]
The fact that $\Psi$ is a bijection follows immediately from the fact that $\Psi$ is an injective map, as it was proven in Prop \ref{injofPsi} and the fact that 
\[
|SM(w_3\oplus w_4)| + |SM(w_5\oplus w_6)| = |SM(w_1\oplus w_2)|
\]as it follows from Remark \ref{termsincoeffreecase}.

\end{proof}

\subsection{ Resolution of grafting}

Similarly to the classic case of unpunctured surfaces, there is another type of crossing between two arcs on the surface. This type, following the terminology of \cite{CANAKCI2013240}, is called grafting and occurs when there is an empty overlap between the associated graphs of the given arcs. The main difference with the previous resolution of the crossing between two arcs, is that, although the two arcs cross on the surface, when constructing the associated snake/ loop graphs or the modules associated to them, there is not a common overlap $m$ to them, as it was the case in the regular crossing. \\
We will mainly focus on the case that the grafting appears around a puncture, since, in the other case, that it occurs near a boundary component, it is just a direct generalization of the classic grafting.\\

Before giving the definition, let as recall what a grafting is. In the classical setting, grafting is basically an operation that given two snake graphs associates four new graphs. In our setting we will do a similar construction working with loopstrings instead of graphs. The main difference with the regular resolution, is that a priori, one could define a grafting between two abstract loopstrings even when they correspond to two arcs which do not cross in the surface. The results proven still hold true in the general case, although in principle we are mainly interested in the cases that the two arcs cross on the surface.\\

Let $w_1=u_1 a v_1$ an abstract string and $w_2 = h u_2$ an abstract loopstring, where $u_1, v_1, u_2$ are substrings, $a\in \{\leftarrow,\rightarrow\}$ and $h$ is a hook. Let also $e$ denote the arrow which connects $t(u_1)$ and $s(u_2)$ in the loopstring $w_2$, where $t(u_1)$ denotes the last letter of $u_1$ and $s(u_2)$ denotes the first letter of $u_2$.\\
We can then construct the following four loopstrings:

\begin{enumerate}[(i)]
    \item $ w_3= u_1eu_2$ 
    \item  
    \begin{itemize}
        \item If $a, v_1\neq \emptyset$ then $w_4 = h v_1'$ where $v_1'$  is such that $v_1 = rv_1'$, where $r$ is a maximal sequence of direct arrows (resp. inverse arrows) if $a$ is a direct arrow (resp. inverse arrow).
        \item If $a, v_1= \emptyset$ then $w_4 = h\setminus a_1$, where $a_1= t(e')$ and $e'$ the denotes the second arrow (apart from $e$) which connects $u_2$ with $h$ in $w_2$.\\
        %{\color{green} This is the first time that I am dealing with the case that the untagged version of an arc is part of the triangulation.}
    \end{itemize}
    \item  $w_5 = u_1' h$ where $u_1$ is such that $u_1 = u_1'r'$, where $r'$ is a maximal sequence of direct arrows (resp. inverse arrows) if $a$ is a direct arrow (resp. inverse arrow).
    \item 
    \begin{itemize}
        \item If $a, v_1\neq \emptyset$ then $w_6= v_1^{-1}au_2$.
        \item If $a, v_1= \emptyset$ then $w_6= u_2'$ where $u_2'$ is such that $u_2= r''u_2'$, where $r''$ is a maximal sequence  of direct arrows (resp. inverse arrows) if $e'$ is a direct arrow (resp. inverse arrow). 
    \end{itemize}
\end{enumerate}

\begin{definition}\label{grafofloopstrings}
The loopstrings $w_3, w_4, w_5$ and $w_6$, as they were just constructed, are called the \textit{resolution of the grafting} of $w_2$ on $w_1$ in $a$ if $a\neq \emptyset$ or in $e$ if $a= \emptyset$.
\end{definition}

\begin{remark}
    One can notice that if $a\neq \emptyset$ then $e= a^{-1}$. Figure \ref{fig:graftting} illustrates the local behavior of a grafting.

\end{remark}

\begin{figure}
    \centering
    \includegraphics[scale=0.8]{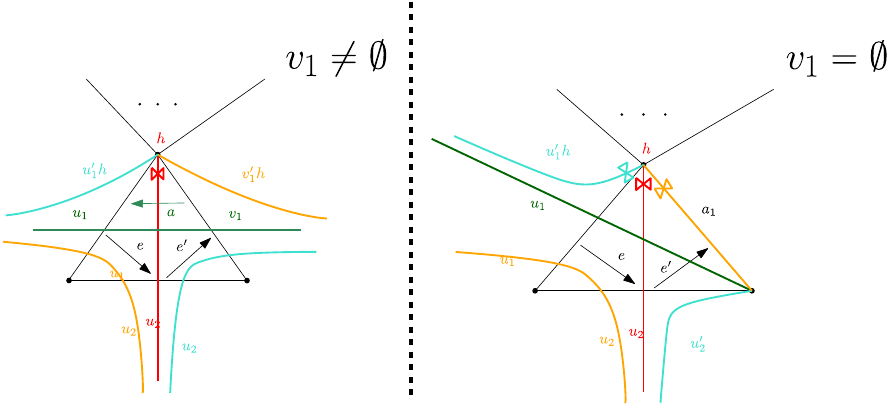}
    \caption{grafting}
    \label{fig:graftting}
\end{figure}

Our aim in this subsection will be to construct a map similar to the one defined in Theorem \ref{bijofmod} and prove that this is also a bijection. 

\begin{theorem}\label{bijofmodgrafting}
    Let $w_1, w_2, w_3, w_4, w_5$ and $w_6$ as they were defined in \ref{grafofloopstrings}. Then there exists a bijection:
    \[
    \Psi' \colon SM(w_3\oplus w_4) \bigcup SM(w_5\oplus w_6) \to SM(w_1\oplus w_2),
    \]
    which is defined as follows:
    \begin{enumerate}[(i)]
        \item If $a=\leftarrow$ then:
        \begin{gather*}
        \Psi'(w_3\oplus w_4) = w_1\oplus w_2, \\
        \Psi'(w_5\oplus w_6) = w_1\oplus (w_2\setminus r').
        \end{gather*}
        \item If $a=\rightarrow$ then:
        \begin{gather*}
        \Psi'(w_3\oplus w_4) = w_1\oplus (w_2\setminus r), \\
        \Psi'(w_5\oplus w_6) = w_1\oplus w_2.
        \end{gather*}
    \end{enumerate}
    In both cases, if $w$ is a submodule of $w_3\oplus w_4$ (or $w_5\oplus w_6$) then we define $\Psi'(w)$ to be the submodule of $w_3\oplus w_4$ (or equivalently of $w_5\oplus w_6$), which is induced by removing the same sequence $(a_i)$ of tops from $w_2\oplus w_4$ (or $w_5\oplus w_6$).
\end{theorem}

\begin{remark}
    Before proving Theorem \ref{bijofmodgrafting}, we need to explain why the strings $w_2\setminus r'$ and $w_2\setminus r$ appearing in (i) and (ii) are well defined in each case.\\
    First of all, since we are dealing with grafting around a puncture, it is easy to see that $r \bigcup r' = h$. If $a=\leftarrow$, then a letter from $r'$ is a top on $h$, while if $a=\rightarrow$, a letter from $r$ is a top on $h$. Therefore in each setting we can define $w_2\setminus r$ or $w_2\setminus r'$ , since we are removing tops in both cases.
\end{remark}

\begin{figure}
    \centering
    \includegraphics[scale=0.5]{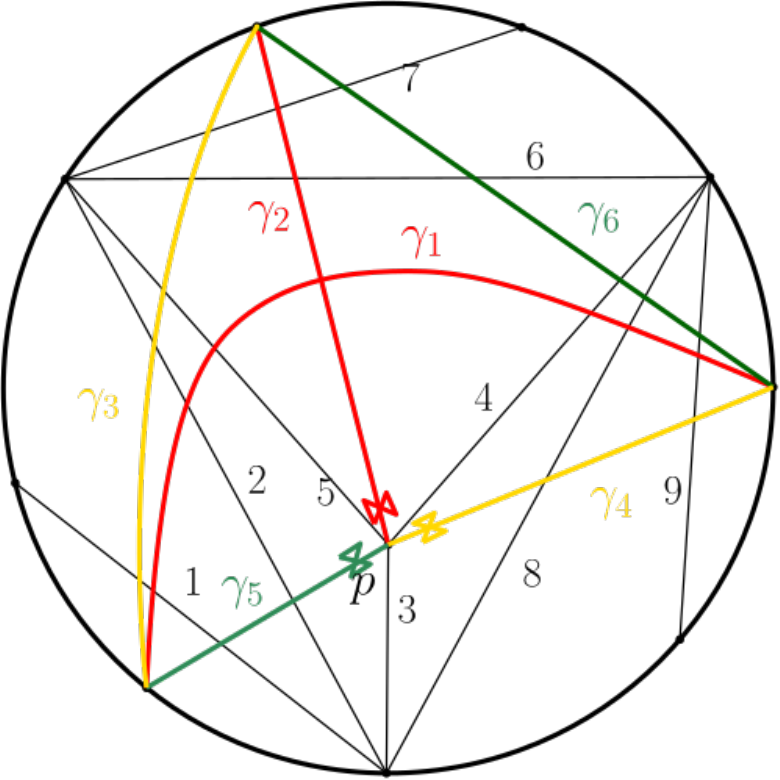}
    \caption{Resolution of grafting incompatibility for regular crossing.}
    \label{fig:Resolution of grafting incompatibility for regular crossing}
\end{figure}

\begin{example}
    Suppose that $(S,M,P,T)$ is the triangulated surface depicted on Figure \ref{fig:Resolution of grafting incompatibility for regular crossing}. We then have that:
    \[
    w_1= 1\rightarrow 2\leftarrow 5 \rightarrow 4 \leftarrow 8 \leftarrow 9,
    \]
    \[
    w_2= 4\rightarrow 5.
    \]
    It is easy to see that in this case we have that:
    \[
    h_1=  3\rightarrow 2\rightarrow 1,
    \]
    \[
    h_2= 6 \leftarrow 7 \leftarrow 8,
    \]
    \[
    m= 4\rightarrow 5.
    \]
    while $u_1$ and $u_2$ are empty! It is also easy to see that since both of these substrings are empty we have two regular incompatibilities, one at each of the two punctures. Following the resolution as it was defined earlier we take:
    \[
    w_3^{o}= \leftarrow3\rightarrow 2\rightarrow 1 \leftarrow 4 \rightarrow 5 \leftarrow 8 \rightarrow 7 \rightarrow 6\leftarrow 5\leftarrow 4\leftarrow,
    \]
    and $w_4= h_1$, $w_5= h_2$ and $w_6= \emptyset$.
\end{example}

\begin{example}
    Suppose that $w_1 = 9\leftarrow 5 \rightarrow 2 \rightarrow 1\rightarrow 8\leftarrow 9 \leftarrow 10$ and $w_2 = 7\leftarrow 3 \hookleftarrow 1 \rightarrow 8 \rightarrow 5 \rightarrow 2$. Then there is a grafting of $w_2$ onto $w_1$, with $ u_1 = 9\leftarrow 5 \rightarrow 2$, $a=\rightarrow$, $v_1 = 1\rightarrow 8\leftarrow 9 \leftarrow 10$, $u_2 = 7\leftarrow 3$ and $h = 1 \rightarrow 8 \rightarrow 5 \rightarrow 2$. Applying theorem \ref{bijofmodgrafting} we have that the grafting is resolved as follows:
    \begin{gather*}
        w_3 = 9\leftarrow 5 \rightarrow 2 \leftarrow  3 \rightarrow 7 ,\\
        w_4 = 10\rightarrow 9 \hookrightarrow 8 \leftarrow 1 \leftarrow 2 \leftarrow 5 ,\\
        w_5 = 10\rightarrow 9 \rightarrow 8 \leftarrow 1 \rightarrow 3 \rightarrow 7,\\
        w_6 = 9 \hookrightarrow 8 \leftarrow 1 \leftarrow 2 \leftarrow 5.
    \end{gather*}
    Notice that we are in case (ii) of theorem \ref{bijofmodgrafting} and indeed $r = 1\rightarrow 8$ is a top of $w_2$ and therefore:
    \[
    w_2\setminus r = 7\leftarrow 3 \rightarrow 2 \leftarrow 5
    \]
    is well defined.
\end{example}

\begin{proposition}\label{injofPsigrafting}
    The map $\Psi'$, as it was defined in \ref{bijofmodgrafting}, is an injection.
    \begin{proof}
        The proof, follows exactly the same idea as the proof of Proposition\ref{injofPsi}. The only difference in this case, is the fact that instead of the overlap $m$ which occurred in the crossing of two arcs, the contradiction here occurs from the existence of $r$ or $r'$ depending on each case.\\ S
        For example assume that $a=\rightarrow$. Suppose that $m_1\in SM(w_3\oplus w_4)$ and $m_2\in SM(w_5\oplus w_6)$ with $\Psi'(m_1)=\Psi'(m_2)$. Then we should be able to 
        first remove the substring $r$ from $w_3\oplus w_4$. However this is a contradiction, since this is a local socle of $w_4$.
    \end{proof}
\end{proposition}

\begin{proof}[Proof of Theorem \ref{bijofmodgrafting}]
The fact that $\Psi'$ is a bijection follows immediately from the fact that $\Psi'$ is an injective map, as it was proven in Prop \ref{injofPsigrafting} and the fact that 
\[
|SM(w_3\oplus w_4)| + |SM(w_5\oplus w_6)| = |SM(w_1\oplus w_2)|
\]as it follows from Remark \ref{termsincoeffreecase}.

\end{proof}

\subsection{ Monomials associated to loopstrings}

Snake graphs and loop graphs are a very important combinatorial tool which helps us calculate cluster variables, by computing perfect matchings and assigning a weight to each face and edge of these graphs. The loopstrings that we have already defined may be considered easier to work in some occasions since they seem to carry less information than the corresponding graphs. However, in order to prove skein relations, the established bijection $\Psi$ is not enough.\\
Our goal for the rest of this section will be to first prove that the module $M(P_{max})$ corresponding to the maximal perfect matching $P_{max}$ of a given snake or loop graph $G$, contains almost all the information needed for the calculation of the corresponding summand of the cluster element associated to the graph $G$. \\
After we have proven this fact, we will have all the necessary tools needed to prove our first main theorem \ref{skein relations regular crossing}, which we will do at the end of this subsection.

\begin{remark}
    Drawing inspiration from the notion of blossoming quiver, as it was defined in Definition 2.1 in \cite{palu2022nonkissing}, we would like to extend a loopstring to the so called blossoming loopstring. The idea of the blossoming quiver and loopstring is similar, but the difference is that in our case we are only interested in extending the loopstring at the start or at the end, since we want to keep track of the frozen variables which should be appearing in the associated monomial of a loopstring which we will define later.
\end{remark}

Assume that $(S,M,P)$ is a punctured surface and $T$ a given triangulation of this surface. Then classically, one can associate a quiver $Q_T$ to the given triangulation, where each vertex of the quiver corresponds to an arc of the triangulation. We can additionally assign a vertex to each boundary component of the surface and extend the quiver $Q_T$ to a new quiver $Q_T'$. Let us note that by construction, in each vertex of the quiver $Q_T'$, which corresponds to an arc of $T$, there are exactly two incoming and two outgoing arrows. Suppose furthermore, that $\gamma$ is an arc on the surface which does not belong to the triangulation $T$. Then, one can associate to the arc $\gamma$ a graph $G_\gamma$  as well as a loopstring $w_\gamma$ and a module $M(w_\gamma)$. In the following definition we will extend the loopstring $w_\gamma$, in a way that the new loopstring $w_\gamma'$ will contain all the information that $P_{max}$ contains.

\begin{definition}\label{monomial associated to loopstring}
    Let $(S,M,P,T)$, be a triangulated punctured surface and $Q_T'$ the extended quiver associated to that surface. Assume that $\gamma$ is a plain arc, which does not belong to the triangulation $T$ and $w_\gamma = aebucfd$ is the associated string, where $a, b, c, d$ are vertices of the quiver $Q_T'$, $e, f$ are arrows and $u$ is a substring of w.
    \begin{itemize}
        \item If $e= \rightarrow$, then there exists exactly one vertex $a'$ of the quiver $Q_T'$ such that locally, we have the following configuration: $a\rightarrow a'$ on the quiver $Q_T'$.
        \item  If $e=\leftarrow$, then there exist two vertices $a'$ and $a''$ of the quiver $Q_T'$, such that locally, we have the following configuration $a'\leftarrow a\rightarrow a''$ and the arcs associated to $a,a''$ and $b$ form a triangle on the surface $S$. 
        \item If $f= \leftarrow$, then there exists exactly one vertex $d'$ of the quiver $Q_T'$ such that locally, we have the following configuration: $a\rightarrow a'$ on the quiver $Q_T'$.
        \item  If $f=\rightarrow$, then there exist two vertices $d'$ and $d''$ of the quiver $Q_T'$, such that locally, we have the following configuration $d'\leftarrow a\rightarrow d''$ and the arcs associated to $d,d''$ and $c$ form a triangle on the surface $S$.
    \end{itemize}
    The \textit{extended string of $w$} is defined to be the string 
    \[
    w_\gamma'=a'aebucfdd'.    
    \]
\end{definition}

\begin{remark}
    We can define the extended loopstring of a loopstring in a similar way, by
just extending it to the side that does not contain the hook.
\end{remark}

\begin{definition}\label{monomofstring}
    Let $w= a_1 \dots a_n$ be a loopstring. The \textit{monomial associated to the loopstring $w$} is defined in the following way:
    \[
    x(w) = \prod_{i=1}^{n} x_{a_i}^{e_i},
    \]
    where, if $a_i$ corresponds to a simple in the socle of the module $M(w)$ and simultaneously is part of the hook of the loopstring, then $e_i= 1$, while otherwise $e_i$ denotes the number of incoming arrows to $a_i$.
\end{definition}

\begin{example}
Let us take: 
\[
w= 1\longrightarrow 2 \hookrightarrow3 \longleftarrow 4,
\] 
or equivalently:
\[
w = 
\begin{tikzcd}
       1 \arrow[dr] \\
      & 2 \arrow[dr] & & 4 \arrow[dl] \arrow[ll, bend right=45] \\
      & & 3
\end{tikzcd}
\]
Assume that the blossoming loopstring $w'$ is the following:
\[
 w'= 0\longleftarrow 1\longrightarrow 2 \hookrightarrow3 \longleftarrow 4.
\] 
or equivalently:
\[
\begin{tikzcd}
     &  1 \arrow[dl] \arrow[dr] \\
     0 & & 2 \arrow[dr] & & 4 \arrow[dl] \arrow[ll, bend right=45] \\
     & & & 3
\end{tikzcd}
\]
Then, following definition. \ref{monomofstring} we have that:
\[
x(w) = x_0 x_2^{2} x_3.
\]
\end{example}

\begin{lemma}
    Let $(S,M,P,T)$, be a triangulated punctured surface and $Q_T'$ the extended quiver associated to that surface. Assume that $\gamma$ is a plain arc, which does not belong to the triangulation $T$. Let $P_{max}$ be the maximal perfect matching of the associated graph $G_\gamma$ and $w_\gamma$ the associated loopstring of the arc $\gamma$. Then, if $x(P_{max})$ denotes the $x$-weight of the maximal perfect matching $P_{max}$, as it is classically defined, we have that
    \[
    x(w_\gamma) = x(P_{max}).
    \]
    \begin{proof}
        The proof of this Lemma is straightforward by comparing the definitions of the weight monomial in each setting.
    \end{proof}
\end{lemma}

\begin{remark}\label{removetoppreservesx}
    Definition \ref{monomofstring} is very useful, but unfortunately, when one associates a module to a different perfect matching (apart from the maximal), then the associated $x$ monomial of the loopstring is not the same as the associated $x$ monomial of that perfect matching. However, using the structure of the extended quiver $Q_T'$, and observing how the bijection between the submodule lattice of a loopstring and the perfect matching lattice of the associated loop graphs was constructed one can notice the following useful fact:\\
    Removing a top $t$ from the loopstring $w$ corresponds to switching the two edges on the corresponding quadrilateral in the associated perfect matching $P_{w}$. If we denote the new loopstring as $w'$ and compute $x(P_{w'})$, then we can see that:
    \begin{align*}  
    x(P_{w'}) = \frac{x(P_w) x_b x_d}{x_a x_c} 
    \end{align*} 
    where $a, c \in P_{w}\setminus P_{w'}$ and $b, d \in P_{w'}\setminus P_{w}$.
    We can notice now that locally in the quiver $Q_T'$ we have the configuration appearing in figure \ref{fig:local conf of top in extended quiver}.
    Therefore, by looking at the local configuration around $t$ in the quiver $Q_T'$ we can compute $x(w')$ given $x(w)$. Additionally, we can notice that since the new monomial depends only on the local configuration of the quiver $Q_T'$, if $x(w_1) = x(w_2)$, removing the same top $t$ from both modules $w_1$ and $w_2$ preserves the equality on the $x$-variables, i.e. $x(w_1\setminus t) = x(w_2\setminus t)$.
\end{remark}

\begin{figure}
    \centering
    \includegraphics[scale=0.7]{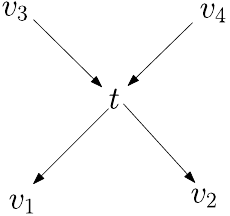}
    \caption{By removing the top $t$ from the string $w$ we obtain $x(P_{w'}) = \frac{x(P_w) x_{v_3} x_{v_4}}{x_{v_1} x_{v_2}}$}.
    \label{fig:local conf of top in extended quiver}
\end{figure}

Before stating the main Theorem of this section, which proves skein relation for a crossing between two possibly tagged arcs, we need to prove two more Lemmas.

\begin{lemma}\label{xvariablesagree1}
     Let $w_1, w_2, w_3, w_4, w_5$ and $w_6$ as they were defined in \ref{resofloopstrings}. Then:
     \[
     x(w_3\oplus w_4) = x(w_1\oplus w_2),
     \]
     and 
     \[
     x(w_5\oplus w_6) = x(w_1\oplus (w_2\setminus m))
     \]
     \begin{proof}
         We will show only the first equality, since the second one is also just a direct computation.\\
         By definition \ref{monomofstring} we can see that if $v_2\neq \emptyset$ then:
         \[
         x(w_3) = x(h_1u_1mv_2h_4) = \frac{x(h_1u_1mv_1h_2)}{x_{t(m)} x(v_1h_2)} x(v_2h_4) x_{s(v_2)},
         \]
         where $t(m)$ denotes the last letter of the string $m$ and $s(v_2)$ denotes the first letter of the string $v_2$, while if $v_2 = \emptyset$ then:
         \[
         x(w_3) = x(h_1u_1mh_4) = \frac{x(h_1u_1mv_1h_2)}{x(v_1h_2)} x(h_4),
         \]
         since then there is an extra term $x(t(m))$ due to the fact that $h_4$ is following $m$.\\ 
         %(look figure \ref{fig:m2 and h adj})
         Working now on the string $w_4$ we take:
         \[
         x(w_4) = x(h_3u_2mv_1h_2) = x(h_3u_2m) x_{t(m)} x(v_1h_2).
         \]
         Multiplying now $x(w_3)$ and $x(w_4)$, if $v_2\neq\emptyset$ we take the following:
         \[
         x(w_3) x(w_4) = x(h_1u_1mv_1h_2) x(v_2h_4) x_{s(v_2)} x(h_3u_2m) = x(h_1u_1mv_1h_2) x(h_3u_2mv_2h_4),
         \]
         which is exactly $x(w_1) x(w_2)$, while if $v_2=\emptyset$ we have:
         \[
         x(w_3) x(w_4) = x(h_1u_1mv_1h_2) x(h_4) x(h_3u_2m) x_{t(m)} = x(w_1) x(w_2).
         \]
     \end{proof}
\end{lemma}

\begin{theorem}
    Let $(S,M,P,T)$ be a triangulated punctured surface, $\gamma_1$ a singly notched arc on a puncture $p$ and $\gamma_2$ a plain arc. Suppose that $\gamma_1$ and $\gamma_2$ cross on the surface and let $(\gamma_3,\gamma_4)$ and $(\gamma_5,\gamma_6)$ be the two pairs of arcs obtained by smoothing the crossing.
    \begin{itemize}
        \item[(i)] 
        If  $\gamma_1$ and $\gamma_2$ have a non empty overlap, then:
        \[
        x_{\gamma_1} x_{\gamma_2} = x_{\gamma_3} x_{\gamma_4}  + Y^{+} x_{\gamma_5} x_{\gamma_6},
        \]
        where $Y^{+} =  Y(w_1\oplus (w_2\setminus m)) \setminus Y(w_5\oplus w_6)$.
        \item[(ii)]
        If $\gamma_1$ and $\gamma_2$ have an empty overlap and $a$ is the arrow appearing in definition \ref{grafofloopstrings}, then we have the following cases.
        \begin{itemize}
            \item[(1)] If $a =\leftarrow$, then:
            \[
            x_{\gamma_1} x_{\gamma_2} = x_{\gamma_3} x_{\gamma_4}  + Y^{+} x_{\gamma_5} x_{\gamma_6},
            \]
            where $Y^{+} =  Y(w_1\oplus (w_2\setminus r')) \setminus Y(w_5\oplus w_6)$.
            \item[(2)] If $a =\rightarrow$, then:
            \[
            x_{\gamma_1} x_{\gamma_2} = Y^{-} x_{\gamma_3} x_{\gamma_4}  + x_{\gamma_5} x_{\gamma_6},
            \]
            where $Y^{-} =  Y(w_1\oplus (w_2\setminus r)) \setminus Y(w_3\oplus w_4)$.
        \end{itemize}
    \end{itemize}

    \begin{proof}
        The proof of this theorem for the case $(i)$ (resp. $(ii)$)  is a direct consequence of Theorem \ref{bijofmod} (resp \ref{bijofmodgrafting}), Lemma \ref{xvariablesagree1} and Remark \ref{removetoppreservesx}.
    \end{proof}
\end{theorem}

\section{ Skein relations for crossing arcs and loops}
Our goal in this section is to prove skein relations in the case that there is a crossing between an arc and a loop on the surface. This covers the following two cases, as they were listed in \cite{MSWpositivity}:

\begin{itemize}
    \item [(i)] a singly notched arc and a loop,
    \item [(ii)] a doubly notched arc and a loop.
\end{itemize}

As in the case of a crossing between two arcs, there are two different possibilities when an arc crosses a loop. They can either cross with a non-empty overlap, or cross in an empty overlap, fro which case we will define a grafting between the loop and the arc. \\

The following Theorem, the proof of which will be covered later on this section deals with these two cases.

\begin{theorem}\label{skein for loop and arc}
        Let $(S,M,P,T)$ be a triangulated punctured surface, $\gamma_1$ a notched arc on a puncture $p$ and $l^{o}$ a loop. Suppose that $\gamma_1$ and $\gamma_l$ cross on the surface and let $\gamma_2$ and $\gamma_3$ be the two arcs obtained by smoothing the crossing.\\

        If  $\gamma_1$ and $\gamma_l$ have a non empty overlap, then:
        \[
        x_{\gamma_1} x_{l^{o}} = x_{\gamma_2}  + Y^{+} x_{\gamma_3},
        \]
        where $Y^{+} =  Y(m)$.
\end{theorem}

\subsection{ Resolution of crossing between an arc and a loop, with non empty overlap}
As it was the case with the previous chapter, suppose that $(S,M,P)$ is a triangulated punctured surface, and $\mathcal{A}$ is the Cluster algebra associated to it. \\

Suppose that there is a crossing between an arc $\gamma_1$ and a loop $\beta$ on the surface, where $\gamma_1$ is notched tagged in at least one of its endpoints. We can then associate a loopstring $w_1$ to $\gamma_1$ and a band string $l$ to $\beta$.\\
In \cite{canakci2} the authors resolved a crossing between a pain arc and a loop, by constructing a bijection of snake and band graphs. In that case the resolution produces two new plain arcs. A similar thing holds, when the arc is not assumed to be plain, as it was to be expected. In the following definition, we resolve the crossing of such two arcs, when viewed as loopstrings or band strings instead. 

\begin{definition}\label{cross of arc and loop}
    Suppose that $w_1$ and $l^{o}$ are two abstract loopstrings and band strings respectively, where $w_1$ possible contains a hook at its start or its end. Assume that $w_1= h_1u_1mh_2$ and $l^{o}= bmau_2b$ where $u_i, m$ are substrings, $a, b$ are arrows (which are opposite) and $h_i$ are hooks. We then say that $w_1$ and $l^{o}$ cross in m.\\
    
    The \textit{resolution of the crossing of $w_1$ and $l^{0}$ with respect to $m$} are the loopstring $w_2$ and $w_3$ which are defined as follows:
    \begin{enumerate}[(i)]
        \item $w_2 = h_1u_1mu_2mh_2$,
        \item If  $u_1\neq \emptyset$ then $w_3 = h_1u_1(u_2^{-1})'h_2$,\\
        where $(u_2^{-1})'$ is such that $u_2^{-1} = (u_2^{-1})' s$ where $s$ is a maximal sequence of direct arrows (inverse arrows) if $a=\rightarrow$ (resp. $a=\leftarrow$).\\
        If $u_1= \emptyset$ then $w_3 = h_1(u_2^{-1})''h_2$,\\
        where $(u_2^{-1})''$ is such that $u_2^{-1} =k (u_2^{-1})' s$ where $s$ is a maximal sequence of direct arrows (inverse arrows) if $a=\rightarrow$ (resp. $a=\leftarrow$) and $k$ is a maximal sequence of inverse arrows (direct arrows) if $a=\rightarrow$ (resp. $a=\leftarrow$).\\
    \end{enumerate}
    
\end{definition}

\begin{remark}\label{hook is also part of the overlap}
    Let us notice that in the above definition \ref{cross of arc and loop} we did not make an assumption that there might exist a loopstring in between the overlap $m$ and the hook $h_2$. We decided to not include this case in the definition since we care about loopstrings and bands associated to tagged arcs and band arcs on a surface respectively, and such behavior cannot occur on such a surface, since loops are also considered up to isotopy. \\
    We also need to elaborate a little bit more, regarding the local configuration of $m$ and $h_2$ in the loopstring $w_1$. Since the substring $m$ is attached to the hook $h_2$, locally there are two arrows connecting the last letter of $m$  and the hook $h_2$. One incoming and one outgoing arrow. Therefore, since the last letter of $m$ is also connected to the substring $u_2$ in $l^{o}$ through the arrow $a$, this arrow $a$ as well as a substring of $u_2$ are part of the hook $h_2$. This means that we could possibly define the overlap of the intersection to be a bigger substring that the one appearing in \ref{cross of arc and loop}. However, we opted to not define like this, since by the definition of the loop graphs and the tagging on a puncture, the hook which was introduced by Wilson \cite{FST} is artificial and does not appear on the surface. This important detail will come into the spotlight in the proof of proposition \ref{remove tops from tagged and loop}.
\end{remark}

The next two examples showcase the two different cases that appear in definition \ref{cross of arc and loop}. One may argue that the definition of $k$ and $s$  in the previous definition is superfluous, since in the cases that $h_1$ and $h_2$ are non-empty, the arcs in $k$ and $s$ are part of the hooks. However, the above definition should be a generalization of the classic case in which there are no tagged arcs, and in those cases the exemption of the arcs in $k$ and $s$ is needed. 

\begin{example}
    Suppose that $w_1 = 1 \longleftarrow 9 \longrightarrow 8 \hookleftarrow 7 \longrightarrow 5 \longrightarrow 6$ and  \\ $l = \longrightarrow 9 \longrightarrow 8 \longleftarrow 7 \longrightarrow 5 \longleftarrow 4 \longrightarrow 3 \longrightarrow 2 \longrightarrow$.\\ 
    Then, following definition \ref{cross of arc and loop}, we have that $h_1=\emptyset$, $u_1= 1$, $m= 9 \longrightarrow 8$, $u_2= 7 \longrightarrow 5 \longleftarrow 4 \longrightarrow 3 \longrightarrow 2$, $a=\longleftarrow$ and $h_2 = \longrightarrow 7 \longrightarrow 5 \longrightarrow 6 \longrightarrow$.\\
    The resolution of the crossing of the loopstring $w_1$ and the loop $l$ in $m$ is the following:
    \[
    w_2 = 1 \longleftarrow 9 \longrightarrow 8 \longleftarrow 7 \longrightarrow 5 \longleftarrow 4 \longrightarrow 3 \longrightarrow 2 \longrightarrow 9 \longrightarrow 8 \hookleftarrow 7 \longrightarrow 5 \longrightarrow 6.
    \]
    \[
    w_3= 1 \longrightarrow 2 \longleftarrow 3 \longleftarrow 4 \hookrightarrow 5 \longleftarrow 7 \longleftarrow 6.
    \]
\end{example}

\begin{proposition}\label{remove tops from tagged and loop}
    Let $w_1, w_2, w_3$ and $l$ as they were defined in \ref{cross of arc and loop}. Suppose that $m_1$ (resp. $m_2$) is a proper submodule of $w_2$ (resp. $w_3$). If $m_1$  (resp. $m_2$) is reached from $w_1$ (resp. $w_2$) by the sequence $(a_i), 1\leq i\leq n$, then we can apply the same sequence $(a_i)$ to the module $w_1\oplus l^o$ (resp. $w_1\oplus (w_2\setminus m)$) to obtain a submodule $m_1'$ (resp. $m_2'$) of it. 
    \begin{proof}
        The proof of this proposition will follow the steps of Prop.\ref{removertopwelldefined} and will be done by induction. We will also assume that the arrow $a$ as it was defined in \ref{cross of arc and loop} is an inverse arrow.\\
        Suppose that $m_1$ is a submodule of $w_2$. If only one top was removed from $w_2$ to generate the submodule $m_1$ it is clear by the construction of $w_2$ and the shape of $w_1$ and $l^o$ that the simple which was removed, is also a top in one of the modules associated to $w_1$ or $l^o$ and therefore removing it gives us a submodule of the module $w_1\oplus l^o$.\\
        Assume now that $m_1$ is reached from $w_2$ by removing a sequence of simples $(a_i), 1\leq i\leq n$ (which is, as we have stated already non-unique). By the inductive assumption, removing the sequence $(a_i), 1\leq i\leq n-1$ from $w_1\oplus l^o$ generates a submodule of this module. In order to showcase how the induction works we first need to do the following construction.\\
        \underline{Construction:}\\
        If $a_i$ for $1\leq i\leq n$ is simple in $h_1$ or $u_1$ viewed as substrings of $w_2$ then there is a unique simple $a_i$ (up to embedding of submodules) in $w_1\oplus l^o$, and therefore we remove that unique one.\\
        If $a_i$ is in $m$, there may be two choices depending on which $a_i$ we can remove from $w_1\oplus l^o$ since $m$ appears in both. Additionally, notice that due to the local behavior of the hook $h_2$ (as it was discussed in Remark\ref{hook is also part of the overlap}), if $a_i$ is in $h_2$ or $u_2$, it may happen that the simple $a_i$ is also part of the hook $h_2$ or the substring $u_2$ making the choice of which simple to be removed from $w_1\oplus l^o$ non-unique.\\
        In those cases, we decide the following: If $a_i$ is the simple in $m$, removed from the first copy of $m$ in $w_2$, then we remove $a_i$ from $w_1$, while if the opposite happens, we remove it from $l^o$, if possible. If $a_i$ is in $h_2$ then we remove it from $h_2$ viewed as the hook of $w_1$ while if $a_1$ is in $u_2$ we remove it from $u_2$ viewed as a substring of $l^o$, if possible.\\
        Nevertheless, it may happen that we are not able to apply the previous construction. If such a removal is not possible, we apply the following exception to the rule:\\
        Suppose that $m^l$ is the last letter of the string $m$ and $a_n= m^l$ removed from the first copy of $w_2$. Since $a$ is an inverse arrow (making $m_l$ not a local top in $w_2$), $u_2^1$ which denotes the first letter of $u_2$ must have been one of the simples in the sequence $a_i$ for $1\leq i\leq n-1$. Therefore, following the steps of the previous construction it should have been removed from $u_2$ viewed as a substring of $l^o$. This would lead to a contradiction, when trying to remove $m^l$ from $w_1$ since it is still not a local top. However, we know that $u_2^1= h_2^1$ or $u_2^1=h_2^l$ where $h_2^1$ (resp. $h_2^l$) denotes the first (resp. last) letter of the hook $h_2$. Therefore, by retracing the removal of tops and removing $u_2^1$ from the hook $h_2$ now, we take that $m^l$ is a local top in $w_1$ and can be removed to produce a submodule of $w_1\oplus l^o$. \\
        The above construction also proves the validity of the induction step, since we can now always remove $a_n$ from $w_1\oplus l^o$\\
        The proof fro $m_2$ being a submodule of $w_3$ is omitted since it is an easy induction following the same steps as the earlier proof without the need of defining an exception, since by construction the choice of which simple to remove is always unique.
    \end{proof}
\end{proposition}

\begin{remark}
    The only reason for making the assumption in the previous proof that the arrow $a$ is considered to be an inverse arrow, is the fact that we need to be precise on when the exception mentioned in the proof occurs. If the arrow $a$ was considered to be a direct arrow, we would just have to adjust part of the proof, since now the exception would occur when we attempt to remove the last letter of the second copy of $m$ in $w_2$. 
\end{remark}

\begin{theorem}\label{bij of mod arc and band}
     Let $w_1, w_2, w_3$ and $l$ as they were defined in \ref{cross of arc and loop}. Then there exists a bijection:
    \[
    \Psi \colon SM(w_2) \bigcup SM(w_3) \to SM(w_1\oplus l),
    \]
    which is defined as follows:\\
    \begin{enumerate}
        \item If $a = \longleftarrow$ then:
          \begin{gather*}
          \Psi(w_2) = w_1\oplus l, \\
          \Psi(w_3) = \begin{cases}
                 (h_1u_1)\oplus (h_2\setminus s)  \oplus l & \text{if $u_1= \neq \emptyset$}.\\
                 (h_1\setminus k)\oplus (h_2\setminus s)  \oplus l & \text{if $u_1 = \emptyset$}.\\
             \end{cases}  
          \end{gather*}
        \item If $a = \longrightarrow$ then:
        \begin{gather*}
          \Psi(w_2) = w_1\oplus l, \\
          \Psi(w_3) = w_1 \oplus u_2,     
        \end{gather*}
    \end{enumerate}
    If $w$ is a submodule of $w_2$  which is reached by the sequence $\{a_i\} 1\leq i\leq n$, and $w'$ is the submodule of $w_1\oplus w_2$ which is reached by the same sequence, then we define:
    \[
    \Psi(w) = w'.
    \]
    If $w$ is a submodule of $w_3$, then $\Psi(w)$ is defined in a similar way. 
\end{theorem}

We will now prove that the map $\Psi$ that was defined in \ref{bij of mod arc and band} is an injection and the surjectivity will follow in a similar way as it was the case with the crossing of two arcs.

\begin{proposition}\label{inj for band and arc}
    The map $\Psi$ as it was defined in \ref{bij of mod arc and band} is an injection.
    \begin{proof}
        Let $m_1, m_2\in SM(w_2) \bigcup SM(w_3)$ such that $\Psi(m_1)=\Psi(m_2)$.\\
        By construction of the map $\Psi$, if both $m_1, m_2\in SM(w_2)$ or $m_1, m_2\in SM(w_3)$, then $m_1= m_2$.\\
        Suppose that $m_1\in SM(w_2)$ and $m_2\in SM(w_3)$. We will show that $\Psi(m_1)\neq \Psi(m_2)$. \\
        Assume that $\Psi(m_1)=\Psi(m_2)$. We need to take cases:
        \begin{itemize}
            \item Suppose that $a = \longrightarrow$. Since $m_1\in SM(w_2)$, we know by construction of the map $\Psi$, that there exists a  sequence $(a_i), 1\leq i\leq n$ of simples, such that when viewed as a loopstring $w_2 \setminus \{a_i\} = m_1$. By proposition \ref{removertopwelldefined} and by construction of $\Psi$, we have that $\Psi(m_1) = \Psi(w_2)\setminus \{a_i\}$.\\
        Similarly, working on $m_2$, there is a sequence $(b_i), 1\leq i\leq k$, such that $m_1= (w_3)\setminus \{b_i\}$ and $\Psi(m_2) = \Psi(w_3)\setminus \{b_i\}$.\\
        However, since $\Psi(m_1)=\Psi(m_2)$ we obtain:
        \[
        \Psi(w_2)\setminus \{a_i\} = \Psi(w_3)\setminus \{b_i\}.
        \]
        Since  $\Psi(w_2) = w_1\oplus l$, and $\Psi(w_3) = w_1\oplus u_2$, we further obtain that:
        \[
         (w_1\oplus l)\setminus\{a_i\} =  (w_1\oplus u_2)\setminus \{b_i\}.
        \]
        The previous equality shows us that the sets $\{a_i\}$ and $(\{b_i\}\bigcup m)$ must be equal. By Remark \ref{backtrackinginlattice}, we can deduce, that there must exist a sequence $\{m_{j_1}, \dots m_{j_l}, b_1, \dots, b_k\}$ leading from $w_1\oplus l$ to $\Psi(m_2)$. However, this also implies that using the same sequence $\{m_{j_1}, \dots m_{j_l}, b_1, \dots, b_k\}$, we can go from the module $w_2$ to the module $m_2$. We will show that this is a contradiction. Notice fist that $w_2$ contains two copies of the string $m$. it is easy to see that we cannot remove a whole copy of $m$ from $w_2$ since neither copy of $m$ is a local top of $w_2$. Suppose now that $\{m_{j_1}, \dots m_{j_k}\}$ , where $k<l$, is removed from the first copy of $m$ in $w_2$ and $\{m_{j_{k+1}}, \dots m_{j_l}\}$ is removed from the second copy of $m$ in $w_2$. By the shape of $w_2$ obviously $s(m)$ must be removed from the second copy of $m$. Furthermore suppose that $m'$ is the connected substring of $m$ which contains $s(m)$ and contains a maximal sequence of direct arrows in $m$. Then we can also see that $t(m')$ must be removed from the second copy of $m$ in $w_2$. Additionally since this is a maximal sequence of direct arrows we also take that locally in $m$ we have that $t(m')\leftarrow s(m\setminus m')$. Therefore, since $t(m')$ is removed from the second copy of $m$, $s(m\setminus m')$ must also be removed from the second copy of $m$. Inductively, working on maximal substrings of $m$ that contacting only let arrows or right arrows, we take that every $m_i$, $1\leq i\leq l$, must be removed from the second copy which is a contradiction. 
        \item Suppose that $a = \longleftarrow$. Similar arguments hold in that case too, since the position of the substring $m$ leads to a similar contradiction as before.
        \end{itemize}
    \end{proof}
\end{proposition}

\begin{proof}[Proof of Theorem \ref{bij of mod arc and band}]
It follows straightforwardly using \ref{inj for band and arc} and \ref{termsincoeffreecase} that the map $\Psi$ is a bijection.

\end{proof}

\subsection{ Monomials associated to loops}

So far we have proven the existence of a bijection between the submodules of the direct sum of a loopstring and a band string and the disjoint union of the resolved modules. However we still need to show that the monomials associated to these modules are equal. 

\begin{remark}
    Suppose that $l$ is a band. Then we can associate a monomial $x(l)$ to this band, which follows exactly the same rules as they were dictated in definition \ref{monomial associated to loopstring}. Let us note though that in the case of bands, we are not defining an analogue of a blossoming string.
\end{remark}

\begin{lemma}\label{xvariablesagree}
     Let $w_1, w_2, w_3$ and $l$ as they were defined in \ref{cross of arc and loop}. Then:
     \[
     x(w_2) = x(w_1\oplus l^{o}),
     \]
     and 
          \begin{gather*}
          x(w_3)\prod_{i} x_{m(i)} = \begin{cases}
                 x(w_1\oplus u_2)& \text{if } a=\rightarrow.\\
                 x((h_1u_1)\oplus (h_2\setminus s)  \oplus l) & \text{if } a=\leftarrow \text{ and } u_1 \neq \emptyset.\\
                 x((h_1\setminus k)\oplus (h_2\setminus s)  \oplus l) & \text{if } a=\leftarrow \text{ and } u_1 = \emptyset.\\
             \end{cases}  
          \end{gather*}
     \begin{proof}
         We will show only one of these equalities since the proof for the others is very similar.\\
         Let us assume that $a=\leftarrow$ and $u_1\neq \emptyset$.
         Isolating the factor $x(l^{o})$ we obtain:
         \[
         x(l^{o})= x(mu_2)= x(m)x(u_2)x_{s(m)}x_{t(m)},
         \]
         where $s(m)$ and $t(m)$ denote the first and last letter of the string $m$ respectively and the equality follows since $a=\leftarrow$ and therefore there ate two extra arrows in $l^{o}$ that point at the start and respectively at the end of the string $m$.\\
         Similarly working on the term corresponding to the band $w_1$ we obtain:
         \[
         x(w_1) = x(h_1u_1)x_{t(u_1)} x(mh_2).
         \]
         Simplifying on the left hand side we have:
         \[
         x(w_2)= x(h_1u_1u_2^{-1}h_2) = x(h_1u_1) x_{t(u_1)} x(m) x_{t(m)} x(u_2^{-1}) x_{s(m)} x(mh_2) = x(w_1) x(l^{o}),
         \]
         which concludes the proof.
     \end{proof}

\end{lemma}

The following theorem is a direct consequence of the previous Lemma~\ref{xvariablesagree} and the existence of the bijection $\Psi$ from Theorem~\ref{bij of mod arc and band}

\begin{theorem}
        Let $(S,M,P,T)$ be a triangulated punctured surface, $\gamma_1$ a notched arc on a puncture $p$ and $l^{o}$ a loop. Suppose that $\gamma_1$ and $\gamma_l$ cross on the surface and let $\gamma_2$ and $\gamma_3$ be the two arcs obtained by smoothing the crossing.\\

        If  $\gamma_1$ and $\gamma_l$ have a non empty overlap, then:
        \[
        x_{\gamma_1} x_{l^{o}} = x_{\gamma_2}  + Y^{+} x_{\gamma_3},
        \]
        where $Y^{+} =  Y(m)$.

\end{theorem}

\section{Skein relations for single incompatibility in a puncture}
So far, we have explored what happens when there is a regular crossing between two arcs. The previous cases could be understood as generalizations of the classical skein relations on unpunctured surfaces, since even if the arcs were not tagged, the arcs are considered incompatible to each other.\\
However starting with this chapter and onward, the cases that we will explore appear only on punctured surfaces since those cases will deal with incompatibilities that occur between two arcs at a puncture.\\
In this chapter more specifically we will deal with the following cases:

\begin{itemize}
    \item a singly notched arc and a plain arc which have an incompatible tagging at a puncture,
    \item a doubly notched arc and a plan arc which have an incompatible tagging at a puncture,
    \item two singly notched arcs which have one incompatible tagging at a puncture.
    \item a doubly notched arc and a singly notched arc which have one incompatible tagging at a puncture.
    %{\color{purple} I need to make sure where this case precisely belongs and if there are substantial differences to include it separately.}
\end{itemize}

The reason, for grouping these cases together is the fact that only the incompatibility of the tagging at a single puncture plays a role to the skein relations. The cases that there are two incompatibilities at two punctured are explored later, since then the situation is drastically different. \\
The following theorem which addresses the above cases will be proven at the end of the chapter.

\begin{theorem}
        Let $(S,M,P,T)$ be a triangulated punctured surface, $\gamma_1$ a notched arc at a puncture $p$ and $\gamma_2$ an arc which has an plain endpoint at the puncture $p$. Let $\gamma_3$ and $\gamma_4$ be the two arcs obtained by smoothing the crossing.
        Then we have that:
        \[
        x_{\gamma_1} x_{\gamma_2} = x_{\gamma_3}  + Y^{+} x_{\gamma_4},
        \]
        where $Y^{+}$ is a monomial on $y$-variables.
\end{theorem}

\subsection{ Resolution of single incompatibility of two arc at a puncture}

As we have seen in the previous cases already, given two arcs which have exactly one incompatibility at one puncture, there are two different cases for the resolutions in the skein relations depending on the triangulation of the surface. \\
We will separate these two cases, calling the first one, which will be dealt with in this chapter as \textit{regular incompatibility} while we will call the second case, which will be dealt in the next section \textit{grafting incompatibility}, since it poses some similarities to the grafting for regular crossing of arcs.\\

Suppose now that $(S, M, P)$ is a punctured surface and $T$ is a triangulation of the surface. Let $\gamma_1$ and $\gamma_2$ be two arcs which have exactly one incompatible tagging at a puncture $p$, and $w_1, w_2$ the associated loopstring to each arc respectively. We will say that $\gamma_1$ and $\gamma_2$ have a \textit{regular incompatibility} at the puncture $p$ when the arcs do not cross the same arc $t$ of the triangulation $T$ before meeting at the puncture $p$ (Figure \ref{fig:regular incompatibility puncture}). \\
looking at the associated loopstrings or modules of the arcs $\gamma_1$ and $\gamma_2$ one can notice that we would not be able to detect when these arcs meet at a puncture, since it can happen that they do not share any crossings. However since there is an incompatibility at that puncture, we know that at least one of these two arcs it tagged at its endpoint at the puncture $p$, o its loopstring should contain a hook $h$ which s all the arcs of the triangulation $T$ which are adjacent to the puncture p. Therefore if we assume that $\gamma_1$ is tagged at $p$ and $\gamma_2$ is untagged at $p$, by considering the extension $w_2'$ of the loopstring $w_2$, we can notice that $w_1$ and $w_2'$ share at least one vertex, which must belong to the hook $h$ of $\gamma_1$. \\
The following definition explains how we can resolve such an incompatibility and Figure \ref{fig:regular incompatibility puncture} showcases the situation that we just explained.

\begin{figure}
    \centering
    \includegraphics[scale=0.9]{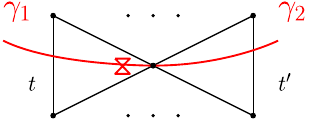}
    \caption{Regular incompatibility at a puncture.}
    \label{fig:regular incompatibility puncture}
\end{figure}

\begin{definition}\label{resolution of incompatible tagging  one puncture}
 Suppose that $w_1$ and $w_2$ are two abstract loopstrings such that $w_1$ contains a hook at its end, while $w_2$ does not contain a hook at its end. Assume that $w_1= h_1u_1h_2$ and $w_2= h_3u_2$ where $u_i$ are substrings and $h_i$ are hooks, with $h_2\neq \emptyset$. Let also $w_2'= h_3u_2k$ be the extended loopstring of $w_2$ with $k\in h_2$ We then say that $w_1$ and $w_2$ have a \textit{regular incompatibility} at their endpoints.\\

Since $k\in h_2$ we can write $h_2= s_1aks_2$ where $s_i$ are not necessarily non empty strings and $a$ is an arrow. 
The \textit{resolution of the regular incompatibility of $w_1$ and $w_2$ with respect to $k$} are the loopstrings $w_3$ and $w_4$ which are defined as follows:
    \begin{enumerate}[(i)]
        \item  $w_3 = h_1u_1s_1bu{_2}^{-1}h_2^{-1}$, \textit{where $b$ is an opposite arrow to $a$}.
        \item  $w_4 = h_1u_1s_2cu{_2}^{-1}h_2^{-1}$, \textit{where $b= a$.}
    \end{enumerate}
    
\end{definition}

\begin{remark}
    Looking at the above Definition \ref{resolution of incompatible tagging  one puncture}, one can notice that the resolution is defined uniquely for every situation of regular incompatibility, and there is no need for extra cases. This occurs since visually the resolution basically "glues" the two arcs $\gamma_1$ and $\gamma_2$ to create two new arcs $\gamma_3$ and $\gamma_4$ which follow $\gamma_3$ and $\gamma_4$ around the puncture in the two possible ways that one can follow. This can be further illustrated in Figure \ref{fig:Resolution of regular incompatibility at a puncture p}, on which also the next example is based on. 
\end{remark}

\begin{figure}
    \centering
    \includegraphics[scale=0.35]{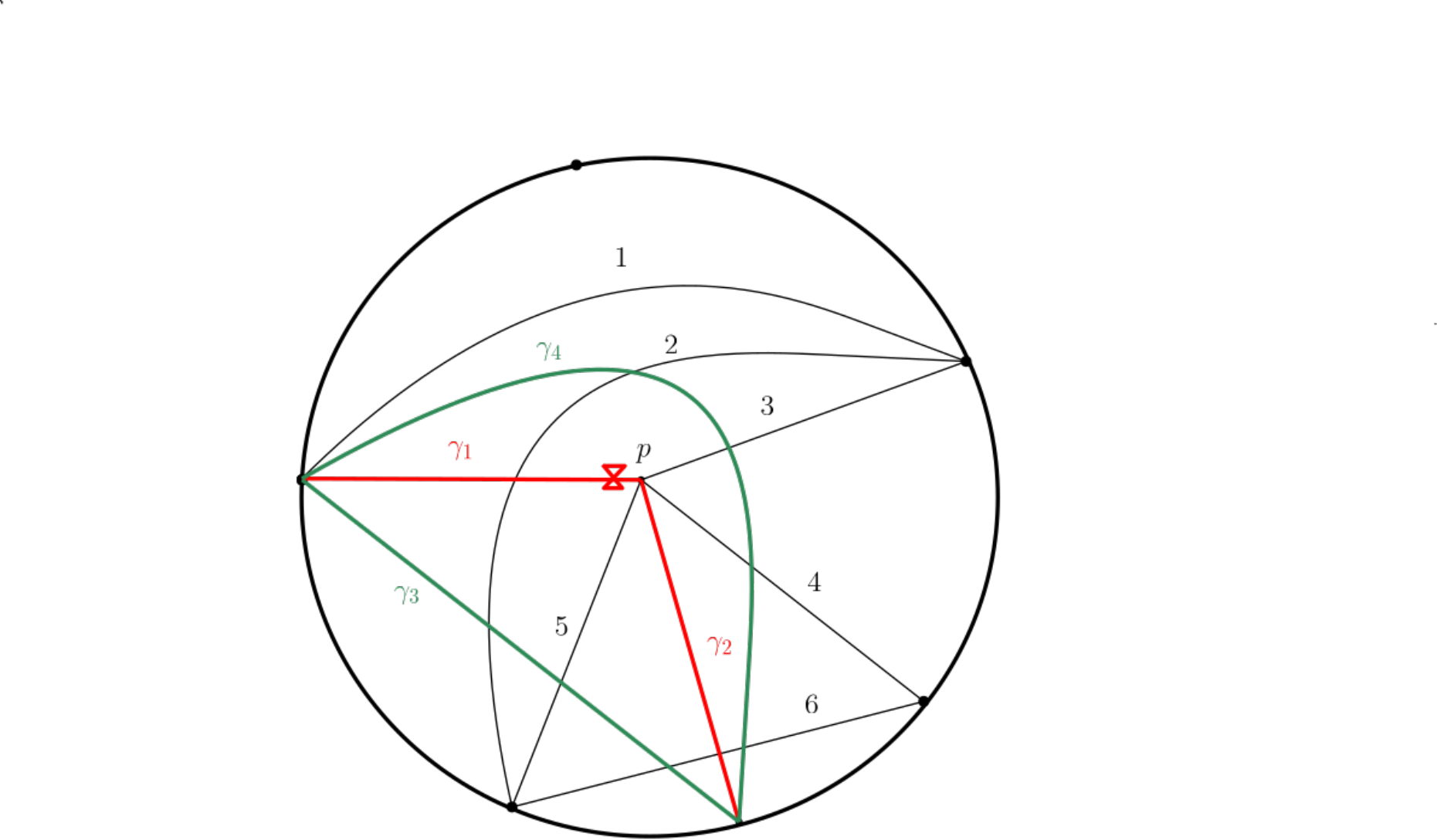}
    \caption{Resolution of regular incompatibility at a puncture p.}
    \label{fig:Resolution of regular incompatibility at a puncture p}
\end{figure}

\begin{example}
    Suppose that $(S, M, P, T)$ is the triangulated surface depicted on figure \ref{fig:Resolution of regular incompatibility at a puncture p}. Then we have that $w_1 = 2\hookrightarrow 5\leftarrow 4\leftarrow 3$ and $w_2 = 6$. Considering the extension of $w_3$ we take $w_3' = 6 \rightarrow 4$. Therefore, so far we have $h_1=\emptyset$, $u_1= 2$, $h_2 = 5\leftarrow 4\leftarrow 3$, $h_3= \emptyset$,
    $u_2= 6$ and $k=4$. \\
    Notice that indeed $k\in h_2$ and we can write $h_2= s_1aks_2$ where $s_1= 5$, $s_2= 4\leftarrow 3$ and $a= \leftarrow$.\\
    Using Definition \ref{resolution of incompatible tagging  one puncture} we can resolve the regular incompatibility with respect to $k= 5$, by defining:
    \[
    w_3= 2\rightarrow 5 \rightarrow 6,
    \]
    \[
    w_4= 2\leftarrow 3\rightarrow 4\leftarrow 6.
    \]
\end{example}

\begin{remark}
    One can notice that the hooks $h_1$ and $h_3$ can be either assumed to be empty or not without anything substantial changing to the definition of the resolution. This happens, since as one can see by the definition of the resolution, we just follow the arc $\gamma_1$ until it "reaches" the puncture $p$ and then we follow the path of $\gamma_2$ by either turning left or right at the puncture $p$. Therefore, the resolution depends on the local neighborhood of the puncture $p$ and the triangulation of the surface around the puncture, so it does not come as a surprise that possible other taggings or equivalently hooks would change anything.
\end{remark}

Our next aim, after having given the definition of the resolution is to actually construct a bijection between the submodules of $w_1\oplus w_2$ and the submodules of $w_3$ and $w_4$. \\
In order to do so we first have to give a proposition which will help us iteratively define the bijection as we have already done in the previous cases too.

\begin{lemma}\label{removing tops for one puncture incompatibility}
 Let $w_1, w_2, w_3$ and $w_4$ as they were defined in \ref{resolution of incompatible tagging  one puncture}. Suppose that $m_1$ (resp. $m_2$) is a proper submodule of  $w_3$ (resp $w_4$). If $m_1$  (resp. $m_2$) is reached from $w_3$ (resp. $w_4$) by the sequence $(a_i), 1\leq i\leq n$, then we can apply the same sequence $(a_i)$ to the module $w_1\setminus s_2\oplus w_2$ (resp. $w_1\oplus w_2$) to obtain a submodule $m_1'$ (resp. $m_2'$) of it.
    \begin{proof}
    Suppose that $m_2$ is a submodule of $w_4$ and it is reached by the sequence $(a_i), 1\leq i\leq n$. First of all, since $k$ belongs to $w_2'$ by definition we have that there exists a right arrow $t(w_2)\rightarrow k$. Therefore, by the construction of $w_4$, $k$ is a socle. This shows us that we can split the sequence $(a_i), 1\leq i\leq n$ w.l.o.g. to two parts, where the the sequence $(a_i), 1\leq i\leq j$ is a sequence of removing tops from $u_2^{-1}$ and $(a_i), j+1\leq i\leq n$ is a sequence of removing tops from the rest of $w_4$. The only extreme case, is the case that $k$ is an element of this sequence, but in order for this to happen we should have that all the elements of the string $s_2$ are also part of the sequence $(a_i), j+1\leq i\leq n$. \\
    The previous argument is enough to show that we can also apply this same sequence to the module $w_1\oplus w_2$, since the two subsequences are independent of each other. \\

    Suppose now that $m_2$ is a submodule of $w_3$ and it is reached by the sequence $(a_i), 1\leq i\leq n$. Notice that we want to show that we can apply the same sequence not the the whole module $w_1\oplus w_2$, but to the submodule of it $w_1\setminus s_2\oplus w_2$. Additionally, it is easy to see that $w_1\setminus s_2= u_1\rightarrow s_1$. Therefore, we can again split the sequence $(a_i), 1\leq i\leq n$ into two smaller ones, $(b_i)=(a_i), 1\leq i\leq j$ and $(c_i)=(a_i), j+1\leq i\leq n$ for some $0\leq j\leq n$, where $(b_i)$ consists of vertices only of the string $u_1\rightarrow s_1$ while $(c_i)$ consists of vertices only of the string $u_1\rightarrow s_1$. It is important to order them in that way, since any top of the substring $u_1\rightarrow s_1$ in $w_3$ is a top, but not necessarily any top of $u_2^{-1}$ is a top in $w_3$. 
    \end{proof}

\end{lemma}

The above lemma, is very important, since it shows us that there is a way to construct a map between the sets that we want to. However, this is not enough, since we want to define that map in a unique way, which will actually preserve the information lost from passing to a loopstring given a loop graph. 

\begin{proposition}\label{unique remove tops incompatibility one puncture}
    Let $w_1, w_2, w_3$ and $w_4$ as they were defined in \ref{resolution of incompatible tagging  one puncture}. Suppose that $m_1$ (resp. $m_2$) is a proper submodule of  $w_3$ (resp $w_4$). If $m_1$  (resp. $m_2$) is reached from $w_3$ (resp. $w_4$) by the sequence $(a_i), 1\leq i\leq n$, then we can apply the same sequence $(a_i)$ to the module $w_1\setminus s_2\oplus w_2$ (resp. $w_1\oplus w_2$) to obtain a unique submodule $m_1'$ (resp. $m_2'$) of it. 
    \begin{proof}
        the uniqueness of the modules $m_1'$ and $m_2'$ follows straightforwardly by the construction of the modules as it was done in lemma \ref{removing tops for one puncture incompatibility}, by noticing that the two subsequences of $(a_i), j+1\leq i\leq n$ as they were defined in the proof were independent of each other for both cases.
    \end{proof}
\end{proposition}

Having established that there is a unique way of passing from submodules of $w_3$ and $w_4$ to submodules of $w_1\oplus w_2$ we are now ready to construct the desired bijection. 

\begin{theorem}\label{bijection of modules incompatibility one puncture}
    Let $w_1, w_2, w_3$ and $w_4$ as they were defined in \ref{resolution of incompatible tagging  one puncture}. Then there exists a bijection:
    \[
    \Psi \colon SM(w_3) \bigcup SM(w_4) \to SM(w_1\oplus w_2),
    \]
    which is defined as follows:
    \begin{gather*}
    \Psi(w_3) = (w_1\setminus s_2)\oplus w_2, \\
    \Psi(w_4) = w_1\oplus w_2,     
    \end{gather*}
    If $w$ is a submodule of $w_3$  which is reached by the sequence $\{a_i\} 1\leq i\leq n$, and $w'$ the unique submodule of $(w_1\setminus s_2)\oplus w_2$ which is reached by the same sequence, then we define:
    \[
    \Psi(w) = w'.
    \]
    If $w$ is a submodule of $w_4$, then $\Psi(w)$ is defined in a similar way.
    
\end{theorem}

As the reader may already be familiar by now, we only need to prove that the map defined on Theorem \ref{bijection of modules incompatibility one puncture} is an injection, since then the bijectivity will follow from Remark \ref{termsincoeffreecase}.\\

\begin{proposition}\label{injectivity of Psi incompatibility one puncture}
    The map $\Psi$ as it was defined in Theorem \ref{bijection of modules incompatibility one puncture} is an injection.
    \begin{proof}

        Let $m_1, m_2\in SM(w_3) \bigcup SM(w_4)$ such that $\Psi(m_1)=\Psi(m_2)$.\\
        By construction of the map $\Psi$, if both $m_1, m_2\in SM(w_3)$ or $m_1, m_2\in SM(w_4)$, then $m_1= m_2$.\\
        We will now show that if $m_1\in SM(w_3)$ and $m_2\in SM(w_4)$ then $\Psi(m_1)\neq \Psi(m_2)$. Assume that $\Psi(m_1)=\Psi(m_2)$. Since $\Psi(m_1)=\Psi(m_2)$ and $\Psi(m_1)\subset w_1\setminus s_2$, by construction of $\Psi$ if $(b_i), 1\leq i\leq n$ is a sequence of the simples removed to reach $\Psi(m_2)$, we must have that all the elements of the string $s_2$ must be contained in that sequence $(b_i)$. Since there exists an arrow $k\leftarrow s(u_2^{-1})$ we deduce that $s(u_2^{-1})$ must also be an element of this sequence.\\
        Assume now that $(a_i), 1\leq i\leq n$ is a sequence of the simples removed from $m_1$ to reach $\Psi(m_1)$.
        Since $s(u_2^{-1})$ is part of the sequence $(b_i)$, it must be also part of the sequence $(a_i)$, since otherwise we could not have $\Psi(m_1)=\Psi(m_2)$. However, this in turn means that $t(u_1)$ as well as the whole string $s_1$ must belong to the sequence $(a_i)$, since there is a sequence of direct arrows starting from $t(u_1)$ and going at least until $s(u_2^{-1})$. The above arguments gives  us that the module $\Psi(m_1)=\Psi(m_2)$ does not contain any element of the hook $h_3$. However this means that a submodule of $w_4$ is mapped to a module which does not contain any element of the hook $h_3$ which is impossible by the construction of the map $\Psi$ and the fact that the elements of the string $s_1$ are not in $w_4$. (Basically submodules of $w_4$ are always mapped to elements which contain the whole string $s_1$).\\
    \end{proof}
\end{proposition}

\begin{proof}[Proof of Theorem \ref{bijection of modules incompatibility one puncture}]
The fact that $\Psi$ is a bijection follows immediately from the fact that $\Psi$ is an injective map (Proposition \ref{injectivity of Psi incompatibility one puncture}) and Remark \ref{termsincoeffreecase}.

\end{proof}

\subsection{ Resolution of grafting incompatibility of two arcs at a puncture}

So far in this chapter we have covered the resolution of an incompatibility of two arcs at a puncture when they possibly do not have any common intersections to the arcs of the given triangulation. However the situation depicted in Figure \ref{fig:Grafting incompatibility at a puncture} can happen, which is reminiscent of the grafting of two arcs which was covered in earlier chapters. This is why we will call such an incompatibility a \textit{grafting incompatibility}.\\
The structure of this section  will be similar to the previous one. We will first give a proper definition of what a grafting incompatibility is and how this can be resolved. Later we will construct and prove the desired bijection which in turn will help us prove the announced theorem of this chapter.\\

Suppose that $(S, M, P)$ is a punctured surface and $T$ is a triangulation of the surface. Let $\gamma_1$ and $\gamma_2$ be two arcs which have exactly one incompatible tagging at a puncture $p$, and $w_1, w_2$ the associated loopstring to each arc respectively. We will say that $\gamma_1$ and $\gamma_2$ have a \textit{grafting incompatibility} at the puncture $p$ when the arcs cross the same arc $t$ of the triangulation $T$ before meeting at the puncture $p$ (Figure \ref{fig:Grafting incompatibility at a puncture}). \\
One big difference of this incompatibility to the regular incompatibility that was covered earlier, is the fact that one does not need to consider the extended loopstring in order to notice that such a problem occurs on the surface, and the definition is more reminiscent of the classical resolution of a crossing of two arcs with common overlap. However, the idea of the resolution still remains the same. One must follow on arc until the puncture and then create two new arcs by turning left and right respectively around the puncture before continuing following the path of the other arc.

\begin{figure}
    \centering
    \includegraphics[scale=0.7]{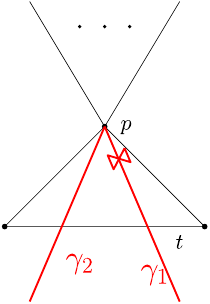}
    \caption{Grafting incompatibility at a puncture.}
    \label{fig:Grafting incompatibility at a puncture}
\end{figure}

\begin{definition}\label{resolution of grafted tagging at one puncture}
 Suppose that $w_1$ and $w_2$ are two abstract loopstrings such that $w_1$ contains a hook at its end, while $w_2$ does not contain a hook at its end. Assume that $w_1= h_1u_1amh_2$ and $w_2= h_3u_2m$ where $u_i$ and $m$ are substrings, $h_i$ are hooks, with $h_2\neq \emptyset$ and $a$ is an arrow. We then say that $w_1$ and $w_2$ have a \textit{grafting incompatibility} at their endpoints.\\

The resolution of the grafting incompatibility of $w_1$ and $w_2$ are the loopstrings $w_3$ and $w_4$ which are defined as follows:
    \begin{enumerate}[(i)]
        \item  $w_3 = h_1u_1mh_2m^{-1}u_2^{-1}$,
        \item
         \begin{equation*}
             w_4 = \begin{cases}
                 h_1u_1eu_2^{-1}h_3 & \text{if $u_1, u_2 \neq \emptyset$ and where $e$ is an opposite arrow to $a$}.\\
                 h_1u_1' & \text{if $u_1\neq \emptyset$ and $u_2= \emptyset$}.\\
                 u_2'h_3 & \text{if $u_2\neq \emptyset$ and $u_1= \emptyset$}.
             \end{cases}
         \end{equation*}
         where $u_1'$ is such that $u_1= u_1'r'$ where $r''$ is a maximal sequence of inverse arrows and $u_2'$ is such that $u_2= u_2'r''$ where $r'$ is a maximal sequence of direct arrows.
    \end{enumerate} 
    
\end{definition}

The following example is based on Figure \ref{fig:Surface grafting incompatibility}, in which we can see how visually the resolution may look the same, but looking at the intersection of the arcs $\gamma_3$ and $\gamma_4$ one can notice that there is a fan (in this example it consists of only the arc 4), which fan the arc $\gamma_4$ is not actually intersecting, explaining the existence of the cases for $w_4$ in the above definition.

\begin{figure}
    \centering
    \includegraphics[scale=0.7]{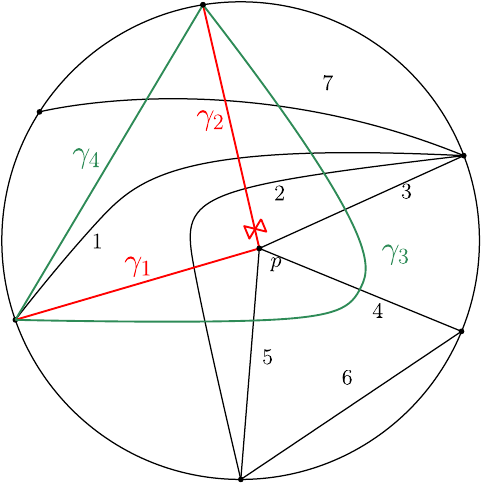}
    \caption{Resolution of grafting incompatibility at a puncture p.}
    \label{fig:Surface grafting incompatibility}
\end{figure}

\begin{example}
Suppose that $(S, M, P, T)$ is the triangulated surface that appears on Figure \ref{fig:Surface grafting incompatibility}. \\
Then we have that $w_1= 2\hookrightarrow 5 \leftarrow 4 \leftarrow 3$ and $w_2= 7\leftarrow 1\leftarrow 2$. Using Definition \ref{resolution of grafted tagging at one puncture} and since we can write $u_1= 7\leftarrow 1$ as $u_1= u_1'r'$ where $u_1'= 7$ and $r'= 1$, we have:
\[
w_3= 7\leftarrow 1\leftarrow 2 \leftarrow 3 \rightarrow 4 \rightarrow 5\leftarrow 2,
\]
\[
w_4= 7.
\]
\end{example}

As we have already done in each chapter we will first present an important lemma and a proposition which will be used for the definition of the desired bijection. In the continuation of this section we will focus only on one case, regarding how $w_4$ is defined in the resolution, since similar arguments and constructions can be given for the rest of the cases.

\begin{lemma}\label{removing tops for one puncture grafting incompatibility}
 Let $w_1, w_2, w_3$ and $w_4$ as they were defined in Definition~\ref{resolution of grafted tagging at one puncture} with $u_1=\emptyset$. Suppose that $m_1$ (resp. $m_2$) is a proper submodule of  $w_3$ (resp $w_4$). If $m_1$  (resp. $m_2$) is reached from $w_3$ (resp. $w_4$) by the sequence $(a_i), 1\leq i\leq n$, then we can apply the same sequence $(a_i)$ to the module $w_1\oplus w_2$ (resp. $w_1\oplus u_2'$) to obtain a submodule $m_1'$ (resp. $m_2'$) of it.
    \begin{proof}
    Suppose that $m_2$ is a submodule of $w_4$ and it is reached by the sequence $(a_i), 1\leq i\leq n$. One can easily see we can just take out the same sequence from the module $w_1\oplus u_2'$, by applying the sequence only to $u_2'\subset w_1\oplus u_2$. \\

    Suppose now that $m_1$ is a submodule of $w_3$ and it is reached by the sequence $(a_i), 1\leq i\leq n$. The only tricky part of this case is the duplication of the string $m$ in $w_3$, since removing elements from $m$ when viewed as a substring of $w_3$, does not give us a natural choice of removing this element from either $w_1$ or $w_2$.\\
    W.l.o.g. let $h_2= x_1,\dots,x_n$ and $w_1=h_1u_1m\hookrightarrow x_k\leftarrow\dots\leftarrow x_1$. We then have that $w_3= h_1u_1m\rightarrow x_k\leftarrow\dots\leftarrow x_1\rightarrow m^{-1}u_2^{-1}$. It is now easy to see that if the sequence $(a_i)$ does not contain any of the elements in the set $\{x_1,\dots,x_k\}$ then if $a_j$ is in $h_1,u_1,h_3,u_2$ or $m$ for some $1\leq j\leq n$, then there is a unique and obvious choice for the same element to be removed from $w_1\oplus w_2$. Assume now that there is at least one $a_j\in h_2$ for some $1\leq j\leq n$. Then this implies also that $x_k\in (a_i)$ since it is locally a top. Therefore we must remove also $x_k$ from $h_w$ viewed as a hook of $w_1$, which makes $m$ locally a top in $w_1$ when viewed next to $h_2$. Therefore when we are removing elements twice from both the substrings $m$ of $w_3$ we can also remove the same elements from $w_1$ or $w_2$ which completes the proof. 
    \end{proof}
\end{lemma}

The next proposition, provides the uniqueness of the construction of the associated submodules as it is needed later for the construction of a suitable mapping.   

\begin{proposition}\label{unique remove tops grafting incompatibility one puncture}
Let $w_1, w_2, w_3$ and $w4$ as they were defined in \ref{resolution of grafted tagging at one puncture} with $u_1=\emptyset$. Suppose that $m_1$ (resp. $m_2$) is a proper submodule of  $w_3$ (resp $w_4$). If $m_1$  (resp. $m_2$) is reached from $w_3$ (resp. $w_4$) by the sequence $(a_i), 1\leq i\leq n$, then we can apply in a unique the same sequence $(a_i)$ to the module $w_1\oplus w_2$ (resp. $w_1\oplus u_2'$) to obtain a submodule $m_1'$ (resp. $m_2'$) of it.
    \begin{proof}
        The uniqueness of the module $m_2'$ follows straightforwardly by the construction of the module as it was done in lemma \ref{removing tops for one puncture grafting incompatibility}, by noticing that we can just``forget" the module $w_1$, and remove the same tops from $u_2'$ when viewed as a submodule of $w_1\oplus u_2'$.
    \end{proof}
\end{proposition}

Having established that there is a unique way of passing from submodules of $w_3$ and $w_4$ to submodules of $w_1\oplus w_2$ we are now ready to construct the following bijection. 

\begin{theorem}\label{bijection of modules grafting incompatibility one puncture}
    Let $w_1, w_2, w_3$ and $w_4$ as they were defined in \ref{resolution of grafted tagging at one puncture}. Then there exists a bijection:
    \[
    \Psi \colon SM(w_3) \bigcup SM(w_4) \to SM(w_1\oplus w_2),
    \]
    which is defined as follows:
    \begin{gather*}
    \Psi(w_3) = w_1\oplus w_2, \\
    \Psi(w_4) = w_1\oplus u_2',     
    \end{gather*}
    If $w$ is a submodule of $w_3$  which is reached by the sequence $\{a_i\} 1\leq i\leq n$, and $w'$ the unique submodule of $(w_1\setminus s_2)\oplus w_2$ which is reached by the same sequence, then we define:
    \[
    \Psi(w) = w'.
    \]
    If $w$ is a submodule of $w_4$, then $\Psi(w)$ is defined in a similar way.
    
\end{theorem}

As per usual we only need to prove that the map defined on Theorem \ref{bijection of modules grafting incompatibility one puncture} is an injection, which is done in the following proposition.\\

\begin{proposition}\label{injectivity of Psi grafted incompatibility one puncture}
    The map $\Psi$ as it was defined in Theorem \ref{bijection of modules grafting incompatibility one puncture} is an injection.
    \begin{proof}

        Let $m_1, m_2\in SM(w_3) \bigcup SM(w_4)$ such that $\Psi(m_1)=\Psi(m_2)$.\\
        By construction of the map $\Psi$, if both $m_1, m_2\in SM(w_3)$ or $m_1, m_2\in SM(w_4)$, then $m_1= m_2$.\\
        We will now show that if $m_1\in SM(w_3)$ and $m_2\in SM(w_4)$ then $\Psi(m_1)\neq \Psi(m_2)$. Assume that $\Psi(m_1)=\Psi(m_2)$ and let $(a_i), 1\leq i\leq n$ be a sequence of the simples removed from $m_1$ to reach $\Psi(m_1)$.\\
        Since $w_4= u_2'$ we take that all the elements of $r'$ and $m$ must belong to the sequence $(a_i)$. However for this to happen, the sequence $(a_i)$ should contain at least one element from the hook, since $r'$ and $m$ are not a top of the module $w_3$. However in turn this would give us that there is a submodule of $w_4= u_2'$ which is not mapped to $w_1\oplus m''$ for $m''$ a submodule of $w_4$, which gives us a contradiction by the construction of the map $\Psi$.\\

    \end{proof}
\end{proposition}

\begin{proof}[Proof of Theorem \ref{bijection of modules grafting incompatibility one puncture}]
The fact that $\Psi$ is a bijection follows immediately from the fact that $\Psi$ is an injective map (Proposition \ref{injectivity of Psi grafted incompatibility one puncture}) and Remark \ref{termsincoeffreecase}.
\end{proof}

\subsection{ Skein relation for arcs with a single incompatibility at a puncture}

Having already established the two bijection \ref{bijection of modules incompatibility one puncture} and \ref{bijection of modules grafting incompatibility one puncture} we would also like to prove that there is no "lost" information from passing to a loopstring from a loop graph. For this, we will use the monomials associated to  loopstrings as they were defined in Chapter 4.

We will not give full proofs for each case, since the idea is very similar and it is basically depending on a series of computational tricks in order to rewrite some monomials.

\begin{lemma}\label{xvariablesagree regular incompatibility at a puncture}
     Let $w_1, w_2, w_3$ and $w_4$ as they were defined in \ref{resolution of incompatible tagging  one puncture}. Then:
     \[
     x(w_3) = x((w_1\setminus s_2)\oplus w_2),
     \]
     and 
     \[
     x(w_4) = x(w_1\oplus w_2)
     \]
     \begin{proof}
         We will show only the first equality, since the second one is also just a direct computation.\\

         By definition \ref{monomofstring} we have that:
         \[
         x(w_3) = x(h_1u_1\rightarrow s_1\rightarrow k\leftarrow_2^{-1}h_3) = \frac{x(h_1u_1\rightarrow s_1)x(u_2^{-1}h_3)}{x_k}
         \]
         However, since $k$ by definition belongs to the extended $w_2$ we have that:
         \[
         x(w_2) = \frac{x(h_3u_2)}{x_k},
         \]
         and therefore we take that
     \[
     x(w_3) = x((w_1\setminus s_2)\oplus w_2),
     \]
         as we wanted to show.

     \end{proof}
\end{lemma}

\begin{theorem}\label{skein relation for single incompatibility at a puncture}
    Let $(S,M,P,T)$ be a triangulated punctured surface, $\gamma_1$ a notched arc at a puncture $p$ and $\gamma_2$ an arc which has an plain endpoint at the puncture $p$. Let $\gamma_3$ and $\gamma_4$ be the two arcs obtained by smoothing the crossing.
    \begin{itemize}
        \item[(i)] 
        If  $\gamma_1$ and $\gamma_2$ have a regular incompatibility at the puncture $p$, then:
        \[
        x_{\gamma_1} x_{\gamma_2} = x_{\gamma_3}  + Y^{+} x_{\gamma_4},
        \]
        where $Y^{+} =  Y(s_1)$.
        \item[(ii)]
        If $\gamma_1$ and $\gamma_2$ have a grafting incompatibility at the puncture $p$, then:
        \[
        x_{\gamma_1} x_{\gamma_2} = x_{\gamma_3}  + Y^{+} x_{\gamma_4},
        \]
        where $Y^{+} =  Y(w_1)$.
    \end{itemize}

    \begin{proof}
        The proof of this theorem is a direct consequence of Theorems \ref{bijection of modules incompatibility one puncture}, \ref{bijection of modules grafting incompatibility one puncture}, Lemma \ref{xvariablesagree regular incompatibility at a puncture} and Remark \ref{removetoppreservesx}.
    \end{proof}
\end{theorem}

\section{ Skein relations for double incompatibilities in two punctures}

In this section we will work on skein relations for arcs which have two incompatibilities in two punctures. This can happen in two different occasions. We can either have a double notched arc with its endpoints in two punctures $p$ and $q$ and a plain arc with its endpoints in the same punctures $p$ and $q$. Additionally we could have two singly notched arcs, which have both of their endpoints in the same punctures $p$ and $q$, but they are tagged in different punctures. \\
We will see that both of these cases give rise to the same skein relation and therefore instead of proving skein relation for these two cases separately we will just deal with one of these cases and then prove that indeed these cases are equivalent. \\
These two cases behave a little bit differently than what we have seen so far in the sense that they can give rise to four different arcs of which only one of these arcs does not include ant additional $y$-terms when resolving the incompatibilities. The deeper reason on why such a thing occurs can be understood through looking at the two incompatibilities that occur separately, instead of both at the same time.\\
To understand the above argument let us consider one doubly notched arc $\gamma_1$, which has a tagging at its endpoints $p$ and $q$, and a plain arc $\gamma_2$ which has its endpoints in $p$ and $q$. One could try and resolve first the tagging at the puncture $q$. This procedure would look like a natural extension of the previous section, since then, we could think of dealing with a single incompatibility at the puncture $q$ and therefore apply the results of the previous section to resolve this incompatibility, by creating two new arcs $\gamma'_1$ and $\gamma'_2$, as they can be seen in Figure~\ref{fig:Step by step resolution of double incomp}, which would have a self incompatibility at the puncture $p$. If we could assume that the results of the previous section can be generalized for self incompatible arcs, we could then deal with these two self incompatibilities separately, giving rise to four new arcs in total. Assuming that the previous results still hold, only one of the newly formed arcs would not have any additional $y$-terms in front of it in the complete skein relations. \\
To sum up and under the fore mentioned assumptions the expected skein relations for such arcs would be the following:
\[
x_{\gamma_1} x_{\gamma_2} = Y_1 x_{\gamma'_1} + Y_2 x_{\gamma'_2} = Y_1 (Y_1^+ \delta_1 + Y_1^- \delta_2) +  Y_2 (Y_2^+ \delta_3 + Y_2^- \delta_4).
\]
Notice also that the initial arbitrary selection of dealing with one specific incompatibility was not restricting. Even if we started by resolving the other incompatibility we would end up to a similar relation, since one can notice that by rearranging the terms we could take the following equality:
\[
x_{\gamma_1} x_{\gamma_2} = Y_1 x_{\gamma'_1} + Y_2 x_{\gamma'_2} = Y_1 (Y_1^+ \delta_1 + Y_1^- \delta_2) +  Y_2 (Y_2^+ \delta_3 + Y_2^- \delta_4).
\]

\begin{figure}
    \centering
    \includegraphics[scale=0.9]{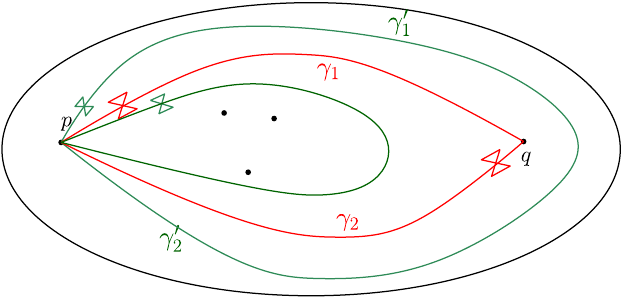}
    \caption{Resolution of single incompatibility and self incompatible arcs.}
    \label{fig:Step by step resolution of double incomp}
\end{figure}

However, the above arguments cannot be applied straightforwardly since we do not deal with arcs that have self incompatibilities. In fact, there was an open question from Musiker, Schiffler  and Williams, regarding such arcs and if those arcs are elements of the Cluster Algebra. We conjecture that such arcs are indeed part of the cluster algebra when we are dealing with bordered surfaces. The idea, is that such arcs can be decomposed into bands, and then in the presence of border components, one can prove that each band is part of the cluster algebra. We will not get into more details now, since this will be further explored in the last chapter of this thesis. \\
To get back on the topic of this chapter, we will prove that indeed the skein relations when resolving double incompatibilities at a puncture between two arcs have exactly the form that we described earlier. Nonetheless, our approach in the proof of this statement will be different. We will prove every statement by resolving both incompatibilities at the same time, avoiding in this way to deal with self incompatible arcs at a puncture. \\
We should also mention an important detail before going to the main part of this section. Not every skein relation gives rise to four bands. It can happen that the resolution gives rise to one and only band arc. This however can occur in some very specific cases, as the one appearing in figure~\ref{fig:Resolution of double incomp resulting to a single band}. Such an observation can be very crucial in proving that some band arcs, even in surfaces without boundary components, are part of the cluster algebra. Notably, problems occur in surfaces with genus bigger that one and when dealing with bands that ``pass around" the handle of the surface. \\

\begin{figure}
    \centering
    \includegraphics[scale=0.9]{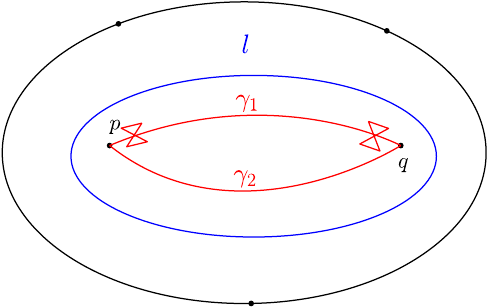}
    \caption{Resolution of double incompatibility, which gives rise to a singe band arc $l$.}
    \label{fig:Resolution of double incomp resulting to a single band}
\end{figure}

To sum up in this chapter we deal with the following cases: 
\begin{itemize}
    \item two singly notched arcs which have two incompatible taggings at two punctures,
    \item a doubly notched arc and a plan arc which have two incompatible taggings at two punctures.
\end{itemize}
The following theorem which addresses these cases and will be proven at the end of this chapter.

\begin{theorem}
        Let $(S,M,P,T)$ be a triangulated punctured surface, $\gamma_1$ a double notched arc at two puncture $p$ and $q$, and $\gamma_2$ an arc which has an plain endpoint at the puncture $p$. Let $\gamma_3$, $\gamma_4$, $\gamma_5$ and $\gamma_5$ be the four arcs obtained by smoothing the incompatibilities at the two punctures.
        Then we have that:
        \[
        x_{\gamma_1} x_{\gamma_2} = Y_1Y_1^- x_{\gamma_3}  + Y_1Y_1^{+} x_{\gamma_4} + Y_2Y_2^- x_{\gamma_5}  + Y_2Y_2^{+} x_{\gamma_6},
        \]
        where $Y_1, Y_2, Y_1^+, Y_1^-, Y_2^+$ and $Y_2^-$ are monomials on $y$-variables and in each pair $(Y_1, Y_2)$, $(Y_1^+, Y_1^-)$ and $(Y_2^+, Y_2^-)$ at least one term of each pair is equal to 1.
\end{theorem}

\subsection{ Resolution of double incompatibility of two arcs at a puncture}

As in all the previous cases discussed in this thesis, the resolution of the incompatibility at the punctures  depends on the local configuration of the given triangulation of a surface. In the case of a double incompatibility, since there are two punctures, the so called grafting incompatibility that was discussed in the previous section can occur in one, both or none of the punctures. Instead of splitting the cases of a regular incompatibility and a grafting incompatibility, we will work on both of them at the same time by assuming that each one of them occurs in exactly one of the punctures. The rest of the cases which consist of all the possible configuration that we can have locally are simple generalisations of the case that we will focus, so we will not discuss them any further.

Suppose that $(S, M, P)$ is a punctured surface and $T$ is a triangulation of the surface. Let $\gamma_1$ and $\gamma_2$ be two arcs which have two incompatible taggings at the punctures $p$ and $q$, and $w_1, w_2$ are the associated loopstring of each arc respectively. We will say that $\gamma_1$ and $\gamma_2$ have a \textit{regular incompatibility} at the puncture $p$ when the arcs cross the same arc $t_3$ of the triangulation $T$ before meeting at the puncture $p$.
We will also say that the arcs have a grafting incompatibility at the puncture $q$ when the arcs do not cross the same arcs $t_1$ and $t_2$ of the triangulation $T$ before meeting at the puncture $q$ (Figure \ref{fig:Two regular incompatibilities}).

\begin{figure}
    \centering
    \includegraphics[scale=0.7]{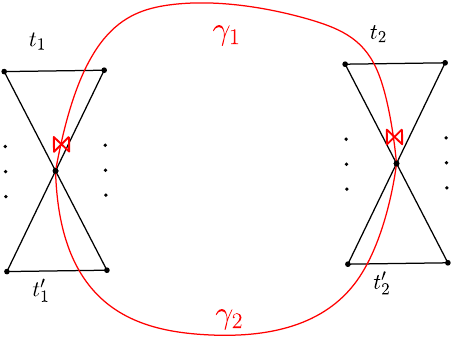}
    \caption{Regular incompatibilities at two punctures.}
    \label{fig:Two regular incompatibilities}
\end{figure}

\begin{definition}\label{resolution of double incompatible taggings two punctures}
 Suppose that $w_1$ and $w_2$ are two abstract loopstrings such that $w_1$ contains two hooks at its endpoints, while $w_2$ does not contain any hook at its endpoints. Assume that $w_1= h_1u_1mh_2$ and $w_2= u_2m$ where $u_i, m$ are substrings and $h_i$ are hooks, with $h_1, h_2\neq \emptyset$. Let also $w_2'= a_1u_2m$ be the extended loopstring of $w_2$ with $a_1\in h_2$ and $w_2''= a_2u_2m$ be the extended loopstring of the dual of $w_2$ with $a_1\in h_2$. We then say that $w_1$ and $w_2$ have a \textit{regular incompatibility} at $p$ and a \textit{grafting incompatibility} at $q$.\\

Since $a_1, a_2\in h_2$ we have that $h_1= h_1'a_1a_2h_1''$ where $h_1'$ and $h_1''$ are not necessarily non empty strings. 
The resolution of the incompatibilities of $w_1$ and $w_2$ are the bandstrings $w_3, w_4, w_5$ and $w_6$ which are defined as follows:
    \begin{enumerate}[(i)]
        \item  $w_3^{o} = a_1h_1'u_1mh_2mu_2$,
        \item  $w_4^{o} = a_1h_1'u_1$, 
        \item  $w_5^{o} = a_2h_1''u_1mh_2mu_2$,
        \item  $w_6^{o} = a_2h_1''u_1$.
    \end{enumerate}
\end{definition}

We will now proceed on constructing two examples that will help us understand how the resolution mentioned in the above definition is applied. In the first example, there will be a regular resolution where all the bandstrings created after the resolution are non empty, while the second one will deal with one of the extreme cases in which there is only one regular bandstring created.

\begin{example}
    Suppose that $(S,M,P,T)$ is the triangulated surface depicted on Figure \ref{fig:Double incompatibility at two punctures}.
    We then have that:
    \[
    w_1= 6 \rightarrow 5 \rightarrow 4 \rightarrow 3 \rightarrow 2 \rightarrow 1 \hookleftarrow 10 \rightarrow 11 \leftarrow 9 \leftarrow 12 \hookleftarrow 14 \rightarrow 13,
    \]
    \[
    w_2= 7\rightarrow 5 \rightarrow 8 \rightarrow 9 \leftarrow 12.
    \]
    Looking at the loopstrings we can notice that there is a regular incompatibility at the end of the end since and we have furthermore the following:
    \[
    h_1=  6 \rightarrow 5 \rightarrow 4 \rightarrow 3 \rightarrow 2 \rightarrow 1,
    \]
    \[
    h_2= 14 \rightarrow 13,
    \]
    \[
    u_1= 10 \rightarrow 11,
    \]
    \[
    u_2= 7\rightarrow 5 \rightarrow 8,
    \]
    \[
    m= 9 \leftarrow 12.
    \]
    In order to see if there is an incompatibility at the start of the loopstrings we need to construct the extended and reverse extended strings $w_2'$ and $w_2''$ respectively, which are as follows:
    \[
    w_2'= 4\leftarrow 7\rightarrow 5 \rightarrow 8 \rightarrow 9 \leftarrow 12,
    \]
    \[
    w_2''= 3\rightarrow 7\rightarrow 5 \rightarrow 8 \rightarrow 9 \leftarrow 12.
    \]
    Looking at the above extended strings we see that $a_1= 4$ and $a_2= 3$ which are indeed part of $h_1$. Therefore, we can rewrite $h_1= h_1'a_1a_2h_2''$ where:
    \[
    h_1'= 6 \rightarrow 5,
    \]
    \[
    h_1''= 2 \rightarrow 1.
    \]
    We now have all the necessary tools to resolve the incompatibilities. Using Definition \ref{resolution of double incompatible taggings two punctures} we take the following bandstrings:
    \[
    w_3^{o}= \leftarrow 3 \rightarrow 2 \rightarrow 1 \leftarrow 10 \rightarrow 11 \leftarrow 9 \leftarrow 12 \leftarrow 14 \rightarrow 13 \leftarrow 12 \rightarrow 9 \leftarrow 8\leftarrow 5\leftarrow 7 \leftarrow,
    \]
    \[
    w_4^{o}= \leftarrow 3 \rightarrow 2 \rightarrow 1 \leftarrow 10 \rightarrow 11 \rightarrow 8\leftarrow 5\leftarrow 7 \leftarrow,
    \]
    \[
    w_5^{o}= \rightarrow 4 \leftarrow 5 \leftarrow 6 \rightarrow 10 \rightarrow 11 \leftarrow 9 \leftarrow 12 \leftarrow 14 \rightarrow 13 \leftarrow 12 \rightarrow 9 \leftarrow 8\leftarrow 5\leftarrow 7 \rightarrow,
    \]
    \[
    w_6^{o}= \rightarrow 4 \leftarrow 5 \leftarrow 6 \rightarrow 10 \rightarrow 11 \rightarrow 8\leftarrow 5\leftarrow 7 \rightarrow,
    \]
    which correspond to the bands appearing in Figure\ref{fig:Resolution of double incompatibility}.
\end{example}

\begin{figure}
    \centering
    \includegraphics[scale=0.35]{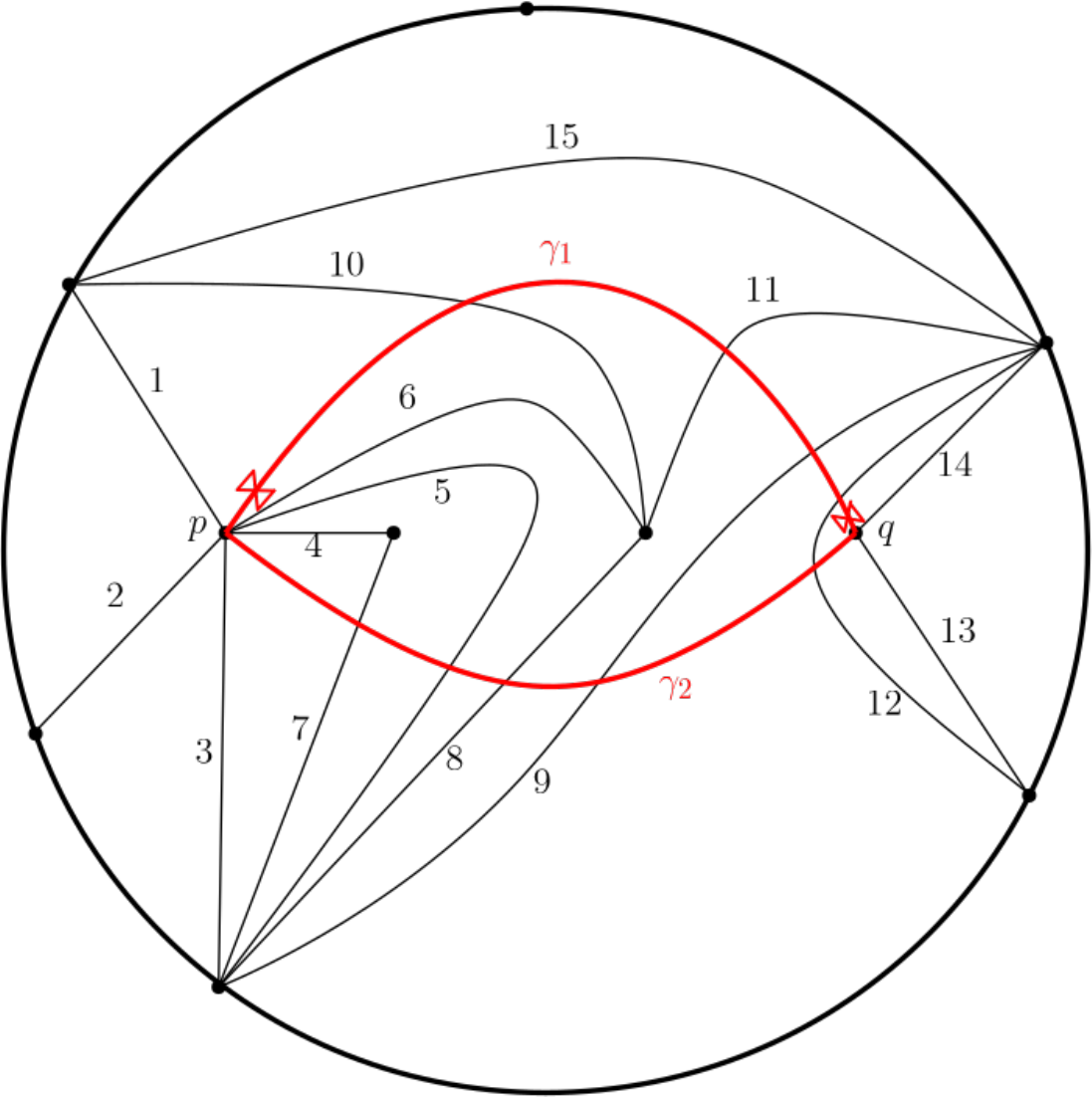}
    \caption{Double incompatibility at two punctures.}
    \label{fig:Double incompatibility at two punctures}
\end{figure}

\begin{figure}
    \centering
    \includegraphics[scale=0.35]{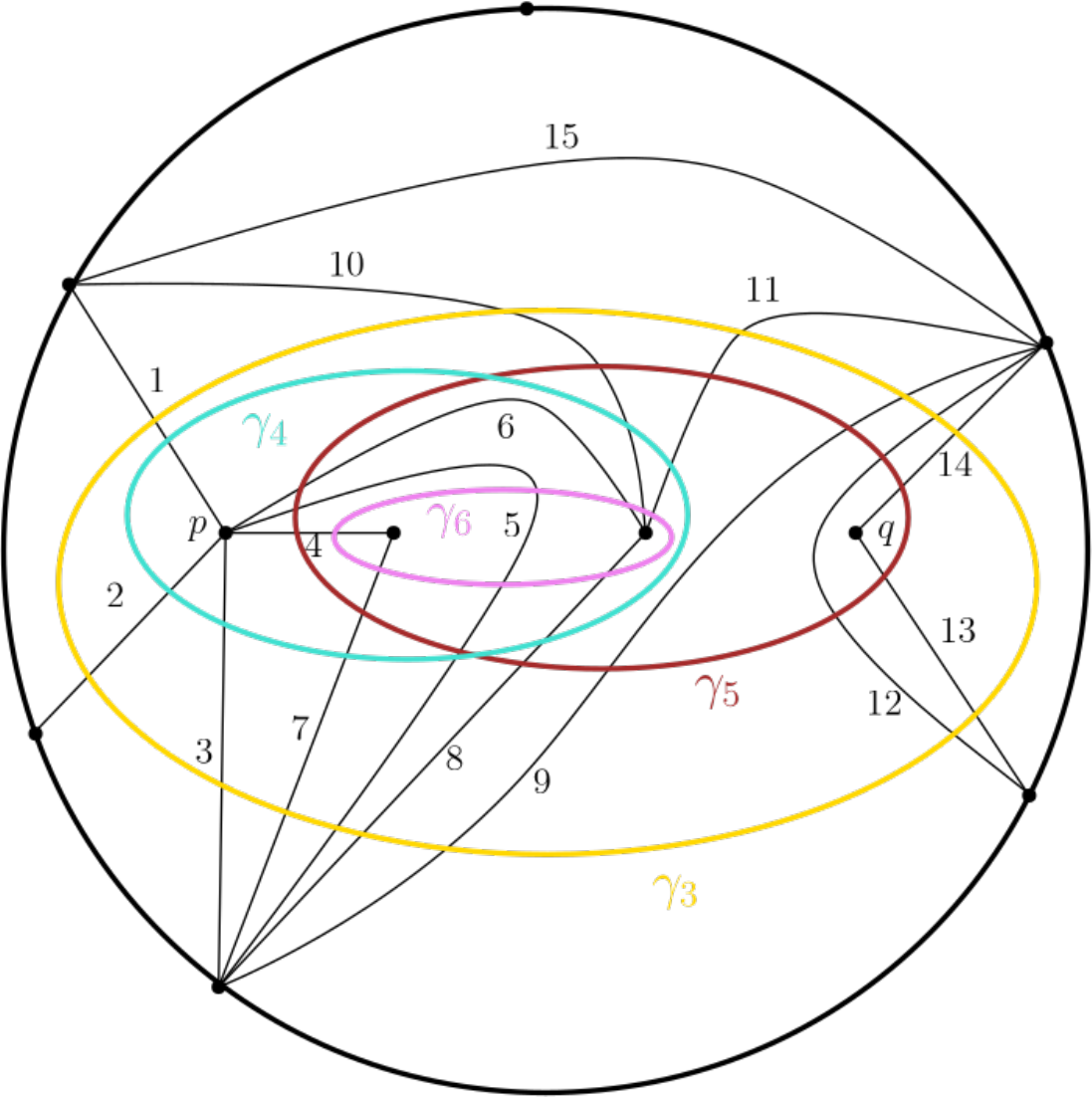}
    \caption{Resolution of the double incompatibility in Figure\ref{fig:Double incompatibility at two punctures}.}
    \label{fig:Resolution of double incompatibility}
\end{figure}

\begin{remark}
    Notice that the not all the bandstrings defined in the previous definition \ref{resolution of double incompatible taggings two punctures} are always non empty. In particular it may happen that $u_1$ and $u_2$ are empty and $a_1, a_2$ are located at the start or the end of the hook $h_2$. In such a situation at least a bandstring would be empty. To make things more clear, the following example is a case where such a resolution occurs.
\end{remark}

\begin{example}
    Suppose that $(S,M,P,T)$ is the triangulated surface depicted on Figure \ref{fig:Extreme case for double incompatibility}. We then have that:
    \[
    w_1= 3\rightarrow 2\rightarrow 1 \hookleftarrow 4 \rightarrow 5 \hookrightarrow 6 \leftarrow 7 \leftarrow 8,
    \]
    \[
    w_2= 4\rightarrow 5.
    \]
    It is easy to see that in this case we have that:
    \[
    h_1=  3\rightarrow 2\rightarrow 1,
    \]
    \[
    h_2= 6 \leftarrow 7 \leftarrow 8,
    \]
    \[
    m= 4\rightarrow 5.
    \]
    while $u_1$ and $u_2$ are empty! It is also easy to see that since both of these substrings are empty we have two regular incompatibilities, one at each of the two punctures. Following the resolution as it was defined earlier we take:
    \[
    w_3^{o}= \leftarrow3\rightarrow 2\rightarrow 1 \leftarrow 4 \rightarrow 5 \leftarrow 8 \rightarrow 7 \rightarrow 6\leftarrow 5\leftarrow 4\leftarrow,
    \]
    and $w_4= h_1$, $w_5= h_2$ and $w_6= \emptyset$.
\end{example}

\begin{figure}
    \centering
    \includegraphics[scale=0.35]{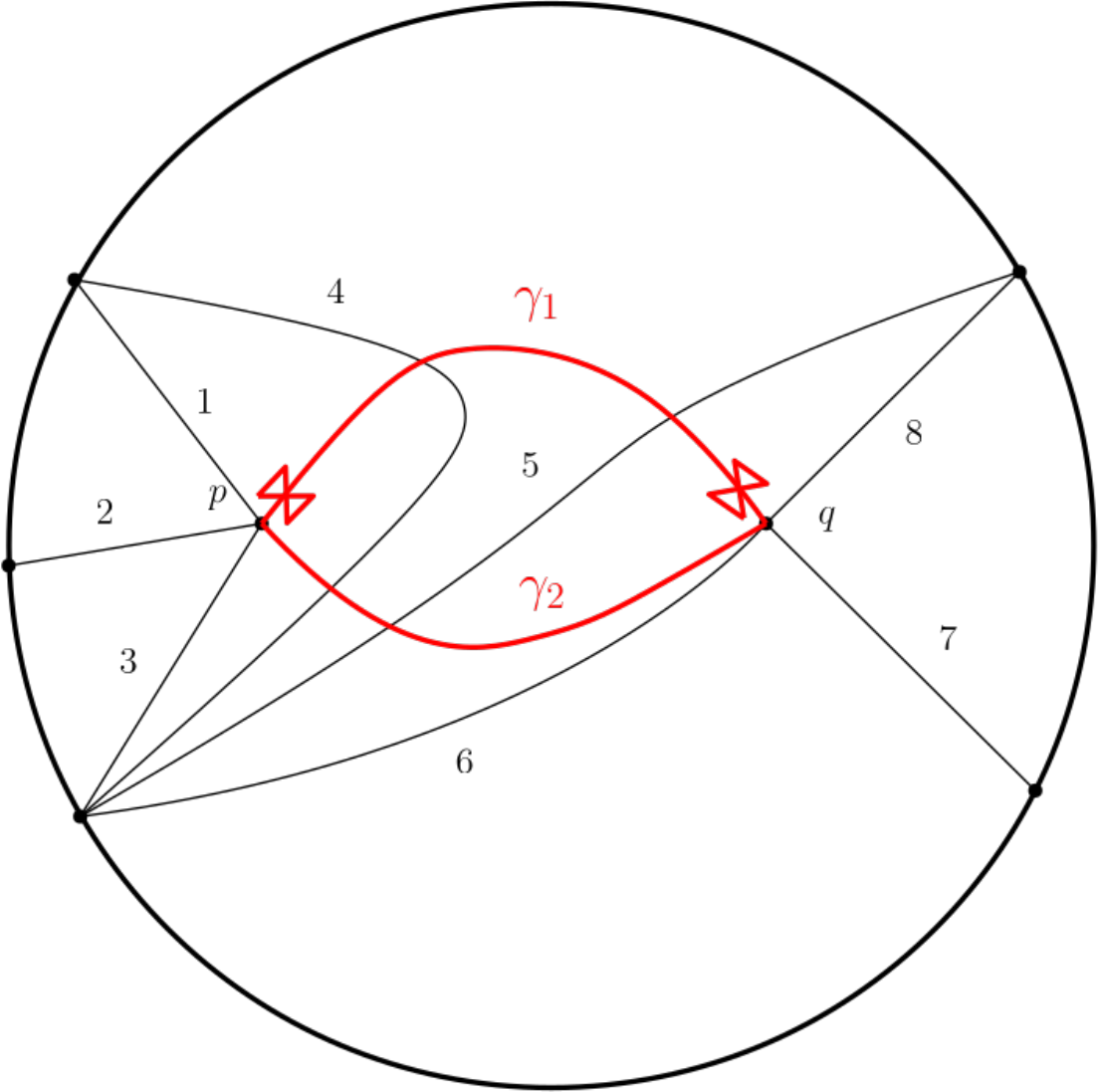}
    \caption{Extreme case for double incompatibility.}
    \label{fig:Extreme case for double incompatibility}
\end{figure}

\begin{figure}
    \centering
    \includegraphics[scale=0.35]{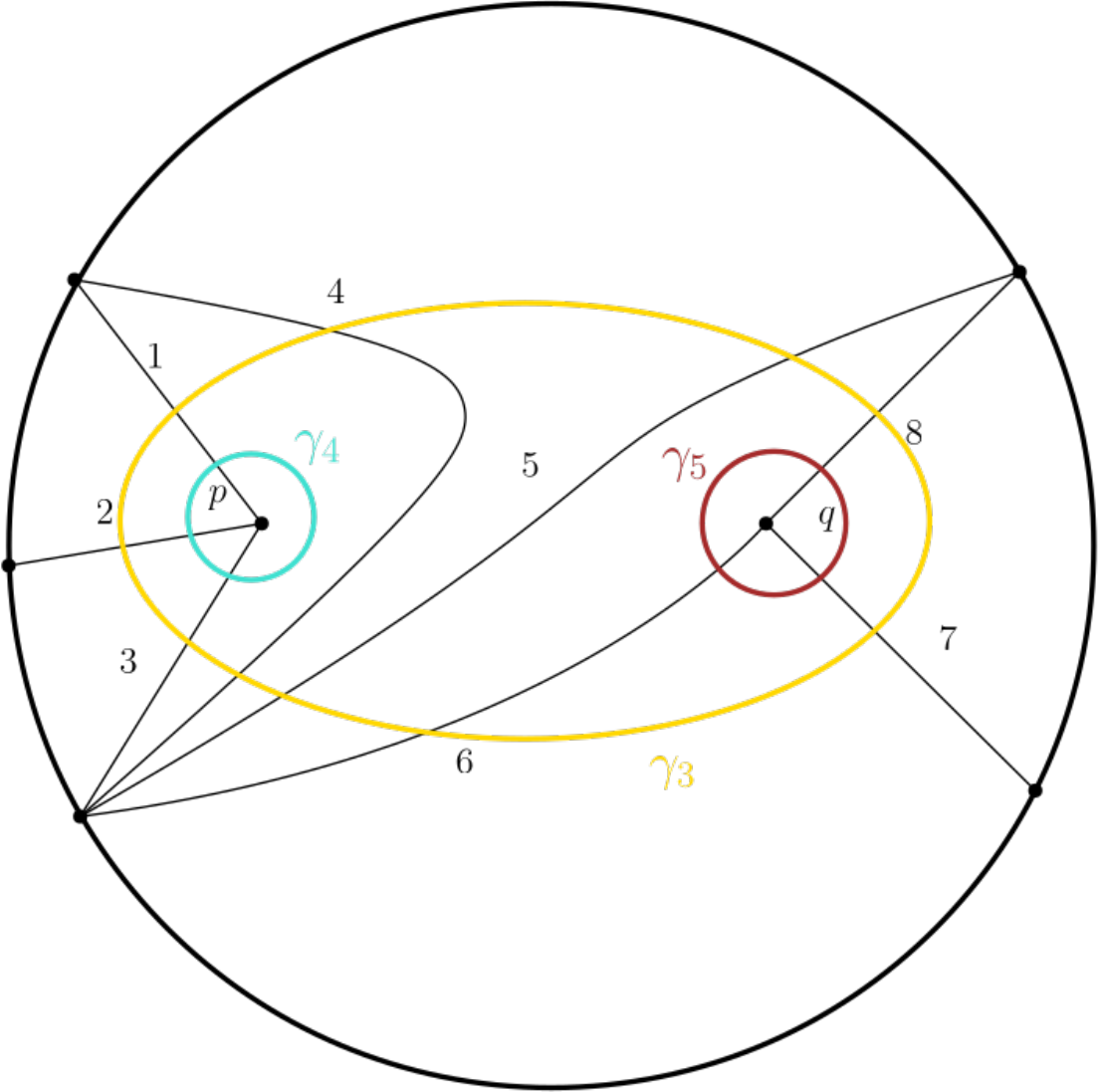}
    \caption{Resolution of extreme case for double incompatibility.}
    \label{fig:Extreme case for double incompatibility resolution}
\end{figure}

Having showcased how the construction of the four bandstrings works, we should now try to show that indeed under this definition there is a nice bijection between the submodules associated to the initial modules and the modules associated to the newly formed ones. \\

\begin{remark}
    Looking at how we obtained $a_1$ and $a_2$ in definition \ref{resolution of double incompatible taggings two punctures}, one can notice that a priori there is no way of telling how the local configuration in the hook $h_1$ looks, meaning we cannot know if there is an arrow from $a_1$ to $a_2$ or the other way around. Looking at how we have constructed the bijections in the previous chapters one can see that by taking cases we were trying to differentiate what happens when a simple is locally a socle or a top of the module. A similar obstruction appears here two. However instead of taking both of the aforementioned cases we will only focus on one of them. The reason for such an assumption is the fact that since there are two pairs of bandstrings in the resolution, the situation is very symmetrical and one could prove what happens in the other case, just by relabeling the two pairs. \\
    Starting from the next proposition which basically indicates how the bijection will be constructed, we will use the assumption that locally there is an arrow from $a_1$ to $a_2$ in the hook $h_1$.
\end{remark}

\begin{proposition}\label{unique remove tops incompatibility two punctures}
    Let $w_1, w_2, w_3, w_4, w_5$ and $w_6$ as they were defined in \ref{resolution of double incompatible taggings two punctures}. Suppose that $m_i$ is a proper submodule of  $w_i$ for $1\leq i\leq 4$, and without loss of generality assume that the local configuration of the hook $h_1$ is $a_1\leftarrow a_2$.
    If $m_3$ is reached from $w_3$  by the sequence $(b_i), 1\leq i\leq n$, then we can apply the same sequence $(a_i)$ to the module $w_1\oplus w_2$ to obtain a unique submodule $m_1'$ of it.\\
    The same statement is true by adjusting the following:\\
    The module $m_4$ is reached by applying the sequence to the module $(w_1\setminus mh_2)\oplus w_2$.\\
    The module $m_5$ is reached by applying the sequence to the module $(w_1\setminus a_1h_1')\oplus w_2$.\\
    The module $m_6$ is reached by applying the sequence to the module $((w_1\setminus a_1h_1')\setminus mh_2)\oplus w_2$.
    \begin{proof}
        The uniqueness of the modules follows from the construction of the modules in 
        the next Lemma \ref{removing tops for double puncture incompatibilities}.
    \end{proof}
\end{proposition}

\begin{lemma}\label{removing tops for double puncture incompatibilities}
    Under the assumptions of Proposition \ref{unique remove tops incompatibility two punctures}, For each of the modules $m_i$ and sequences $(b_i), 1\leq i\leq n$, we can apply the same sequence in the associated submodule of $w_1\oplus w_2$ to obtain a new submodule $m_1'$
    \begin{proof}
        We will prove this statement indicatively only for one of the four cases. \\
        Let us assume that $m_1$ is a submodule of $w_3$ and is reached by the sequence $(b_i), 1\leq i\leq n$. We need to show that we can apply the same sequence to $w_1 \oplus w_2$ and obtain a submodule $m_1'$ of it.\\
        First of all, suppose that there are simples in $a_1h_1'$ that are part of the sequence. By construction since locally in $w_1$ we have $a_1\leftarrow a_2$, therefore the $h_1'$ contains the top of the hook. Therefore, by an iterative process of applying the sequence on $w_1\oplus w_2$ when we have to remove a simple from the hook, we can always do it. The above argument indicates that simples belonging to the hook cannot cause problems when we try to remove them as local tops from the module $w_1\oplus w_2$.\\
        The cases in which we have to remove a top from $u_1$ or $u_2$ are obviously straightforward too, since the local configuration in both modules $w_3$ and $w_1\oplus w_2$ is the same.\\
        The last case that we have to investigate is the case that we have to remove simples that belong to the common substring $m$ or in the hook $h_2$. However this part follows the same principles of the proofs of Propositions \ref{removertopwelldefined} and \ref{remove tops from tagged and loop} and so we omit the respective arguments. % {\color{blue} I can make the construction here and showcase it with an additional example.}
    \end{proof}
\end{lemma}

Let us now construct the bisection which will be used for the proof of the skein relation.

\begin{theorem}\label{bijection of modules incompatibilities two puncture}
    Let $w_1, w_2, w_3, w_4, w_5$ and $w_6$ as they were defined in \ref{resolution of double incompatible taggings two punctures}. Then there exists a bijection:
    \[
    \Psi \colon SM(w_3) \bigcup SM(w_4) \bigcup SM(w_5) \bigcup SM(w_6)\to SM(w_1\oplus w_2),
    \]
    which is defined as follows:
    \begin{gather*}
    \Psi(w_3) = w_1\oplus w_2, \\
    \Psi(w_4) = (w_1\setminus mh_2)\oplus w_2,  \\
    \Psi(w_5) = (w_1\setminus a_1h_1')\oplus w_2, \\
    \Psi(w_6) = ((w_1\setminus a_1h_1')\setminus mh_2)\oplus w_2.
    \end{gather*}
    If $m$ is a submodule of $w_3$  which is reached by the sequence $\{a_i\} 1\leq i\leq n$, and $m'$ the unique submodule of $(w_1\setminus s_2)\oplus w_2$ which is reached by the same sequence, then we define:
    \[
    \Psi(m) = m'.
    \]
    If $m$ is a submodule of $w_4, w_5$ or $w_6$, then $\Psi(m)$ is defined in a similar way.

    \begin{proof}
    The fact that $\Psi$ is a bijection follows from Proposition \ref{injectivity of Psi incompatibilities two punctures} and Remark \ref{termsincoeffreecase}.

\end{proof}
\end{theorem}

\begin{proposition}\label{injectivity of Psi incompatibilities two punctures}
    The map $\Psi$ as it was defined in Theorem \ref{bijection of modules incompatibilities two puncture} is an injection.
    \begin{proof}
        The proof of this proposition is omitted, since it follows the same logic as in previous chapters. \\
        The only difference here is that instead of having to compare submodules of two modules, we have to check all possible configurations between the four given modules $w_i$, $3\leq i\leq 4$. However this does not create any problem, since the gist of the argument given in previous proposition resolves around the fact that there is always a part of one module that is locally a top, while in the other module it is locally a socle. This principle is not violated in any of the possible cases since each one of the modules apart from $w_3$ takes out a different part of the hook $h_1$ or $mh_2$ or a combination of these two.\\
    \end{proof}
\end{proposition}

\subsection{ Skein relation for arcs with two incompatibilities at two puncture}

In this chapter we will prove skein relations when there are two arcs which have two different incompatibilities, one at each puncture. \\
The most interesting part, is the fact that out of the four terms appearing in the resolution, only one term does not contain any $y$-variables, which was partially explained at the start of this section earlier. \\
The bijection given in Theorem \ref{bijection of modules incompatibilities two puncture} as usual indicates to us, which monomials associated to each module on the left hand side of the bijection should be equal to the ones associated to modules on the right hand side of the bijection. We will indicatively prove some of these equalities first before going on the main statement of this subsection.

\begin{remark}
    Before going any further, we need to deal with an anomaly that may appear depending on the geometry of the surface. As it was showcased in the second example of the previous section the resolution of such incompatibilities may result in the loss of one of the four expected terms. However, what is equally important is the fact that sometimes an expected bandstring may end up being equal to a hook. A natural question would be, what does this even mean, since this cannot be a module. For this reason we will make a small abuse of notation and not indicate each time when this is not a module, but still refer to it like this. The propositions and the theorems in the previous sections still hold valid, under this abuse of notation by noticing one very important thing. \\
    Assuming that $w_4$ is equal to a hook , we will view this as a simple module. The idea behind it, is going back to the surface and the associated snake graphs, one should expect that the arc linked to $w_4$ is a band which just goes around a puncture. If one tries to construct the band graph of such an arc, it is easy to see that it would only have two perfect matchings. Since in our language perfect matchings are linked to submodules, it makes sense to view this $w_4$ as a simple module. Therefore the results of the previous section are still true under this assumption.\\
    One more thing that we need to do before going to the main theorem, is to define a monomial to such arcs. This is essential, since the definition given earlier cannot be used now. The reason for this, is that given a perfect matching of the graph associated to the such a bandstring, one cannot produce a new one by just flipping two edges in one of the  squares of the band graph. Therefore a new but intuitive definition is needed.
\end{remark}

\begin{definition}
    Assume that $w$ is a bandstring produced by a resolution of an incompatibility and is equal to the hook $h=\rightarrow a_1 \rightarrow \dots \rightarrow a_n\rightarrow$. Then \textit{the monomial associated to the bandstring $w$} is defined to be:
    \[
    x(w)= 1 + y_1 \dots y_n.
    \]
\end{definition}

\begin{example}
    Assume that $w= h$ where h is the hook around the puncture $p$ on the surface of Figure \ref{fig:Monomial of band loop}. Then we have that:
    \[
    x(w)= 1 + y_1y_2y_3y_4.
    \]
\end{example}

\begin{figure}
    \centering
    \includegraphics[scale=0.6]{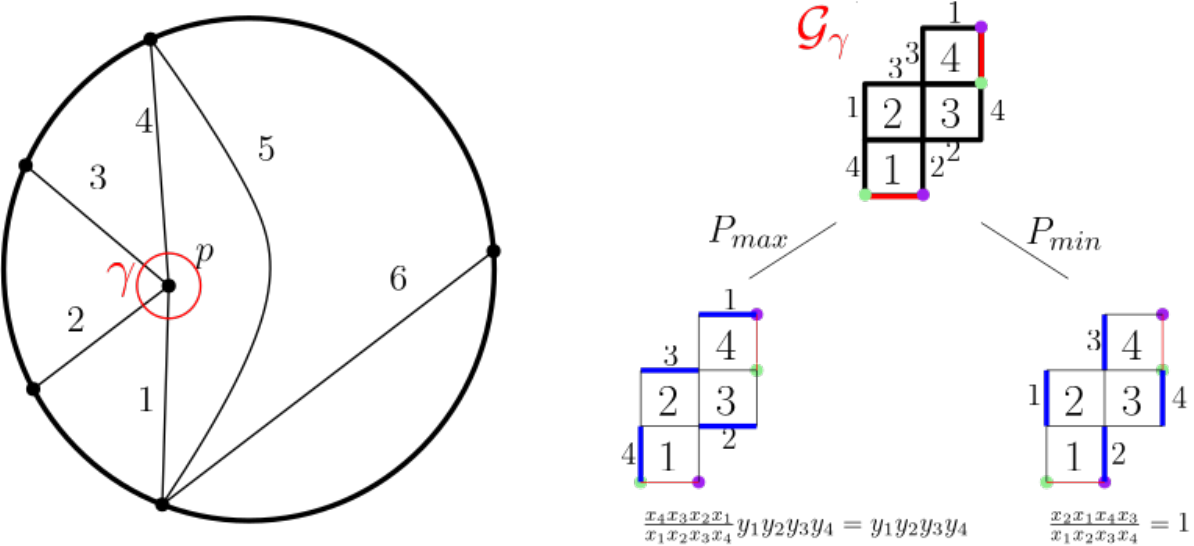}
    \caption{Band around puncture}
    \label{fig:Monomial of band loop}
\end{figure}

\begin{proposition}\label{xvariablesagree double incompatibility puncures}
     Let $w_1, w_2, w_3, w_4, w_5$ and $w_6$ as they were defined in \ref{resolution of double incompatible taggings two punctures}. Then:
     \[
     x(w_3)= x(w_1\oplus w_2),
     \]
     \[
     x(w_4)= x((w_1\setminus mh_2)\oplus w_2)),
     \]
     \[
     x(w_5)= x((w_1\setminus a_1h_1')\oplus w_2),
     \]
     \[
     x(w_6)= x((w_1\setminus a_1h_1')\setminus mh_2)\oplus w_2).
     \]
     \begin{proof}
     The proof here goes by induction as in the case of propositions \ref{removertopwelldefined} and \ref{remove tops from tagged and loop}.
         %{\color{blue} I can add the construction here separately after the proof including an example}.
     \end{proof}
\end{proposition}

\begin{theorem}\label{Skein relations for double incompatibilities at two punctures}
    Let $(S,M,P,T)$ be a triangulated punctured surface, $\gamma_1$ a double notched arc at the puncture $p$ and $q$ and $\gamma_2$ a plain arc which has its endpoints also at the punctures $p$ and $q$. Suppose that $\gamma_3,\gamma_4, \gamma_5$ and $\gamma_6$ are the bands obtained by resolving the incompatibilities at the punctures.
        Then:
        \[
        x_{\gamma_1} x_{\gamma_2} =Y(m)Y(a_2h_2'') x_{\gamma_3} + Y(a_2h_2'') x_{\gamma_4}  +  x_{\gamma_5} + Y(m) x_{\gamma_6}.
        \]

    \begin{proof}
        The proof of this theorem is a direct consequence of Theorem \ref{bijection of modules incompatibilities two puncture}, Proposition \ref{xvariablesagree double incompatibility puncures} and Remark \ref{removetoppreservesx}.
    \end{proof}
\end{theorem}

\begin{figure}
    \centering
    \includegraphics[scale=0.8]{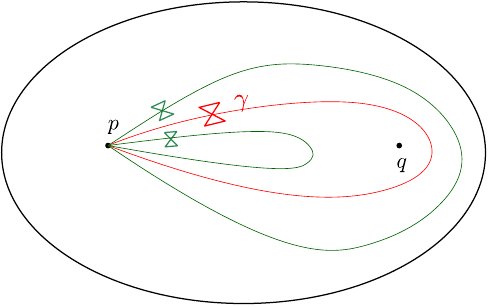}
    \caption{Self incompatible arc.}
    \label{fig:Self incompatible arc}
\end{figure}

\begin{remark}
    By taking a look at the $y$-variables appearing on the terms of Theorem \ref{Skein relations for double incompatibilities at two punctures} one can notice the intricacies of this case and how one could possibly follow a different root on the proof of this statement. For this, we should pay attention on the possible grouping of the terms that was mentioned at the start of this section. One can see that each puncture "generates" its own terms on the $y$-variables, which is a strong indication that this resolution could be viewed as a two step procedure. \\
    It can also indicate that an arc which has a self incompatibility at a puncture could be viewed under some assumptions as an element of the Cluster algebra, since this is a term appearing after the first step of such a procedure (Figure \ref{fig:Self incompatible arc}).  
\end{remark}

\section{ Skein relations for self crossing tagged arcs}

In this section we will work on the final cases of possible incompatibilities of tagged arcs on punctured surfaces. These two cases are namely the following:

\begin{itemize}
    \item a single notched generalized arc with a self-crossing,
    \item a doubly notched generalized arc with a self crossing.
\end{itemize}

These two cases end up being extremely similar to the standard cases of self crossing untagged generalized arcs, that were studied in \cite{canakci2}. The reason for these similarities is the same as that observed in the regular crossing of tagged arcs. Since the crossing takes place away from the puncture, the hook has minimal impact on the resolution.\\

As we will see later on this section, one can apply the resolution of arcs as it was described in \cite{canakci2} and just add the hook where it is necessary. However, we still need to do some work, since a priori the skein relations are not guaranteed to be the same as in the regular set up of generalized plain self crossing arcs. Nonetheless going through all the cases that were described in their paper adding possible hooks to the arcs would not be mathematically interesting, since one can just recycle all of the arguments combining the ideas of their paper and the suitable generalization on how someone deals with the hooks. This is why we will only focus on some of the more interesting cases that can be encountered on this set up.\\
The two cases that we will focus on will be about self crossing doubly notched arcs that have both of their endpoints at the same puncture. On these two cases one can encounter either a \textit{self-grafting incompatibility} (Figure \ref{fig:Self crossing grafting} ) or a \textit{regular self crossing incompatibility} (Figure \ref{fig:Self crossing regular}). The purpose of this distinction is the same as in the previous intersections. Depending on the triangulation there may be a common overlap locally at the intersection or not.\\
The final result of this chapter which will combine both of the previous cases (Theorem~\ref{skein for self crossing of same and opposite incompatibility}, Theorem~\ref{skein for self crossing of grafting incompatibility}) that will be explored in detail, as well as the cases that we will not be fully explored due to their similarities will be summed up in the following theorem:

\begin{theorem}\label{skein relation for generalized tagged self crossing arcs.}
        Let $(S,M,P,T)$ be a triangulated punctured surface and $\gamma_1$ a self intersecting arc. Let $c_1$ and $c_2$ be the two multicurves obtained by smoothing the crossing.
        Then we have that:
        \[
        x_{\gamma_1} = x_{c_1}  + Y^{+} x_{c_2},
        \]
        where $Y^{+}$ is a monomial on $y$-variables.
\end{theorem}

\begin{figure}
    \centering
    \includegraphics[scale=0.8]{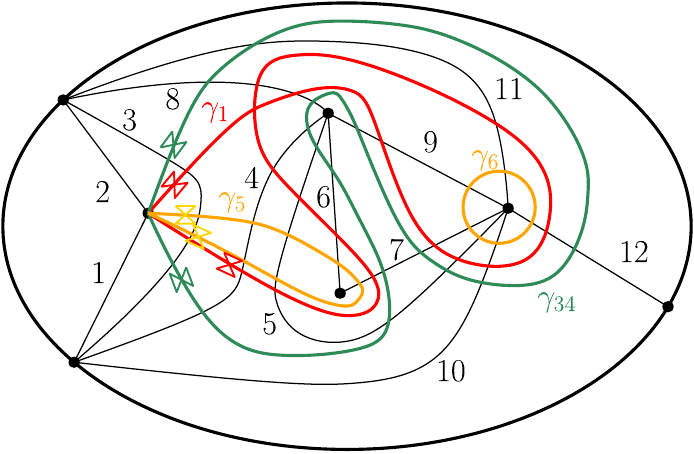}
    \caption{Regular self crossing of a doubly notched arc in the opposite direction.}
    \label{fig:Self crossing regular}
\end{figure}

\subsection{Resolution of regular self incompatibility of a doubly notched arc}

Suppose that $(S, M, P, T)$ is a triangulated punctured surface and let $\gamma$ be an arc which has both of its endpoints at the same puncture and has a tagging at both of them. Assume that there is a common overlap $m$ locally, where the self-intersection occurs. Then we say that there is a \textit{regular self-incompatibility}.\\
We will interpret the conditions under which this incompatibility occurs by analyzing the corresponding loop string associated with the arc $\gamma$. The main distinction is about the overlap and if it occurs on the same or on the opposite direction.

\begin{definition}\label{self intersection opp and same direction}
    Suppose that $w$ is an abstract loopstring with the same hook at both of its ends. \\
    Assume that $w= hu_1ambu_2cm^{-1}du_3h$ where $u_i$ and $m$ are substrings, $a$ and $d$ are inverse arrows, $b$ and $c$ are inverse arrows and $h$ ais a hook. We then say that $w$ has a \textit{regular self-crossing in the opposite direction} in $m$.\\
    The resolution of the self crossing of $w$ with respect to $m$ are the loopstrings $w_5, w_{34}$ and the band $w_6^{o}$ which are defined as follows:
    \begin{itemize}
        \item $w_{34}= hu_1amc^{-1}u_2^{-1}b^{-1}m^{-1}du_3h$,
        \item $w_5= hu_1eu_3h,$ where $e$ is a direct arrow,
        \item $w_6^{o}= u_2^{o}$.
    \end{itemize}
    Assume now that $w= hu_1ambu_2cmdu_3h$ where $u_i, m, a, b, c, d$ as they were defined earlier. We then say that $w$ has a \textit{regular self-crossing in the same direction} in $m$.
    The resolution of the self crossing of $w$ with respect to $m$ is the loopstring $w_{34}$ and the bands $w_5^{o}, w_6^{o}$ which are defined as follows:
    \begin{itemize}
        \item $w_{3}= hu_1amdu_3h$,
        \item $w_4^{o}= cmbu_2c$,
        \item $w_{56}^{o}= (u_3-u_1)'\rightarrow (u_2^{-1})'$, where $(u_3-u_1)'$ is the string $u_3$ subtracting the string $u_1$ and the intersection of $u_2$ and $u_3$, while $(u_2^{-1})'$ is the inverse of the string $u_2$ subtracting the intersection of $u_2$ and $u_3$.
    \end{itemize}
\end{definition}

We will now present one example where there is a self intersection in the opposite direction so that Definition~\ref{self intersection opp and same direction} becomes clearer.

\begin{example}
    Suppose that $(S,M,P,T)$ is the triangulated surface appearing in Figure~\ref{fig:Self crossing regular}. Let $\gamma_1$ be the red arc appearing in the same figure. Then the string $w_1$ corresponding to $\gamma_1$ is the following:
    \begin{gather*}
    w_1 = 1\rightarrow 2 \hookleftarrow 3 \leftarrow 8 \rightarrow 9 \leftarrow 7\leftarrow 5\leftarrow 10\leftarrow 12\leftarrow 11\rightarrow 8 \leftarrow\\
    \leftarrow 4\leftarrow5\leftarrow6\rightarrow 7\leftarrow5 \rightarrow 4\leftarrow 3\hookleftarrow 1 \rightarrow 2.
    \end{gather*}
    Following Definition~\ref{self intersection opp and same direction} we see that there is a self intersection in the opposite direction with overlap $m=8$. We can additionally break the string $w_1$ in the following substrings:
    \begin{gather*}
        h= 1\leftarrow 2,\\
        u_1 = 3, \\
        u_2 = 9 \leftarrow 7\leftarrow 5\leftarrow 10\leftarrow 12\leftarrow 11, \\
        u_3 = 4\leftarrow5\leftarrow6\rightarrow 7\leftarrow5 \rightarrow 4\leftarrow 3.
    \end{gather*}
    We therefore get the following resolution of $w_1$ with respect to $m$.
    \begin{gather*}
        w_{34} = 1\rightarrow 2 \hookleftarrow 3 \leftarrow 8 \leftarrow 11 \rightarrow 12\rightarrow 10\rightarrow 5\rightarrow 7\rightarrow 9 \\
        \leftarrow 8 \leftarrow 4 \leftarrow 5 \leftarrow 6 \rightarrow 7 \leftarrow 5 \rightarrow 4 \leftarrow 3\hookleftarrow 1 \rightarrow 2, \\
        w_5 = 1\rightarrow 2 \hookleftarrow 3 \rightarrow 4\leftarrow5\leftarrow6\rightarrow 7\leftarrow5 \rightarrow 4\leftarrow 3 \hookleftarrow 1 \rightarrow 2, \\
        w_6^{o} = \leftarrow 9 \leftarrow 7\leftarrow 5\leftarrow 10\leftarrow 12\leftarrow 11 \leftarrow.
    \end{gather*}
    It is easy to see that the loopstrings/bands $w_{34}, w_5. w_6^{o}$ in the resolution correspond to the respective arcs/loops $\gamma_{34}, \gamma_5, \gamma_6^{o}$ appearing in Figure~\ref{fig:Self crossing regular} as it should be expected.
\end{example}

The following two propositions, each deal with a separate case of opposite and same direction self crossing incompatibility respectively.

\begin{proposition}\label{unique remove tops self incompatibility opp direction}
    Let $w_1, w_{34}, w_5$ and $w_6^{o}$ as they were defined in Definition~\ref{self intersection opp and same direction}, where $w_1$ has a self intersection $m$ in the opposite direction. Suppose that $m_1$ is a proper submodule of $w_{34}$. If $m_1$ is reached from $w_{34}$  by the sequence $(b_i), 1\leq i\leq n$, then we can apply the same sequence $(a_i)$ to the module $w_1$ to obtain a unique submodule $m_1'$ of it.\\
    Similarly, assuming that $m_2$ is a proper submodule of $w_5\oplus w_6^{o}$ and is reached from $w_5\oplus w_6^{o}$ by the sequence $(b_i), 1\leq i\leq n$, we can apply the same sequence $(a_i)$ to the module $w_1-m =  hu_1\oplus u_2cm^{-1}du_3h$ to obtain a unique submodule $m_1'$ of it.
    
    \begin{proof}
        The uniqueness of the modules follows from the construction of the modules in 
        the next Lemma~\ref{removing tops for self incompatibility opp direction}.
    \end{proof}
\end{proposition}

\begin{proposition}\label{unique remove tops self incompatibility same direction}
    Let $w_1, w_{3}, w_4^{o}$ and $w_{56}^{o}$ as they were defined in Definition~\ref{self intersection opp and same direction}, where $w_1$ has a self intersection $m$ in the same direction. Suppose that $m_1$ is a proper submodule of $w_{3}\oplus w_4^{o}$. If $m_1$ is reached from $w_{3}\oplus w_4^{o}$  by the sequence $(b_i), 1\leq i\leq n$, then we can apply the same sequence $(a_i)$ to the module $w_1$ to obtain a unique submodule $m_1'$ of it.\\
    Similarly, assuming that $m_2$ is a proper submodule of $w_{56}^{o}$ and is reached from $w_{56}^{o}$ by the sequence $(b_i), 1\leq i\leq n$, we can apply the same sequence $(a_i)$ to the module $u_2cmd(u_3-u_1)'$, to obtain a unique submodule $m_1'$ of it.
    
    \begin{proof}
        The uniqueness of the modules follows from the construction of the modules in 
        the next Lemma~\ref{removing tops for self incompatibility opp direction}.
    \end{proof}
\end{proposition}

\begin{lemma}\label{removing tops for self incompatibility opp direction}
    Under the assumptions of Proposition~\ref{unique remove tops self incompatibility opp direction} and Proposition~\ref{unique remove tops self incompatibility same direction}, for each of the modules $m_i$ and sequences $(b_i), 1\leq i\leq n$, we can apply the same sequence in the associated submodule of $w_1$ to obtain a new submodule $m_1'$
    \begin{proof}
    Let us assume first that the self intersection is in the opposite direction.\\
    If $m_1$ is a submodule of $w_{34}$ it is easy to see that the statement is true since it follows the same arguments from previous chapters and the local position of the substrings $u_1h, u_2, u_3h$ is the same.\\
    If $m_2$ is a submodule of $w_5\oplus w_6^{o}$, having removed the substring $m$ from the module $w_1$ we can again easily see that each top removed from $w_5\oplus w_6^{o}$ has an obvious counterpart from $hu_1\oplus u_2cm^{-1}du_3h$.\\
    Suppose now that the self intersection is in the same direction.
    The only interesting case is when we consider $m_1$ to be a submodule of $w_{56}^{o}$ when $w$ has a self intersection in the same direction $m$. However, even this case is straightforward since $u_2$ and $u_3'$ are local tops in $u_2cmd(u_3-u_1)'$, due to the $c$ being a direct arrow and $d$ an inverse arrow.
    \end{proof}
\end{lemma}

\begin{theorem}\label{bijection of modules self incompatibility same and opp direction}
    Let $w_1, w_{34}, w_5$ and $w_6^{o}$ as they were defined in Definition~\ref{self intersection opp and same direction}, where $w_1$ has a self intersection $m$ in the opposite direction. Then there exists a bijection:
    \[
    \Psi \colon SM(w_{34}) \bigcup SM(w_5\oplus w_6^{o})\to SM(w_1),
    \]
    which is defined as follows:
    \begin{gather*}
    \Psi(w_{34}) = w_1, \\
    \Psi(w_5\oplus w_6^{o}) = hu_1\oplus u_2cm^{-1}du_3h.
    \end{gather*}
    If $m$ is a submodule of $w_{34}$  which is reached by the sequence $\{a_i\} 1\leq i\leq n$, and $m'$ the unique submodule of $w_1$ which is reached by the same sequence, then we define:
    \[
    \Psi(m) = m'.
    \]
    If $m$ is a submodule of $w_5\oplus w_6^{o}$, then $\Psi(m)$ is defined in a similar way.\\

    Let also  $w_1, w_{3}, w_4^{o}$ and $w_{56}^{o}$ as they were defined in Definition~\ref{self intersection opp and same direction}, where $w_1$ has a self intersection $m$ in the same direction. Then there exists a bijection:
    \[
    \Phi \colon SM(w_{3}\oplus w_4^{o}) \bigcup SM(w_{56}^{o})\to SM(w_1),
    \]
    which is defined as follows:
    \begin{gather*}
    \Psi(w_{3}\oplus w_4^{o}) = w_1, \\
    \Psi(w_{56}^{o}) = u_2m(u_3-u_1)'.
    \end{gather*}

    \begin{proof}
    The fact that $\Psi$  and $
    Phi$ are bijections follows from Proposition~\ref{injectivity of Psi self incompatibility same and opp direction} and Remark~\ref{termsincoeffreecase}.

\end{proof}
\end{theorem}

\begin{proposition}\label{injectivity of Psi self incompatibility same and opp direction}
    The maps $\Psi$ and $Phi$ as they were defined in Theorem~\ref{bijection of modules self incompatibility same and opp direction} are injections.
    \begin{proof}
        The proof of this proposition is similar to the one given in Proposition~\ref{injofPsi} and is omitted.
    \end{proof}
\end{proposition}

We will proceed by giving the theorem describing the skein relations for both cases of same and opposite self crossing without explicitly proving it. The only argument missing is the proof that the $x$-variables of the corresponding whole modules agree in each case, but this has been done multiple times already and the proof technique is exactly the same.

\begin{theorem}\label{skein for self crossing of same and opposite incompatibility}
    Let $(S,M,P,T)$ be a triangulated punctured surface, $\gamma_1$ a doubly notched arc at a puncture $p$ which has a self crossing. Then:
    \begin{itemize}
        \item[(i)] 
        If the overlap $m$ in $\gamma_1$ occurs in the opposite direction, then:
        \[
        x_{\gamma_1} = x_{\gamma_{34}}  + Y^{+} x_{\gamma_5} x_{\gamma_6^{o}},
        \]
        where $Y^{+} =  Y(m)$.
        \item[(ii)]
        If the overlap $m$ in $\gamma_1$ occurs in the same direction, then:
        \[
        x_{\gamma_1} = x_{\gamma_{3}} x_{\gamma_4^{o}} + Y^{+} x_{\gamma_{56}^{o}},
        \]
        where $Y^{+} =  Y(m)Y(u_2\cap u_3)$.
    \end{itemize}
\end{theorem}

\subsection{Resolution of grafting self incompatibility of a doubly notched arc}

We will now deal with the case where a tagged arc $\gamma$ in a triangulated punctured surface $(S, M, P, T)$ has a self intersection but the overlap is empty. In other words we will resolve a grafting self incompatibility of a doubly notched arc. In this section, as it was the case with the previous one, we will limit ourselves to the case where the given arc $\gamma$ has both of its endpoints at the same puncture. \\
Contrary to the classical grafting, this case is unique since even though the overlap is empty we can still recognize the grafting case given one abstract loopstring.

\begin{definition}\label{self intersection grafting}
    Suppose that $w$ is an abstract loopstring with the same hook $h=\{x_1,\dots,x_n\}, 2\leq n$, at both of its endpoints. \\
    Assume that $w_1= hu_1x_kax_{k+1}u_2h$ where $1\leq k \leq n-1$ and $a=\leftarrow$. Then, we say that $w$ has a \emph{grafting incompatibility at $x_k,x_{k+1}$}.\\
    The \emph{resolution of the grafting incompatibility at $x_kx_{k+1}$} is defined to be the loopstrings $\gamma_{34}, \gamma_5$ and the bandstring $\gamma_6^{o}$ which are defined as follows:
    \begin{itemize}
        \item $w_{34}= h u_1^{-1} x_{k+1} u_2h$,
        \item $w_5= h(u_2)'h$, where $(u_2)' = u_2 - (u_2\cap h)$, 
        \item $w_6^{o}= (x_k u_1^{-1})^{o}$.
    \end{itemize}    
\end{definition}

\begin{remark}
In Definition~\ref{self intersection grafting} the string $(u_2)'$ is basically defined to be the string $u_2$ if we subtract the common elements of $u_2$ and $h$. \\
Additionally as a reminder of the notation used in the above definition, if $u$ is a string then we denote by $u^{o}$ the bandstring which can be created by adding the appropriate arrow at the start and end of $u$. Of course this cannot be done on every string, but in our case this is possible.
\end{remark}

\begin{example}
    Suppose that $(S,M,P,T)$ is the triangulated surface appearing in Figure~\ref{fig:Self crossing grafting}. Let $\gamma_1$ be the red arc appearing in the same figure. Then the string $w_1$ corresponding to $\gamma_1$ is the following:
    \begin{gather*}
    w_1 = 3\rightarrow 2\rightarrow 1\rightarrow 7\rightarrow 6\rightarrow 5\rightarrow 4 \hookleftarrow 11 \rightarrow 12\rightarrow 3 \leftarrow 4 \\ 
    \leftarrow 5\rightarrow 10 \leftarrow 9 \leftarrow 8 \hookrightarrow 1 \leftarrow 2 \leftarrow 3 \leftarrow 4 \leftarrow 5 \leftarrow 6\leftarrow 7.
    \end{gather*}
    Following Definition~\ref{self intersection grafting} we see that there is a grafting incompatibility at $3,4$. Breaking down the string $w_1$ we also see the following:
    \begin{gather*}
        h= 1 \leftarrow 2 \leftarrow 3 \leftarrow 4 \leftarrow 5 \leftarrow 6\leftarrow 7, \\
        u_1 = 12 \leftarrow 11 ,\\
        u_2 =  5\rightarrow 10 \leftarrow 9 \leftarrow 8. 
    \end{gather*}
    Our goal is to resolve this grafting incompatibility using Definition~\ref{self intersection grafting}. For that, let us first compute $u_2'$. According to our definition, $u_2'$ should not contain the common elements of $u_2$ and $h$, so in our case this is just $5$. Therefore:
    \[
    u_2' = 10 \leftarrow 9\leftarrow 8.
    \]
    We are now ready to compute the resolution:
    \begin{gather*}
        w_{34} = 2\rightarrow 1\rightarrow 7\rightarrow 6\rightarrow 5\rightarrow 4 \rightarrow 3 \hookleftarrow 12\leftarrow 11 \rightarrow 4 \leftarrow 5\\
        \rightarrow 10 \leftarrow 9 \leftarrow 8 \hookrightarrow 1 \leftarrow 2 \leftarrow 3 \leftarrow 4 \leftarrow 5 \leftarrow 6\leftarrow 7,\\
        w_5 = 5\rightarrow 4 \rightarrow 3\rightarrow 2 \rightarrow 1\rightarrow 7 \rightarrow 6 \hookleftarrow 10 \leftarrow 9\\
        \leftarrow 8 \hookrightarrow 1\leftarrow 2\leftarrow 3\leftarrow 4\leftarrow 5 \leftarrow 6 \leftarrow 7,\\
        w_6^{o} =  \leftarrow 3 \leftarrow 12 \leftarrow 11 \leftarrow.
    \end{gather*}
    It is now easy to see that the loopstrings/bandstrings $w_{34}, w_5. w_6^{o}$ in the resolution correspond to the respective arcs/loops $\gamma_{34}, \gamma_5, \gamma_6^{o}$ appearing in Figure~\ref{fig:Self crossing grafting}.
\end{example}

\begin{figure}
    \centering
    \includegraphics[scale=0.8]{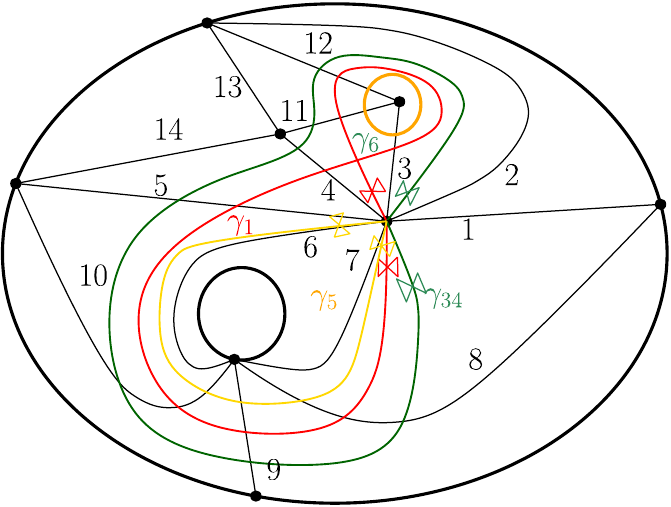}
    \caption{Grafting self crossing of a doubly notched arc.}
    \label{fig:Self crossing grafting}
\end{figure}

In this last section we will omit the propositions and lemmas leading to the final theorem which showcases how the bijection between the submodules of $w_1$ and on the other hand the submodules of $w_{34}$ and $w_5\oplus w_6^{o}$ is built, since all the arguments and ideas are following preexisting results.

\begin{theorem}\label{bijection of modules self crossing grafting}
    Let $w_1, w_{34}, w_5$ and $w_6^{o}$ as they were defined in Definition~\ref{self intersection grafting}, where $w_1$ has a grafting incompatibility at $x_k,x_{k+1}$. Then there exists a bijection:
    \[
    \Psi \colon SM(w_{34}) \bigcup SM(w_5\oplus w_6^{o})\to SM(w_1),
    \]
    which is defined as follows:
    \begin{gather*}
    \Psi(w_{34}) = (h-x_k)bu_1x_kax_{k+1}u_2h, \\
    \Psi(w_5\oplus w_6^{o}) = w_1,
    \end{gather*}
    where $(h-x_k)$ is the substring $x_{k-1}\rightarrow \dots \rightarrow x_{k+1}$ and $b=\leftarrow$.
    If $m$ is a submodule of $w_{34}$  which is reached by the sequence $\{a_i\} 1\leq i\leq n$, and $m'$ the unique submodule of $w_1$ which is reached by the same sequence, then we define:
    \[
    \Psi(m) = m'.
    \]
    If $m$ is a submodule of $w_5\oplus w_6^{o}$, then $\Psi(m)$ is defined in a similar way.\\
\end{theorem}

As in the previous section we  proceed on by finally stating the last theorem dealing with the skein relations without explicitly working the details of the proof.

\begin{theorem}\label{skein for self crossing of grafting incompatibility}
    Let $(S,M,P,T)$ be a triangulated punctured surface, $\gamma_1$ a doubly notched arc at a puncture $p$ which has a self crossing at the first triangle of $T$ that $\gamma_1$ passes through. Then there is a grafting incompatibility and we the smoothing of the intersection is the following:
    \[
    x_{\gamma_1} = x_{\gamma_{34}}  + Y^{+} x_{\gamma_5} x_{\gamma_6^{o}},
    \]
    where $Y^{+} =  Y(x_kx_{k+1})Y(u_2\cap h)$.
\end{theorem}

\printbibliography

\backmatter

\end{document}